\documentclass[a4paper,11 pt]{article}

\usepackage[T1]{fontenc}
\usepackage[english]{babel}
\usepackage{times}

\usepackage{amsmath}
\usepackage{amsfonts}
\usepackage{amssymb}
\usepackage{amsmath}
\usepackage{amsthm}
\usepackage{graphicx}
\usepackage{array}
\usepackage{stmaryrd}

\usepackage{color}

\usepackage{multirow}
\usepackage{url}
\usepackage{hyperref}
\usepackage[all]{hypcap}
\usepackage{fancyhdr}

\usepackage{frcursive}
\usepackage{enumerate}
\usepackage{dsfont}
\usepackage{mathrsfs}
\usepackage{verbatim}
\usepackage[font=small,labelfont=bf]{caption}
\usepackage{upgreek}
\usepackage{setspace}

\evensidemargin\oddsidemargin

\makeatletter
\g@addto@macro\th@plain{\thm@headpunct{}}
\makeatother
\theoremstyle{plain}
\newtheorem*{theorem*}{Theorem}
\newtheorem{theorem}{Theorem}
\newtheorem{lemma}{Lemma}[section]
\newtheorem{proposition}[lemma]{Proposition}

\newtheorem{definition}[lemma]{Definition}
\newtheorem{remark}[lemma]{Remark}
\newtheorem{example}[lemma]{Example}

\newtheorem{assumption}{Assumption}

\renewenvironment{proof}[1][Proof]{\begin{trivlist}
\item[\hskip \labelsep {\bfseries #1}]}{\end{trivlist}}

\newcommand{\Extinction}{\emph{Extinction}}
\newcommand{\Explosion}{\emph{Explosion}}
\newcommand{\InfNo}{\emph{Infinite lifetime - no extinction of ancestors}}
\newcommand{\InfPos}{\emph{Infinite lifetime - possible extinction of ancestors}}

\def\cqfd{ \hfill $\blacksquare$}

\DeclareMathOperator{\Coag}{Coag}

\newcommand{\scK}{\textsc{K}}

\newcommand{\rmm}{\mathrm{m}}
\newcommand{\rme}{\mathrm{e}}
\newcommand{\rmr}{\mathrm{r}}

\newcommand{\tun}{\mathtt{1}}

\newcommand{\tO}{\mathtt{O}}

\newcommand{\rF}{\mathrm{F}}

\newcommand{\rH}{\mathrm{H}}

\newcommand{\rL}{\mathrm{L}}

\newcommand{\rS}{\mathrm{S}}
\newcommand{\rT}{\mathrm{T}}

\newcommand{\rX}{\mathrm{X}}
\newcommand{\rY}{\mathrm{Y}}
\newcommand{\rZ}{\mathrm{Z}}

\newcommand{\bbN}{\mathbb{N}}

\newcommand{\cB}{\mathcal{B}}

\hypersetup{
colorlinks=true, 
breaklinks=true, 
urlcolor= blue, 
linkcolor= black, 
bookmarksopen=true, 
pdftitle={Genealogy of flows of continuous-state branching processes via flows of partitions}, 
pdfauthor={Cyril Labb\'e}, 
pdfsubject={PhD Thesis} 
}      

\title{Genealogy of flows of continuous-state branching processes via flows of partitions and the Eve property}
\author{Cyril \textsc{Labb\'e}\thanks{Laboratoire de Probabilit\'es et Mod\`eles Al\'eatoires, Universit\'e Pierre et Marie Curie (Paris 6)}}

\begin{document}
\maketitle
\vspace{-0.8cm}
\begin{abstract}
We encode the genealogy of a continuous-state branching process associated with a branching mechanism $\Psi$ - or $\Psi$-CSBP in short - using a stochastic flow of partitions. This encoding holds for all branching mechanisms and appears as a very tractable object to deal with asymptotic behaviours and convergences. In particular we study the so-called Eve property - the existence of an ancestor from which the entire population descends asymptotically - and give a necessary and sufficient condition on the $\Psi$-CSBP for this property to hold. Finally, we show that the flow of partitions unifies the lookdown representation and the flow of subordinators when the Eve property holds.
\end{abstract}
\renewcommand{\abstractname}{R\'esum\'e}
\begin{abstract}
Nous construisons la g\'en\'ealogie d'un processus de branchement \`a espace d'\'etats et temps continus associ\'e \`a un m\'ecanisme de branchement $\Psi$ - ou $\Psi$-CSBP - \`a l'aide d'un flot stochastique de partitions. Cette construction est valable quel que soit le m\'ecanisme de branchement et permet de d\'efinir un objet remarquablement efficace pour \'etudier les comportements asymptotiques et les convergences. En particulier, nous \'etudions la propri\'et\'e d'Eve - l'existence d'un anc\^etre dont descend asymptotiquement toute la population - et donnons une condition n\'ecessaire et suffisante sur le $\Psi$-CSBP pour que cette propri\'et\'e soit v\'erifi\'ee. Finalement, nous montrons que le flot de partitions unifie la repr\'esentation lookdown et le flot de subordinateurs lorsque la propri\'et\'e d'Eve est v\'erifi\'ee.
\end{abstract}
\let\thefootnote\relax\footnotetext[1]{\textit{\href{http://www.ams.org/msc/}{MSC 2010 subject classifications.}} Primary 60J80; secondary 60G09, 60J25}
\let\thefootnote\relax

\section{Introduction}
A continuous-state branching process (CSBP for short) is a Markov process $(\rZ_t,t \geq 0)$ that takes its values in $[0,+\infty]$ and fulfils the branching property: for all $z,z' \in [0,+\infty]$, $(\rZ_t+\rZ'_t,t\geq 0)$ is a CSBP, where $(\rZ_t,t\geq 0)$ and $(\rZ'_t,t\geq 0)$ are two independent copies started from $z$ and $z'$ respectively. Such a process describes the evolution of an initial population size $\rZ_0$, and the branching property implies that two disjoint subpopulations have independent evolutions. To alleviate notation, we will implicitly consider an initial population size $\rZ_0 = 1$. A CSBP has a Feller semigroup entirely characterized by a convex function $\Psi$ called its branching mechanism, so we will write $\Psi$-CSBP to designate the corresponding distribution. The Feller property entails the existence of a c\`adl\`ag modification, still denoted $(\rZ_t,t\geq 0)$ and thus allows to define the lifetime of $\rZ$ as the stopping time
\begin{equation*}
\rT := \inf\{t\geq 0: \rZ_t \notin(0,\infty)\}
\end{equation*}
with the convention $\inf\emptyset = \infty$. The denomination \textit{lifetime} is due to the simple fact that both $0$ and $\infty$ are absorbing states.\\

The process $\rZ$ can be seen as the total-mass of a measure-valued process $(\rmm_t,t\in[0,\rT))$ on $[0,1]$ (or any compact interval), started from the Lebesgue measure on $[0,1]$ and such that for all $x\in[0,1]$, $(\rmm_t([0,x]),t\geq 0)$ and $(\rmm_t((x,1]),t\geq 0)$ are two independent $\Psi$-CSBP corresponding to the sizes of the subpopulations started from $[0,x]$ and $(x,1]$ respectively. The process $\rmm$ is called a measure-valued branching process or $\Psi$-MVBP for short. Note that when $\rZ_\rT = \infty$, the measure is no longer finite and therefore we set $\rmm_\rT = \Delta$, see Subsection \ref{SubsectionMVBP} for further details.
\begin{definition}\label{DefEveProperty}
We say that the branching mechanism $\Psi$ satisfies the Eve property if and only if there exists a random variable $\rme$ in $[0,1]$ such that
\begin{equation}
\frac{\rmm_{t}(dx)}{\rmm_{t}([0,1])} \underset{t \uparrow \rT}{\longrightarrow} \delta_{\rme}(dx)\mbox{ a.s.}
\end{equation}
in the sense of weak convergence of probability measures. The r.v. $\rme$ is called the primitive Eve of the population.
\end{definition}
This property means that a fraction asymptotically equal to $1$ of the population descends from a single individual located at $\rme$ as $t$ gets close to the lifetime $\rT$. This property seems to have never been studied before, except by Tribe~\cite{Tribe92} in the case of the Feller diffusion with a spatial motion. From the branching property, we will show that $\rme$ is necessarily uniform$[0,1]$, when the Eve property is verified. The goal of the present paper is to study this Eve property in connection with the genealogy of the $\Psi$-CSBP. Note that the complete classification of the asymptotic behaviour of $\frac{\rmm_{t}(dx)}{\rmm_{t}([0,1])}$ will be provided in a forthcoming work~\cite{DuquesneLabbe13}.\\

A CSBP describes the evolution of the population size, but does not provide clear information on the genealogy. In recent years, several approaches have been proposed to study the genealogical structure: one can cite the historical superprocess of Dawson and Perkins~\cite{DawsonPerkins91}, the continuum random tree introduced by Aldous in~\cite{AldousCRT1}, the L\'evy trees of Le Gall, Le Jan and Duquesne~\cite{DLG02,LeGallLeJan98}, we also refer to~\cite{BertoinLeGall-1,GrevenPfaffelhuberWinter12,GrevenPopovicWinter09} for the genealogy of related population models. This paper presents a new object, called a stochastic flow of partitions associated with a branching mechanism, that unifies two well-known approaches: the flow of subordinators of Bertoin and Le Gall~\cite{BertoinLeGall-0} and the lookdown representation of Donnelly and Kurtz~\cite{DK99}. Let us mention that this object focuses on the genealogical structure, and does not pay attention to the genetic types carried by the individuals: hence it would not be appropriate to deal with more elaborate models incorporating mutations or spatial motions. We first introduce this object and its relationships with these two representations, before presenting the connection with the Eve property.\\

As mentioned above, the population size does not define in itself the genealogy. Therefore we start from a c\`adl\`ag $\Psi$-CSBP $(\rZ_t,t\in [0,\rT))$ and enlarge the probability space in order to add more information to this process. This is achieved by defining a random point process ${\cal P}$ with values in $[0,\rT)\times\mathscr{P}_{\infty}$, where $\mathscr{P}_{\infty}$ stands for the space of partitions of the integers $\mathbb{N}$. To each jump $(t,\Delta \rZ_t)$ of the CSBP is associated a point $(t,\varrho_t)$ in $\cal P$ such that the random partition $\varrho_t$ is distributed according to the paint-box scheme with mass-partition $(\frac{\Delta \rZ_t}{\rZ_t},0,0,\ldots)$, see Subsection \ref{SubsectionPartitions} for a precise definition of the paint-box scheme. The genealogical interpretation is the following: $(t,\Delta \rZ_t)$ corresponds to a reproduction event where a parent, chosen uniformly among the population alive at time $t-$, gives birth to a subpopulation of size $\Delta \rZ_t$; therefore a fraction $\frac{\Delta \rZ_t}{\rZ_t}$ of the individuals at time $t$ descends from this parent. In addition, when $\rZ$ has a diffusion part, ${\cal P}$ contains points of the form $(t,\tun_{\{i,j\}})$ that model binary reproduction events, that is, events where an individual $i$ at time $t-$ is the parent of two individuals $i$ and $j$ at time $t$, the partition $\tun_{\{i,j\}}$ having a unique non-singleton block $\{i,j\}$, with $i < j$. A precise definition of the point process ${\cal P}$ will be given in Subsection \ref{SubsectionStoFlow}, but it should be seen as an object that collects all the \textit{elementary reproduction events} as time passes.\\
We then introduce a collection of random partitions $(\hat{\Pi}_{s,t},0 \leq s \leq t < \rT)$ by "composing" the partitions contained in ${\cal P}$. In order not to burden this introduction, we do not provide the precise definition of these partitions but, roughly speaking, $\hat{\Pi}_{s,t}$ is the result of the composition forward-in-time of all the elementary reproduction events provided by ${\cal P}$ on the interval $(s,t]$. Therefore the partitions collect the following information
\begin{itemize}
\item Backward-in-time : the process $s\mapsto \hat{\Pi}_{t-s,t}$ gives the genealogy of the population alive at time $t$.
\item Forward-in-time : the process $s\mapsto \hat{\Pi}_{t,t+s}$ gives the descendants of the population alive at time $t$.
\end{itemize}
$(\hat{\Pi}_{s,t},0 \leq s \leq t < \rT)$ is called a $\Psi$ flow of partitions and $\rZ$ its \textit{underlying $\Psi$-CSBP}.\\\\
\textit{Connection with the lookdown representation}\\
This object is intimately related to the lookdown representation of Donnelly and Kurtz~\cite{DK99}. A lookdown process is a particle system entirely characterized by a sequence of \textit{initial types}, that provides a sampling of the initial population, and a so-called \textit{lookdown graph}, that stands for the genealogical structure. In a previous work~\cite{Labbe11}, we showed that the flow of partitions formalizes and clarifies the notion of lookdown graph which was implicit in the lookdown construction of Donnelly and Kurtz~\cite{DK99}. To complete the picture of the lookdown construction, note that the limiting empirical measure of the particle system at time $t$, say $\Xi_t$, is a probability measure such that the process $\rZ\cdot\Xi$ is a $\Psi$-MVBP, see Section \ref{SectionFlow} for further details.\\\\
\textit{Connection with the flow of subordinators}\\
It is well-known that the process $x\mapsto \rmm_t([0,x])$ is a subordinator whose Laplace exponent $u_t(\cdot)$ is related to the branching mechanism $\Psi$ via forthcoming Equation (\ref{EquationUt}). In addition, the branching property ensures that $\rmm_{t+s}$ is obtained by composing the subordinator $\rmm_t$ with an independent subordinator distributed as $\rmm_s$. This is the key observation that allowed Bertoin and Le Gall~\cite{BertoinLeGall-0} to describe the genealogy of the $\Psi$-MVBP with a collection of subordinators. Formally, a $\Psi$ flow of subordinators $(\rS_{s,t}(a),0 \leq s \leq t, a \geq 0)$ is a collection of random processes that verify
\begin{itemize}
\item For every $0 \leq s \leq t$, $(\rS_{s,t}(a),a\geq 0)$ is a subordinator with Laplace exponent $u_{t-s}$.
\item For every integer $p\geq 2$ and $0 \leq t_1\leq \ldots \leq t_p$, the subordinators $\rS_{t_1,t_2},\ldots,\rS_{t_{p-1},t_p}$ are independent and
\begin{equation*}
\rS_{t_1,t_p}(a) = \rS_{t_{p-1},t_p}\circ\ldots\circ \rS_{t_1,t_2}(a),\ \forall a \geq 0\mbox{ a.s. (cocycle property)}
\end{equation*}
\item For all $a \geq 0$, $(\rS_{0,t}(a),t\geq 0)$ is a $\Psi$-CSBP started from $a$.
\end{itemize}
Each subordinator $[0,1] \ni x \mapsto \rS_{0,t}(x)$ can be seen as the distribution function of a random measure $\rmm_{0,t}$ on $[0,1]$ so that $(\rmm_{0,t},t\geq 0)$ forms a $\Psi$-MVBP. In particular, $\rS_t := \rS_{0,t}(1)$ is its total-mass process and one can define $\rT^\rS$ as its lifetime. Hence, all the relevant information about this initial population $[0,1]$ is contained into the flow $(\rS_{s,t}(a),0 \leq s \leq t,0 \leq a \leq \rS_s)$. Fix $0 \leq s < t < \rT^\rS$ and consider a point $a\in[0,\rS_s]$ such that $\rS_{s,t}(a)-\rS_{s,t}(a-) > 0$. Bertoin and Le Gall interpreted $a$ as an ancestor alive at time $s$ and $\rS_{s,t}(a)-\rS_{s,t}(a-)$ as its progeny alive at time $t$. We show that our collection of partitions actually formalizes this genealogical structure. To state this result we use the notation $\mathscr{P}(\rS_{s,t})$ that stands for the paint-box distribution based on the mass-partition obtained from the rescaled jumps $\{\frac{\rS_{s,t}(a)-\rS_{s,t}(a-)}{\rS_t},a\in[0,\rS_s]\}$ of the subordinator $\rS_{s,t}$, here again we refer to Subsection \ref{SubsectionPartitions} for a precise definition.
\begin{theorem}\label{TheoremFoP}
The collection of partitions $(\hat{\Pi}_{s,t},0 \leq s \leq t < \rT)$, together with its underlying CSBP $\rZ$, satisfies
\begin{itemize}
\item For all $n\geq 1$ and all $0 < t_1 < \ldots < t_n$, 
\begin{equation*}
(\rZ_{t_1},\ldots,\rZ_{t_n},\hat{\Pi}_{0,t_1},\ldots,\hat{\Pi}_{t_{n-1},t_n}|t_n < \rT)\! \stackrel{(d)}{=}\! (\rS_{t_1},\ldots,\rS_{t_n},\mathscr{P}(\rS_{0,t_1}),\ldots,\mathscr{P}(\rS_{t_{n-1},t_n})|t_n<\rT^\rS)
\end{equation*}
\item For all $0 \leq r < s < t < \rT$, a.s. $\hat{\Pi}_{r,t} = \Coag(\hat{\Pi}_{s,t},\hat{\Pi}_{r,s})$ (cocycle property).
\end{itemize}
\end{theorem}
Note that the operator $\Coag$ is a composition operator for partitions, see Section 4.2 in~\cite{BertoinRandomFragmentation} or Subsection \ref{SubsectionPartitions} of the present paper.\\\\
\textit{Main results}\\
We now study the connection between the Eve property and the genealogy. To alleviate notation, we set $\hat{\Pi}_t := \hat{\Pi}_{0,t}$ and let $\sigma$ be the diffusion coefficient appearing in the branching mechanism $\Psi$. In addition, we let $\tun_{[\infty]} := \{\{1,2,3,\ldots\}\}$ denote the partition with a unique block containing all the integers.
\begin{theorem}\label{TheoremEquivalence}
There exists an exchangeable partition $\hat{\Pi}_{\rT}$ such that $\hat{\Pi}_{t} \rightarrow \hat{\Pi}_{\rT}$ almost surely as $t \uparrow \rT$. Moreover, these three assumptions are equivalent\begin{enumerate}[i)]
\item $\Psi$ satisfies the Eve property.
\item $\hat{\Pi}_{\rT} = \tun_{[\infty]}$ a.s.
\item $\displaystyle\sum_{\{s < \rT : \Delta \rZ_s > 0\}}\Big(\frac{\Delta \rZ_s}{\rZ_s}\Big)^2+\int_{0}^{\rT}\frac{\sigma^2}{\rZ_s}ds = \infty$ a.s.
\end{enumerate}
\end{theorem}
This result allows to define $\hat{\Pi}_{t}:=\hat{\Pi}_{\rT}$ for all $t\geq \rT$.\\

If there are individuals who do not share their ancestors with any other individuals then the partition has singleton blocks: we say that the partition has dust. It is well-known that for coalescent processes with multiple collisions, a dichotomy occurs (except in a very trivial case) between those coalescent processes that have infinitely many singletons at every time $t > 0$ almost surely and those that have no singletons at every time $t > 0$ almost surely, see~\cite{Pitman99}. It is striking that a similar dichotomy holds in the branching process setting.
\begin{theorem}\label{ThDust}
The following dichotomy holds:
\begin{itemize}
\item If $\Psi$ is the Laplace exponent of a L\'evy process with finite variation paths, then almost surely for all $t \in (0,\rT)$, the partition $\hat{\Pi}_{t}$ has singleton blocks.
\item Otherwise, almost surely for all $t\in(0,\rT)$, the partition $\hat{\Pi}_t$ has no singleton blocks.
\end{itemize}
Furthermore when $\sigma = 0$, almost surely for all $t \in (0,\rT]$ the asymptotic frequency of the dust component of $\hat{\Pi}_t$ is equal to $\prod_{s \leq t}(1-\frac{\Delta \rZ_s}{\rZ_s})$ whereas when $\sigma > 0$, almost surely for all $t \in (0,\rT]$ there is no dust.
\end{theorem}

A flow of partitions also appears as an efficient tool to deal with convergences. We illustrate this fact with the following problem. Consider a sequence of branching mechanisms $(\Psi_m)_{m\geq 1}$ that converges pointwise to another branching mechanism $\Psi$. Implicitly, $\rZ^m$, $\hat{\Pi}^{m}$ will denote $\Psi_m$-CSBP and $\Psi_m$ flow of partitions, for every $m\geq 1$. It is easy to deduce from~\cite{CaballeroLambertBravo09} that $\rZ^m \rightarrow \rZ$ in a sense that will be made precise in Subsection \ref{SubsectionLimitTh}, so that a similar result for the corresponding genealogies is expected.
\begin{theorem}\label{TheoremLimit}
Suppose that
\begin{enumerate}[i)]
\item For all $u \in \mathbb{R}_+$, $\Psi_m(u) \rightarrow \Psi(u)$ as $m\rightarrow\infty$.
\item The branching mechanism $\Psi$ satisfies the Eve property.
\item $\Psi$ is not the Laplace exponent of a compound Poisson process.
\end{enumerate}
then
\begin{equation*}
(\hat{\Pi}^{m}_{t},t \geq 0) \stackrel{(d)}{\underset{m\rightarrow\infty}{\longrightarrow}}(\hat{\Pi}_{t},t \geq 0)
\end{equation*}
in $\mathbb{D}(\mathbb{R}_+,\mathscr{P}_{\infty})$.
\end{theorem}

The Eve property says that the rescaled $\Psi$-MVBP can be approximated by a Dirac mass as $t$ gets close to $\rT$. It is natural to ask if finer results can be obtained: for instance, does there exist a second Eve that carries a significant part of the remaining population ?\\
We call \textit{ancestor} a point $x \in [0,1]$ for which there exists $t \in [0,\rT)$ such that $\rmm_t(\{x\}) > 0$; in that case, $\rmm_t(\{x\})$ is called the \textit{progeny} of $x$ at time $t$. We will prove in Subsection \ref{SubsectionAncestors} that the collection of ancestors is countable. Roughly speaking, the progeny of a given ancestor is a $\Psi$-CSBP started from $0$. Therefore, one can naturally compare two ancestors: either by persistence, i.e. according to the extinction times of their progenies (if they become extinct in finite time); or by predominance, i.e. according to the asymptotic behaviours of their progenies (if their lifetimes are infinite). Notice that these two notions (persistence/predominance) are mutually exclusive.
\begin{theorem}\label{ThEves}
Assume that $\rZ$ does not reach $\infty$ in finite time. If the Eve property holds then one can order the ancestors by persistence/predominance. We denote this ordering $(\rme^{i})_{i\geq 1}$ and call these points the Eves. In particular, $\rme^{1}$ is the primitive Eve.
\end{theorem}
The Eves enjoy several nice properties. For instance, Proposition \ref{PropositionLawAncestors} shows that the sequence $(\rme^{i})_{i\geq 1}$ is i.i.d. uniform$[0,1]$. Also, the Eves will be of major importance in the last part of this work we now present.\\

Theorem \ref{TheoremFoP} shows that flows of subordinators and flows of partitions are related by their finite-dimensional marginals. One could wonder if the connection is deeper: does there exist a flow of partitions embedded into a flow of subordinators ? It turns out that the Eve property plays a crucial r\^ole in this topic.\vspace{2pt}

Consider a $\Psi$ flow of subordinators $(\rS_{s,t}(a),0 \leq s \leq t < \rT^\rS, 0 \leq a \leq \rS_s)$, and for simplicity let $\rZ_s := \rS_s$ denote the total population size and $\rT:=\rT^\rS$ its lifetime. For all $ s \leq t$, the subordinator $\rS_{s,t}$ defines a random measure $\rmm_{s,t}$ on $[0,\rZ_s]$ with total mass $\rZ_t$. Assume that $\rZ$ does not reach $\infty$ in finite time and that the Eve property is verified. Theorem \ref{ThEves} allows to introduce the \textit{Eves process} $(\rme^{i}_s,s\in [0,\rT))_{i\geq 1}$ by considering at each time $s \in[0,\rT)$, the sequence of Eves of the $\Psi$-MVBP $(\rmm_{s,t},t\in[s,\rT))$ that starts from the Lebesgue measure on $[0,\rZ_s]$. Notice that we actually rescale the Eves $(\rme^{i}_s)_{i\geq 1}$ by the mass $\rZ_s$ in order to obtain r.v. in $[0,1]$.\\
The Eves process is the set of individuals that play a significant r\^ole in the population as time passes. One is naturally interested in the genealogical relationships between these Eves, so we introduce a collection of partitions $(\hat{\Pi}_{s,t},0\leq s \leq t < \rT)$ by setting
\begin{equation*}
\hat{\Pi}_{s,t}(i) := \{j\in\mathbb{N}:\rme^j_t\mbox { descends from }\rme^i_s\}
\end{equation*}
Here "$\rme^j_t$ descends from $\rme^i_s$" means that $\rZ_t\cdot\rme^j_t\in\big(\rS_{s,t}(\rZ_s\cdot\rme^i_s-),\rS_{s,t}(\rZ_s\cdot\rme^i_s)\big]$.
\begin{theorem}\label{TheoremFoP2}
The collection of partitions $(\hat{\Pi}_{s,t},0 \leq s \leq t < \rT)$ defined from the flow of subordinators and the Eves process is a $\Psi$ flow of partitions.
\end{theorem}
We end with a decomposition result similar to the main theorem of~\cite{Labbe11}. For each time $s\in[0,\rT)$, let $\mathscr{E}_s(\hat{\Pi},(\rme^{i}_{s})_{i\geq 1})$ be the measure-valued process defined by
\begin{equation*}
[s,\rT)\ni t\mapsto\sum_{i\geq 1}|\hat{\Pi}_{s,t}(i)|\delta_{\rme^i_s}(dx) + \Big(1-\sum_{i\geq 1}|\hat{\Pi}_{s,t}(i)|\Big)dx
\end{equation*}
and $\rmr_{s,t}$ the probability measure on $[0,1]$ defined by
\begin{equation*}
\rmr_{s,t}(dx):= \frac{\rmm_{s,t}(\rZ_s\cdot dx)}{\rZ_t}
\end{equation*}
\vspace{-10pt}
\begin{theorem}\label{ThLookdown}
The flow of subordinators can be uniquely decomposed into two random objects: the Eves process $(\rme^{i}_s,s\in [0,\rT))$ and the flow of partitions $(\hat{\Pi}_{s,t},0\leq s \leq t < \rT)$.
\begin{enumerate}[i)]
\item \textbf{Decomposition}. For each $s \in \mathbb{R}$, a.s. $\mathscr{E}_s(\hat{\Pi},(\rme^{i}_{s})_{i\geq 1}) = (\rmr_{s,t},t \in [s,\rT))$
\item \textbf{Uniqueness}. Let $(\rH_{s,t},0 \leq s \leq t < \rT)$ be a $\Psi$ flow of partitions defined from the $\Psi$-CSBP $\rZ$, and for each $s \in [0,\rT)$, consider a sequence $(\upchi_s(i))_{i\geq 1}$ of r.v. taking distinct values in $[0,1]$. If for each $s \in [0,\rT)$, a.s. $\mathscr{E}_{s}(\rH,(\upchi_s(i))_{i\geq 1}) = (\rmr_{s,t},t \in [s,\rT))$ then\begin{itemize}
\item For each $s \in [0,\rT)$, a.s. $(\upchi_s(i))_{i\geq 1} = (\rme^{i}_s)_{i\geq 1}$.
\item Almost surely $\rH = \hat{\Pi}$.
\end{itemize}
\end{enumerate}
\end{theorem}
This theorem provides an embedding of the lookdown representation into a flow of subordinators and thus, unifies those two representations. Note that the Eve property is actually a necessary condition for the uniqueness. Indeed when the Eve property does not hold, there is no natural order on the ancestors and therefore no uniqueness of the embedding.

\section{Preliminaries}\label{SectionPreliminaries}
\subsection{Partitions of integers}\label{SubsectionPartitions}
For every $n\in\mathbb{N}\cup\{\infty\}$, let $\mathscr{P}_{n}$ be the set of partitions of $[n]:=\{1,\ldots,n\}$. We equip $\mathscr{P}_{\infty}$ with the distance $d_{\mathscr{P}}$ defined as follows. For all $\pi,\pi' \in \mathscr{P}_{\infty}$
\begin{equation}\label{EqMetricPartition}
d_{\mathscr{P}}(\pi,\pi') = 2^{-i} \Leftrightarrow i=\sup\{j \in \mathbb{N}:\pi^{[j]}=\pi'^{[j]}\}
\end{equation}
where $\pi^{[j]}$ is the restriction of $\pi$ to $[j]$. $(\mathscr{P}_{\infty},d_{\mathscr{P}})$ is a compact metric space. We also introduce for every $n\in\mathbb{N}\cup\{\infty\}$, $\mathscr{P}_{n}^*$ as the subset of $\mathscr{P}_{n}$ whose elements have a unique non-singleton block. In particular, for all subsets $\scK \subset \mathbb{N}$, we denote by $\tun_{\scK}$ the element of $\mathscr{P}_{\infty}^*$ whose unique non-singleton block is $\scK$. Also we denote by $\tO_{[\infty]}:=\{\{1\},\{2\},\ldots\}$ the trivial partition of $\bbN$ into singletons.\\
Let $\pi \in \mathscr{P}_{\infty}$, for each $i\geq 1$ we denote by $\pi(i)$ the $i$-th block of $\pi$ in the increasing order of their least element. Furthermore, the asymptotic frequency of $\pi(i)$ when it exists is defined to be
\begin{equation*}
|\pi(i)| = \lim\limits_{n\rightarrow\infty}\frac{1}{n}\sum_{j=1}^{n}\mathbf{1}_{\{j\in\pi(i)\}}
\end{equation*}
When all the blocks of a partition $\pi$ admit an asymptotic frequency, we denote by $|\pi|^{\downarrow}$ the sequence of its asymptotic frequencies in the decreasing order. We consider the Borel $\sigma$-field of $(\mathscr{P}_{\infty},d_{\mathscr{P}})$, and define an exchangeable random partition $\pi$ as a random variable on $\mathscr{P}_{\infty}$ whose distribution is invariant under the action of any permutation of $\mathbb{N}$, see Section 2.3.2 in~\cite{BertoinRandomFragmentation} for further details.\\
We define the coagulation operator $\Coag:\mathscr{P}_{\infty}\times\mathscr{P}_{\infty}\rightarrow \mathscr{P}_{\infty}$ as follows. For any elements $\pi,\pi' \in \mathscr{P}_{\infty}$, $\Coag(\pi,\pi')$ is the partition whose blocks are given by
\begin{equation}\label{EqCoag}
\Coag(\pi,\pi')(i) = \underset{j \in \pi'(i)}{\bigcup} \pi(j)
\end{equation}
for every $i \in \mathbb{N}$. This is a Lipschitz-continuous operator and we have
\begin{equation}\label{EqCoagAssociativity}
\Coag\big(\pi,\Coag(\pi',\pi'')\big) = \Coag\big(\Coag(\pi,\pi'),\pi''\big)
\end{equation}
for any elements $\pi,\pi',\pi'' \in \mathscr{P}_{\infty}$, see Section 4.2 in~\cite{BertoinRandomFragmentation} for further details.\\
We call mass-partition a sequence $s=(s_i)_{i\geq 1}$ such that $s_1 \geq s_2 \geq \ldots \geq 0, \sum_{i\geq 1}s_i\leq 1$. From a mass-partition $s$ one can define the paint-box based on $s$, that is, the distribution $\mathscr{P}(s)$ of the random exchangeable partition whose sequence of asymptotic frequencies is $s$. This can be achieved by considering a sequence $(U_i)_{i\geq 1}$ i.i.d. uniform$[0,1]$ and defining the random partition $\pi$ via the following equivalence relation
\begin{equation*}
i \sim j \Leftrightarrow \exists\, p \geq 1\mbox{ s.t. } U_i,U_j \in \left[\sum_{k=1}^{p-1}s_k,\sum_{k=1}^{p}s_k\right)
\end{equation*}
In this work, we will consider the mass-partition $(x,0,\ldots)$ associated to a point $x\in(0,1]$ and the corresponding paint-box distribution $\mathscr{P}(x,0,\ldots)$ in order to define the flow of partitions, see Subsection \ref{SubsectionStoFlow}.\\
Finally consider a subordinator $X$ restricted to $[0,a]$, with $a > 0$. On the event $\{X_a > 0\}$, the sequence $(\frac{\Delta X_t}{X_a})^{\downarrow}_{t \in[0,a]}$ will be called the mass-partition induced by the subordinator $X$, and the paint-box based on this sequence will be denoted by $\mathscr{P}(X)$. This can be achieved by considering an i.i.d. sequence $(U_{i})_{i\geq 1}$ of uniform$[0,1]$ r.v., and defining on the event $\{X_a > 0\}$ the exchangeable random partition $\pi$ by the following equivalence relation
\begin{equation}\label{EquationPB}
i \stackrel{\pi}{\sim} j \Leftrightarrow X^{-1}(X_{a}U_{i})=X^{-1}(X_{a}U_{j})
\end{equation}
where $X^{-1}$ denotes the right continuous inverse of $t\mapsto X_t$. We also complete the definition by setting $\mathscr{P}(X) := \tun_{[\infty]}=\{\{1,2,3,\ldots\}\}$ on the event $\{X_a=0\}$.

\subsection{Continuous-state branching processes}
\label{SubsectionCSBP}
We recall the definition of the continuous-state branching processes introduced in the celebrated article of Jirina~\cite{Jirina58}. A continuous-state branching process (CSBP for short) started from $a \geq 0$ is a Markov process $(\rZ_t^{a},t\geq 0)$ with values in $[0,\infty]$ such that $(\rZ_t^{a+b},t\geq 0)$ has the same distribution as $(\rZ_t^a + \rZ_t^b,t \geq 0)$ where $\rZ^a$ and $\rZ^b$ are two independent copies started from $a$ and $b$ respectively. Such a process is entirely characterized by a convex function $\Psi : [0,+\infty) \rightarrow (-\infty,+\infty)$, called its branching mechanism, via the following identity
\begin{equation}
\mathbb{E}[e^{-\lambda \rZ_t^{a}}] = e^{-au_t(\lambda)},\;\forall\lambda >0
\end{equation}
where the function $u_t(\lambda)$ solves
\begin{equation}\label{EquationUt}
\frac{\partial u_t(\lambda)}{\partial t} = -\Psi(u_t(\lambda)),\;\;u_0(\lambda) = \lambda
\end{equation}
and $\Psi$ is the Laplace exponent of a spectrally positive L\'evy process. Thus $\Psi$ has the following form
\begin{equation}\label{EquationPsi}
\Psi(u) = \gamma u + \frac{\sigma^2}{2}u^2 + \int_{0}^{\infty}\!\!\big(e^{-hu}-1+hu\mathbf{1}_{\{h \leq 1\}}\big)\,\nu(dh)
\end{equation}
where $\gamma \in \mathbb{R},\sigma \geq 0$ and $\nu$ is a measure on $(0,\infty)$ such that $\int_{0}^{\infty}(1\wedge x^2)\nu(dx) < \infty$. In the sequel, we will omit the symbol $a$ and consider $a = 1$ as the results we will expose do not depend on this value. Note that the semigroup is Feller, so a $\Psi$-CSBP admits a c\`adl\`ag modification. In the rest of this subsection, we consider implicitly a c\`adl\`ag modification of $\rZ$.\\
We say that the $\Psi$-CSBP is subcritical, critical or supercritical according as $\Psi'(0+)$ is positive, null, or negative. Furthermore since $0$ and $\infty$ are two absorbing states, we introduce the following two stopping times, namely the extinction time and the explosion time by setting
\begin{equation}
\rT_0 := \inf{\{t \geq 0 : \rZ_t = 0\}},\;\;\rT_{\infty} := \inf\{t \geq 0 : \rZ_t = \infty\}
\end{equation}
Let also $\rT := \rT_0 \wedge \rT_{\infty}$ denote the \textit{lifetime} of the $\Psi$-CSBP $\rZ$. Classical results entail that $\mathbb{P}(\rZ_{\rT}=0) = e^{-q}$, and therefore $\mathbb{P}(\rZ_{\rT}=\infty) = 1-e^{-q}$ where $q := \sup\{u\geq 0:\Psi(u)\leq 0\}$. Note that we use the convention $\sup\mathbb{R}_+ = \infty$. In~\cite{Grey74}, Grey provided a complete classification of the possible behaviours of $\rZ$ at the end of its lifetime:\vspace{-12pt}
\paragraph{Extinction.} For all $t > 0$, we have $\mathbb{P}(\rT_0 \leq t)= e^{-u_t(\infty)}$ and
\begin{equation*}
u_t(\infty) < \infty \Leftrightarrow \Psi(v) > 0\mbox{ for large enough }v\mbox{ and }\int_v^{\infty}\frac{du}{\Psi(u)} < \infty
\end{equation*}
If $u_t(\infty)$ is finite, then $u_t(\infty) \downarrow q$ as $t\rightarrow\infty$. This ensures that on the event $\{\rZ_{\rT}=0\}$ either $\rT <\infty$ a.s., or $\rT=\infty$ a.s.\vspace{-8pt}
\paragraph{Explosion.} For all $t > 0$, we have
\begin{equation}\label{EqExplosion}
\mathbb{P}(\rT_{\infty} > t) = \lim\limits_{\lambda\rightarrow 0+}\mathbb{E}[e^{-\lambda \rZ_t}] = e^{-u_t(0+)}
\end{equation}
Using this last equality, Grey proved that $\rT_{\infty} \stackrel{a.s.}{=} \infty \Leftrightarrow \int_{0+}\frac{du}{\Psi(u)} = \infty$. When this condition holds, we say that the CSBP is conservative. Here again, on the event $\{\rZ_{\rT}=+\infty\}$ either $\rT <\infty$ a.s., or $\rT=\infty$ a.s.\\
The proofs of the following two lemmas are postponed to Section \ref{Proofs}.
\begin{lemma}\label{LemmaDistributionAbsContinuous}
On the event $\{\rT < \infty\}$, $\rT$ has a distribution absolutely continuous with respect to the Lebesgue measure on $\mathbb{R}_+$.
\end{lemma}
For all $\epsilon \in (0,1)$, we introduce $\rT_\epsilon := \inf\{t \geq 0: \rZ_t \notin (\epsilon,1/\epsilon)\}$, and notice that $\rT_\epsilon < \rT$ a.s. for all $\epsilon \in (0,1)$.
\begin{lemma}\label{LemmaIntegrabilityCSBP}
For all $t \geq 0$ and all $\epsilon \in (0,1)$
\begin{equation*}
\mathbb{E}\left[\sum_{s\leq t\wedge \rT_\epsilon:\Delta \rZ_s> 0}\left(\frac{\Delta \rZ_s}{\rZ_s}\right)^2\right] < \infty
\end{equation*}
\end{lemma}

\subsection{Measure-valued branching processes and flows of subordinators}\label{SubsectionMVBP}
In this subsection, we introduce the measure-valued branching processes associated to a branching mechanism $\Psi$. For the sake of simplicity, we will consider measures on the interval $[0,1]$, but the definition holds for any other compact interval. Let $\mathscr{M}_f$ denote the set of finite measures on $[0,1]$ and let $\Delta$ be an extra point that will represent infinite measures. We set $\overline{\mathscr{M}_f} := \mathscr{M}_f\cup\{\Delta\}$ and equip this space with the largest topology that makes continuous the map
\begin{eqnarray*}
[0,+\infty]\times\mathscr{M}_1 &\rightarrow& \overline{\mathscr{M}_f}\\
(\lambda,\mu) &\mapsto& \begin{cases}\lambda\cdot\mu&\mbox{ if }\lambda<\infty\\
\Delta&\mbox{ if }\lambda=\infty\end{cases}
\end{eqnarray*}
This topology is due to Watanabe~\cite{Watanabe68}.\\
We denote by $\cB^{++}$ the set of bounded Borel functions on $[0,1]$ that admit a strictly positive infimum. We call measure-valued branching process associated with the branching mechanism $\Psi$, or $\Psi$-MVBP in short, a $\overline{\mathscr{M}_f}$-valued Markov process $(\rmm_t,t \geq 0)$ started from a given measure $\rmm_0\in \overline{\mathscr{M}_f}$ that verifies for all $f \in \cB^{++}$
\begin{equation*}
\mathbb{E}\big[\exp(-\left\langle \rmm_t,f\right\rangle)\big] = \exp(-\left\langle \rmm_0,u_t\circ f\right\rangle)
\end{equation*}
Note that $\left\langle \Delta,f\right\rangle = +\infty$, thus $\Delta$ is an absorbing point. The existence of this process can be obtained using a flow of subordinators as it will be shown below. The uniqueness of the distribution derives from the Markov property and the characterization of the Laplace functional on $\cB^{++}$.

It is straightforward to check that the total-mass process $(\rmm_t([0,1]),t \geq 0)$ is a $\Psi$-CSBP, say $\rZ$, started from $\rmm_0([0,1])$. As proved in~\cite{ElKarouiRoelly91}, this process verifies the branching property: for every $\rmm_0,\rmm'_0 \in \mathscr{M}_f$, the process $(\rmm_t+\rmm'_t,t \geq 0)$ is a $\Psi$-MVBP started from $\rmm_0+\rmm'_0$, where $(\rmm_t,t\geq 0)$ and $(\rmm'_t,t \geq 0)$ are two independent $\Psi$-MVBP started from $\rmm_0$ and $\rmm'_0$ respectively.\\
Finally, from Lemma 3.5.1 in~\cite{Dawson93} one can prove that its semigroup verifies the Feller property. This implies that the $\Psi$-MVBP admits a c\`adl\`ag modification. In the rest of this subsection, we consider implicitly a c\`adl\`ag modification of $\rmm$ and will denote by $\rT$ the lifetime of its total-mass process (which is necessarily a c\`adl\`ag $\Psi$-CSBP).\\

Suppose that $(\rmm_t,t \geq 0)$ starts from the Lebesgue measure on $[0,1]$. It is then immediate to deduce that for all $t \geq 0$, the process $x\mapsto \rmm_t([0,x])$ is a (possibly killed) subordinator whose Laplace exponent is given by $(u_t(\lambda),\lambda >0)$. From the L\'evy-Khintchine formula, we deduce that there exists a real number $d_t \geq 0$ and a measure $w_t$ on $(0,\infty)$ that verifies $\int_{0}^{\infty}(1\wedge h)w_t(dh) < \infty$, such that
\begin{equation}\label{EquationLaplaceExponent}
u_t(\lambda) = u_t(0+) + d_t\lambda + \int_{0}^{\infty}(1-e^{-\lambda h})\,w_t(dh)\mbox{ , for all }\lambda > 0
\end{equation}
Notice that $u_t(0+)$ is the instantaneous killing rate of the subordinator, which is related to the explosion of the total mass process $(\rmm_t([0,1]),t\geq 0)$. Indeed $\mathbb{P}(\rT_{\infty} \leq t)= \mathbb{P}(\rmm_t([0,1]) = \infty) = 1-e^{-u_t(0+)}$, see Equation (\ref{EqExplosion}).\\

From this observation, Bertoin and Le Gall introduced an object called flow of subordinators. Proposition 1 in~\cite{BertoinLeGall-0} asserts the existence of a process $(\rS_{s,t}(a),0 \leq s \leq t, a \geq 0)$ such that
\begin{itemize}
\item For every $0 \leq s \leq t$, $(\rS_{s,t}(a),a\geq 0)$ is a subordinator with Laplace exponent $u_{t-s}$.
\item For every integer $p\geq 2$ and $0 \leq t_1\leq \ldots \leq t_p$, the subordinators $\rS_{t_1,t_2},\ldots,\rS_{t_{p-1},t_p}$ are independent and
\begin{equation*}
\rS_{t_1,t_p}(a) = \rS_{t_{p-1},t_p}\circ\ldots\circ \rS_{t_1,t_2}(a),\ \forall a \geq 0\mbox{ a.s. (cocycle property)}
\end{equation*}
\item For all $a \geq 0$, $(\rS_{0,t}(a),t\geq 0)$ is a $\Psi$-CSBP started from $a$.
\end{itemize}
Actually in their construction, they excluded the non-conservative branching mechanisms but one can easily adapt their proof to the general case.\\
Let us now present the connection with the $\Psi$-MVBP. Introduce the random Stieltjes measures
\begin{equation*}
\rmm_{0,t}(dx) := d_x\rS_{0,t}(x),\;\forall x \in[0,1]
\end{equation*}
From the very definition of the flow of subordinators, one can prove that $(\rmm_{0,t},t\geq 0)$ is a $\Psi$-MVBP started from the Lebesgue measure on $[0,1]$. One can consider a c\`adl\`ag modification still denoted $(\rmm_{0,t},t\geq 0)$, and let $\rT^\rS$ be the lifetime of its total-mass process $\rS_t:=\rmm_{0,t}([0,1]),t\geq 0$ (which is a c\`adl\`ag $\Psi$-CSBP). It is then natural to introduce the random Stieltjes measures
\begin{equation*}
\rmm_{s,t}(dx) := d_x\rS_{s,t}(x),\;\forall x \in[0,\rS_s]
\end{equation*}
for every $0 < s \leq t < \rT^\rS$. Each process $(\rmm_{s,t},t\in[s,\rT^\rS))$ is a $\Psi$-MVBP and admits a c\`adl\`ag modification still denoted $(\rmm_{s,t},t\in[s,\rT^\rS))$. Then we obtain a flow of $\Psi$-MVBP $(\rmm_{s,t},0 \leq s \leq t < \rT^\rS)$ that describes the evolution of an initial population $[0,1]$.

\section{Flows of partitions and the lookdown representation}\label{SectionFlow}
The goal of this section is to develop the construction of $\Psi$ flows of partitions presented in the introduction. To that end, we first recall the definition of deterministic flows of partitions as introduced in~\cite{Labbe11} since the one-to-one correspondence with lookdown graphs is deterministic. Then, we define a random point process ${\cal P}$ pathwise from a c\`adl\`ag $\Psi$-CSBP $\rZ$, which will allow us to construct a $\Psi$ flow of partitions. Finally we give a precise characterization of its jump rates which will be necessary in the proof of Theorem \ref{TheoremLimit}, this last subsection can be skipped on first reading.

\subsection{Deterministic flows of partitions}
\label{SubsectionDetFoP}
Fix $T \in (0,+\infty]$. In~\cite{Labbe11}, we introduced deterministic flows of partitions and proved they are in one-to-one correspondence with the so-called lookdown graphs. Lookdown graphs are implicit in the lookdown construction of Donnelly and Kurtz~\cite{DK99}, and the upshot of the flows of partitions is to clarify and formalize this notion. In the present paper, we do not recall the definition of the lookdown graph and we refer to~\cite{Labbe11} for further details.\\
Below this formal definition of deterministic flows of partitions, the reader should find intuitive comments.
\begin{definition}
A deterministic flow of partitions on $[0,T)$ is a collection $\hat{\pi} = (\hat{\pi}_{s,t},0 \leq s \leq t < T)$ of partitions such that
\begin{itemize}
\item For every $r < s < t \in [0,T)$, $\hat{\pi}_{r,t} = \Coag(\hat{\pi}_{s,t},\hat{\pi}_{r,s})$.
\item For every $s \in (0,T)$, $\lim\limits_{r\uparrow s}\lim\limits_{t\uparrow s}\hat{\pi}_{r,t} =:\hat{\pi}_{s-,s-} = \tO_{[\infty]}$.
\item For every $s \in [0,\rT)$, $\lim\limits_{t\downarrow s}\hat{\pi}_{s,t} = \hat{\pi}_{s,s} = \tO_{[\infty]}$.
\end{itemize}
Furthermore, if for all $s \in (0,T)$, $\hat{\pi}_{s-,s}$ has at most one unique non-singleton block, then we say that $\hat{\pi}$ is a deterministic flow of partitions without simultaneous mergers.
\end{definition}
The first property asserts a cocycle property for the collection of partitions: the evolution forward-in-time is obtained by coagulating consecutive partitions. The second and third properties ensure that for all $n\geq 1$ and every compact interval $[r,t] \subset [0,T)$, only a finite number of partitions $\hat{\pi}_{s-,s}^{[n]}$ differ from the trivial partition $\tO_{[n]}$. Note that in this paper, we will only consider flows of partitions without simultaneous mergers.\vspace{3pt}\\
\textit{Construction from a point process}\\
Let $p$ be a \textit{deterministic} point process on $[0,T)\times\mathscr{P}_{\infty}^*$ whose restriction to any subset of the form $(s,t]\times\mathscr{P}_{n}^*$ has finitely many points. Fix an integer $n \in \mathbb{N}$ and two real numbers $s \leq t \in [0,T)$. Let $(t_i,\varrho_i)_{1 \leq i \leq q}$ be the finitely many points of $p_{|(s,t]\times\mathscr{P}^*_{n}}$ in the increasing order of their time coordinate. We introduce
\begin{equation}
\hat{\pi}^{[n]}_{s,t} := \Coag(\varrho_q,\Coag(\varrho_{q-1},\ldots,\Coag(\varrho_{2},\varrho_1)\ldots))
\end{equation}
Obviously, the collection of partitions $(\hat{\pi}^{[n]}_{s,t},n \in \mathbb{N})$ is compatible and defines by a projective limit a unique partition $\hat{\pi}_{s,t}$ such that its restriction to $[n]$ is $\hat{\pi}^{[n]}_{s,t}$, for each $n \in \mathbb{N}$. Then, one easily verifies that $(\hat{\pi}_{s,t},0 \leq s \leq t < T)$ is a deterministic flow of partitions without simultaneous mergers.
\begin{remark}
This construction gives a hint of the one-to-one correspondence with lookdown graphs. See~\cite{Labbe11} for further details.
\end{remark}

We can now introduce the lookdown representation using a deterministic flow of partitions. Let $(\upxi_{s,s}(i))_{i\geq 1}$ be a sequence of points in $\mathbb{R}_+$ and define the particle system $(\upxi_{s,t}(i),t \in [s,T))_{i\geq 1}$ as follows. For all $t \geq s$ and all $i,j \geq 1$,
\begin{equation}
\upxi_{s,t}(j) = \upxi_{s,s}(i) \Leftrightarrow j \in \hat{\pi}_{s,t}(i)
\end{equation}
\vspace{-15pt}
\begin{definition}
We use the notation $\mathscr{L}_s(\hat{\pi},(\upxi_{s,s}(i))_{i\geq 1})$ to denote the deterministic lookdown function $(\upxi_{s,t}(i),t \in [s,\infty))_{i\geq 1}$ defined from the flow of partitions $\hat{\pi}$ and the initial types $(\upxi_{s,s}(i))_{i\geq 1}$.
\end{definition}
Moreover, for all $t \in [s,T)$, set
\begin{equation}
\Xi_{s,t}(dx) := \lim\limits_{n\rightarrow\infty}\frac{1}{n}\sum_{i=1}^{n}\delta_{\upxi_{s,t}(i)}(dx)
\end{equation}
when this is well-defined.
\begin{definition}
We denote by $\mathscr{E}_s(\hat{\pi},(\upxi_{s,s}(i))_{i\geq 1})$ the collection of limiting empirical measures $(\Xi_{s,t},t \in [s,\infty))$ defined from the flow of partitions $\hat{\pi}$ and the initial types $(\upxi_{s,s}(i))_{i\geq 1}$, when it exists.
\end{definition}
Let us give an intuitive explanation of this particle system. If one considers each point $\upxi_{s,s}(i)$ as some characteristic (type or location for instance) of the $i$-th ancestor at time $s$, then the underlying idea of the lookdown representation is to give the same characteristic to the descendants of this ancestor at any time $t > s$. Therefore, the measure $\Xi_{s,t}(dx)$ describes the composition of the population at time $t$: $\Xi_{s,t}(\{\upxi_{s,s}(i)\})$ is the proportion of individuals at time $t$ who descend from the $i$-th ancestor alive at time $s$. In the next subsection, we will see that if one applies this scheme with a random flow of partitions, whose distribution is well chosen, then $\Xi$ is a MVBP (rescaled by its total-mass).

\subsection{Stochastic flows of partitions associated with a branching mechanism}\label{SubsectionStoFlow}
We randomize the previous definitions using a point process ${\cal P}$ on $[0,\rT)\times\mathscr{P}_{\infty}$ where $\rT$ is a random positive time in order to introduce flows of partitions associated with a branching mechanism $\Psi$. As mentioned in the introduction, this point process is obtained as the union of two point processes: ${\cal N}_{\sigma}$ that stands for the binary reproduction events due to the diffusion of the underlying CSBP, and ${\cal N}_{\nu}$ that encodes the positive frequency reproduction events due to the jumps of the CSBP. As these objects rely on many definitions, one should refer on first reading to the heuristic definitions given in the introduction.\\

For every $z > 0$, we introduce the map $\phi_z:\mathbb{R}_+^*\rightarrow [0,1]$, that will be used to consider rescaled jumps of a CSBP, by setting
\begin{equation*}
\phi_z:h\mapsto\frac{h}{h+z}
\end{equation*}
We define a measure $\mu^{\mbox{\tiny binary}}$ on $\mathscr{P}_{\infty}$ that will encode binary reproduction events often called Kingman reproduction events. Recall that $\tun_{\{i,j\}}$ stands for the element of $\mathscr{P}_{\infty}^*$ whose unique non-singleton block is $\{i,j\}$ for every integers $1 \leq i < j$.
\begin{equation}
\mu^{\mbox{\tiny binary}}(d\pi) := \sum_{i<j}\delta_{\tun_{\{i,j\}}}(d\pi)
\end{equation}

This ends the introduction of preliminary notation. Fix a branching mechanism $\Psi$ and consider a $\Psi$-CSBP $(\rZ_t,t \geq 0)$ started from $1$ assumed to be c\`adl\`ag. We keep the notation of Section \ref{SubsectionCSBP}, in particular $\rT$ denotes the lifetime of $\rZ$.
We start with the definition of ${\cal N}_{\nu}$. Consider the random point process
\begin{equation*}
{\cal Q} := \underset{\{t\geq 0: \Delta \rZ_t > 0\}}{\bigcup}\Big\{\Big(t,\frac{\Delta \rZ_t}{\rZ_t}\Big)\Big\}
\end{equation*}
and define a $\mathscr{P}$-randomization ${\cal N}_{\nu}$ of ${\cal Q}$ in the sense of Chapter 12 in~\cite{KallenbergBook}, where $\mathscr{P}$ is the paint-box probability kernel introduced in Subsection \ref{SubsectionPartitions}. The point process ${\cal N}_{\nu} := \cup\{(t,\frac{\Delta \rZ_t}{\rZ_t},\varrho_t)\}$ on $\mathbb{R}_+\times[0,1]\times\mathscr{P}_{\infty}$ can be described as follows. For all $t \geq 0$ such that $\Delta \rZ_t > 0$, $\varrho_t$ is a r.v. on $\mathscr{P}_{\infty}$ distributed according to the paint-box distribution $\mathscr{P}(\frac{\Delta \rZ_t}{\rZ_t},0,\ldots)$. It is more convenient to consider the restriction of this point process to $\mathbb{R}_+\times\mathscr{P}_{\infty}$ still denoted by ${\cal N}_{\nu} = \cup\{(t,\varrho_t)\}$.\vspace{5pt}

Second, we define a doubly stochastic Poisson point process ${\cal N}_{\sigma}$ on $\mathbb{R}_+\times\mathscr{P}_{\infty}$, in the sense of Chapter 12 in~\cite{KallenbergBook}, with a random intensity measure given by
\begin{equation}
\mathbf{1}_{\{t<\rT\}}dt\otimes\frac{\sigma^2}{\rZ_{t}}\mu^{\mbox{\tiny binary}}(d\pi)
\end{equation}

We finally define the point process $\cal P$ on $\mathbb{R}_+\times\mathscr{P}_{\infty}$ as
\begin{equation}
{\cal P} := {\cal N}_{\sigma} \cup {\cal N}_{\nu}
\end{equation}
Notice that almost surely this point process takes its values in $\mathbb{R}_+\times\mathscr{P}_{\infty}^*$, and has finitely many points in any set of the form $[0,t]\times\mathscr{P}_{n}^*$ with $t < \rT$ and $n \in \mathbb{N}$, as we will see in Proposition \ref{PropCompensator}. Thus for each $\omega\in\Omega$, we define a deterministic flow of partitions without simultaneous mergers $(\hat{\Pi}_{s,t}(\omega),0 \leq s \leq t < \rT)$ using the point collection ${\cal P}(\omega)$ and the pathwise construction of Subsection \ref{SubsectionDetFoP}.\\

Let us now explain how one defines a lookdown process associated with a $\Psi$-MVBP. Fix $s \geq 0$ and condition on $\{s < \rT\}$. Consider a sequence $(\upxi_{s,s}(i))_{i\geq 1}$ of i.i.d. uniform$[0,1]$ r.v. and define the lookdown process $(\upxi_{s,t}(i),t \in [s,\rT))_{i\geq 1} := \mathscr{L}_s(\hat{\Pi},(\upxi_{s,s}(i)))_{i\geq 1}$. Lemma 3.5 in~\cite{DK99} ensures that almost surely this particle system admits a process of limiting empirical measures $(\Xi_{s,t},t \in [s,\rT)) := \mathscr{E}_s(\hat{\Pi},(\upxi_{s,s}(i))_{i\geq 1})$, and almost surely for all $t \in [s,\rT)$ we have
\begin{equation*}
\Xi_{s,t}(dx) = \sum_{i\geq 1}|\hat{\Pi}_{s,t}(i)|\delta_{\upxi_{s,s}(i)}(dx) + \Big(1-\sum_{i\geq 1}|\hat{\Pi}_{s,t}(i)|\Big)dx
\end{equation*}
Moreover, Section 2 in~\cite{Article7} shows that the process $(\rZ_t\cdot\Xi_{s,t}(\rZ_s\cdot dx),t \in [s,\rT))$ is a c\`adl\`ag $\Psi$-MVBP started from the Lebesgue measure on $[0,\rZ_s]$, conditionally on $\rZ_s$.
\begin{remark}
The results in~\cite{Article7,DK99} are stated with the usual notion of lookdown graph. But they are immediately translated in terms of flows of partitions thanks to our one-to-one correspondence.
\end{remark}
\begin{remark}
We can define from any time $s \in [0,\rT)$, a $\Psi$-MVBP with total-mass process $\rZ$ using an independent sequence of initial types $(\upxi_{s,s}(i))_{i\geq 1}$ and the flow $\hat{\Pi}$. Then, it could seem simple to define a flow of $\Psi$-MVBP using this lookdown representation simultaneously for all $s \in [0,\rT)$. However, this is far from being trivial since the initial types $s \mapsto (\upxi_{s,s}(i))_{i\geq 1}$ have to be suitably coupled. In Section \ref{SectionPathwise}, we will show that these initial types have to be the Eves.
\end{remark}
An important property of the lookdown process (see~\cite{DK99}) is that for all $t\geq s$, conditionally given $\rZ_t$ the sequence $(\upxi_{s,t}(i))_{i\geq 1}$ is exchangeable on $[0,1]$. This implies that conditionally given $\rZ_t$ the partition $\hat{\Pi}_{s,t}$ has the paint-box distribution on the subordinator $x\mapsto \rZ_t\cdot\Xi_{s,t}([0,x])$. More generally, we have
\begin{theorem*}\textup{\textbf{\ref{TheoremFoP}}}
The collection of partitions $(\hat{\Pi}_{s,t},0 \leq s \leq t < \rT)$, together with its underlying CSBP $\rZ$, satisfies
\begin{itemize}
\item For all $n\geq 1$ and all $0 < t_1 < \ldots < t_n$, 
\begin{equation*}
(\rZ_{t_1},\ldots,\rZ_{t_n},\hat{\Pi}_{0,t_1},\ldots,\hat{\Pi}_{t_{n-1},t_n}|t_n < \rT)\! \stackrel{(d)}{=}\! (\rS_{t_1},\ldots,\rS_{t_n},\mathscr{P}(\rS_{0,t_1}),\ldots,\mathscr{P}(\rS_{t_{n-1},t_n})|t_n<\rT^\rS)
\end{equation*}
\item For all $0 \leq r < s < t < \rT$, a.s. $\hat{\Pi}_{r,t} = \Coag(\hat{\Pi}_{s,t},\hat{\Pi}_{r,s})$ (cocycle property).
\end{itemize}
\end{theorem*}
\begin{proof}
The cocycle property is a consequence of our construction as we have defined the restrictions of the partitions by coagulating elementary reproduction events. We turn our attention to the finite dimensional distributions. Fix an integer $n \geq 1$ and a $n$-tuple $0 = t_0 < t_1 <\ldots < t_n$. Let $(\rS_{s,t}(a),0 \leq s \leq t,0 \leq a \leq \rS_{0,s}(z))$ be a $\Psi$ flow of subordinators restricted to an initial population $[0,z]$ instead of $[0,1]$, and keep the notation $\rT^\rS$ to denote the lifetime of the total mass process $(\rS_t:=\rS_{0,t}(z),t\geq 0)$ which is a $\Psi$-CSBP started from $z$. For every $i\in[n]$, let $\rH_{t_{i-1},t_{i}}$ be distributed according to the paint-box $\mathscr{P}(\rS_{t_{i-1},t_i})$. Note that $\rH_{0,t_1},\ldots,\rH_{t_{n-1},t_{n}}$ are coupled only through their mass-partitions. We use our construction of the beginning of this subsection to define pathwise from the CSBP $(\rS_t,t\geq 0)$ a collection $(\hat{\Pi}^\rS_{s,t},0 \leq s \leq t < \rT^\rS)$. We will use the notation $\mathbb{P}_z$ to emphasize the dependence on the initial value $z$. Implicitly, $f_i$ will denote a bounded Borel map from $\mathscr{P}_{\infty}$ to $\mathbb{R}$ and $g_i$ a bounded Borel map from $\mathbb{R}_+$ to $\mathbb{R}$. We now prove by recursion on $n\geq 1$ that
\begin{eqnarray*}
&&\mathbb{E}_z[f_1(\hat{\Pi}^\rS_{0,t_1})g_1(\rS_{t_1})\ldots f_{n}(\hat{\Pi}^\rS_{t_{n-1},t_{n}})g_{n}(\rS_{t_{n}})\mathbf{1}_{\{t_{n} < \rT^\rS\}}]\\
&=&\mathbb{E}_z[f_1(\rH_{0,t_1})g_1(\rS_{t_1})\ldots f_n(\rH_{t_{n-1},t_n})g_n(\rS_{t_n})\mathbf{1}_{\{t_n < \rT^\rS\}}]
\end{eqnarray*}
The case $n=1$ follows from the discussion above the statement of the theorem. Fix $n\geq 2$ and suppose that for all $z > 0$ and all $f_1,\ldots,f_{n-1}, g_1,\ldots,g_{n-1}$, we have
\begin{eqnarray*}
&\mathbb{E}_z[f_1(\hat{\Pi}^\rS_{0,t_1})g_1(\rS_{t_1})\ldots f_{n-1}(\hat{\Pi}^\rS_{t_{n-2},t_{n-1}})g_{n-1}(\rS_{t_{n-1}})\mathbf{1}_{\{t_{n-1} < \rT^\rS\}}]\\
&=\mathbb{E}_z[f_1(\rH_{0,t_1})g_1(\rS_{t_1})\ldots f_{n-1}(\rH_{t_{n-2},t_{n-1}})g_{n-1}(\rS_{t_{n-1}})\mathbf{1}_{\{t_{n-1} < \rT^\rS\}}]
\end{eqnarray*}
Then, we obtain at rank $n$ for any given $z > 0$ and any $f_1,\ldots,f_n,g_1,\ldots,g_n$ 
\begin{eqnarray*}
&&\mathbb{E}_z[f_1(\hat{\Pi}^\rS_{0,t_1})g_1(\rS_{t_1})\ldots f_n(\hat{\Pi}^\rS_{t_{n-1},t_n})g_n(\rS_{t_n})\mathbf{1}_{\{t_n < \rT^\rS\}}]\\
&=&\mathbb{E}_z\left[f_1(\hat{\Pi}^\rS_{0,t_1})g_1(\rS_{t_1})\mathbf{1}_{\{t_1<\rT^\rS\}}\mathbb{E}_{\rS_{t_1}}\big[f_2(\hat{\Pi}^\rS_{0,t_2-t_1})\ldots f_n(\hat{\Pi}^\rS_{t_{n-1}-t_1,t_n-t_1})g_n(\rS_{t_n-t_1})\mathbf{1}_{\{t_n-t_1 < \rT^\rS\}}\big]\right]\\
&=&\mathbb{E}_z\left[f_1(\rH_{0,t_1})g_1(\rS_{t_1})\mathbf{1}_{\{t_1<\rT^\rS\}}\mathbb{E}_{\rS_{t_1}}\big[f_2(\rH_{0,t_2-t_1})\ldots f_n(\rH_{t_{n-1}-t_1,t_n-t_1})g_n(\rS_{t_n-t_1})\mathbf{1}_{\{t_n-t_1 < \rT^\rS\}}\big]\right]\\
&=&\mathbb{E}_z[f_1(\rH_{0,t_1})g_1(\rS_{t_1})\ldots f_n(\rH_{t_{n-1},t_n})g_n(\rS_{t_n})\mathbf{1}_{\{t_n < \rT^\rS\}}]
\end{eqnarray*}
where the first (resp. last) equality comes from the Markov property applied to the process $(\rS_t,t\in[0,\rT^\rS))$ (resp. to the homogeneous chain $(\rS_{t_i},\rH_{t_{i-1},t_i},t_{i+1}-t_i)_{1\leq i \leq n}$) while the second equality makes use of the recursion hypothesis and the case $n=1$.\cqfd
\end{proof}
This result motivates the following definition.
\begin{definition}
A collection of random partitions $(\hat{\Pi}_{s,t},0 \leq s \leq t < \rT)$ defined on a same probability space as a c\`adl\`ag $\Psi$-CSBP $(\rZ_t,t\in [0,\rT))$ and that verifies
\begin{itemize}
\item For all $n\geq 1$ and all $0 < t_1 < \ldots < t_n$, 
\begin{equation*}
(\rZ_{t_1},\ldots,\rZ_{t_n},\hat{\Pi}_{0,t_1},\ldots,\hat{\Pi}_{t_{n-1},t_n}|t_n < \rT)\! \stackrel{(d)}{=}\! (\rS_{t_1},\ldots,\rS_{t_n},\mathscr{P}(\rS_{0,t_1}),\ldots,\mathscr{P}(\rS_{t_{n-1},t_n})|t_n<\rT^\rS)
\end{equation*}
\item For all $0 \leq r < s < t < \rT$, a.s. $\hat{\Pi}_{r,t} = \Coag(\hat{\Pi}_{s,t},\hat{\Pi}_{r,s})$ (cocycle property).
\end{itemize}
is called a $\Psi$ flow of partitions. $\rZ$ is called its underlying CSBP.
\end{definition}
\begin{remark}
In our construction from a point process, we can verify that the cocycle property is fulfilled almost surely simultaneously for all triplets, that is,
\begin{equation*}
\mathbb{P}\big[\,\forall\, 0 \leq r < s < t < \rT,\;\hat{\Pi}_{r,t} = \Coag(\hat{\Pi}_{s,t},\hat{\Pi}_{r,s})\big]=1
\end{equation*}
This is not necessarily the case for a general $\Psi$ flow of partitions: however Proposition \ref{PropositionRegularization} will show that we can define a regularized modification which fulfils that property.
\end{remark}

\subsection{A characterization of the jump rates}\label{SubsectionJumpRates}
The formalism of partitions enables one to restrict to $n$ individuals sampled uniformly among the population. In this subsection, we give a characterization of the dynamics of this finite-dimensional process. The restriction of $\cal P$ to $\mathbb{R}_+\times\mathscr{P}_{n}^*$ is denoted ${\cal P}^{[n]}$. We introduce, for any integer $2 \leq k \leq n$ and any subset $\scK \subset [n]$ such that $\#\scK = k$, the quantity
\begin{equation}
\rL_{t}(n,\scK) := \#\big\{r \in (0,t] : (r,\tun_\scK^{[n]}) \in {\cal P}^{[n]}\big\}
\end{equation}
where $\tun_\scK^{[n]}$ is the restriction of $\tun_\scK$ to $[n]$ and $\tun_\scK$ is the partition whose unique non-singleton block is $\scK$. Moreover, we set
\begin{equation}
\rL_{t}(n) := \sum_{\{\scK \subset[n]:\#\scK\geq 2\}} \rL_{t}(n,\scK)
\end{equation}
In words, $\rL_{t}(n)$ is the total number of points of ${\cal P}$ restricted to $(0,t]\times{\mathscr{P}_{n}^*}$. Note that the collection of processes $\{(\rL_t(n,\scK),t \in [0,\rT));\scK \subset [n],\#\scK\geq 2\}$ is completely equivalent with the restricted flow $(\hat{\Pi}_{s,t}^{[n]},0 \leq s \leq t < \rT)$ : the knowledge of any one of them is sufficient to recover the other. We denote by $d_n$ the number of subsets of $[n]$ with at least $2$ elements, that is, $d_n := \sum_{k=2}^{n}\binom{n}{k}$ and we introduce the filtration $({\cal F}_t, t\geq 0)$ by setting for all $t \geq 0$
\begin{equation}
{\cal F}_t := \sigma\big\{\rZ_s,s \in [0,t]\big\}\bigvee\sigma\big\{{\cal P}_{[0,t]\times\mathscr{P}_{\infty}}\big\}
\end{equation}
For every integer $k$ such that $2 \leq k \leq n$, we set
\begin{equation*}
\lambda_{n,k}(z,\Psi) := \int_{0}^{1}x^k(1-x)^{n-k}\Big(\frac{\sigma^2}{z}x^{-2}\delta_{0}(dx)+z\,\nu\circ\phi_z^{-1}(dx)\Big)
\end{equation*}
where $\nu\circ\phi_z^{-1}$ is the pushforward measure of $\nu$ through the map $\phi_z$. Notice that $\lambda_{n,k}$ can be seen as a map from $\mathbb{R}_{+}^*\times\mathscr{M}_f(\mathbb{R_+})$ to $\mathbb{R}_+$. Indeed, any element of $\mathscr{M}_f(\mathbb{R_+})$ has the form $\tilde{\sigma}^2\delta_{0}(dh) + (1\wedge h^2)\tilde{\nu}(dh)$, where $\tilde{\sigma} \geq 0$ and $\tilde{\nu}$ is a measure on $(0,\infty)$ such that $\int_{0}^{\infty}(1\wedge h^2)\tilde{\nu}(dh) < \infty$, so it can be associated to the branching mechanism $\tilde{\Psi}$ defined by the triplet $(\tilde{\gamma} = 0,\tilde{\sigma},\tilde{\nu})$.\\
This ends the introduction of notation. We now state two results: the first is a technical continuity statement and the second is the characterization of the jump rates. They will be of main importance for the proof of Theorem \ref{TheoremLimit}.
\begin{proposition}\label{PropContinuityLambda}
Fix $k,n$ such that $2 \leq k \leq n$. The map $\lambda_{n,k}$ is continuous from $\mathbb{R}_{+}^*\times\mathscr{M}_f(\mathbb{R_+})$, endowed with the product topology, to $\mathbb{R}_+$.
\end{proposition}
The proof of this first result will be given in Section \ref{Proofs}, as it is rather technical. For the next result, we rely on notions of stochastic calculus introduced in Chapters I.3 and II.2 in~\cite{JacodShiryaev}.
\begin{proposition}\label{PropCompensator}
The collection of counting processes $\{(\rL_t(n,\scK),t \in [0,\rT));\scK \subset [n],\#\scK\geq 2\}$ is a pure-jump $d_n$-dimensional semimartingale on $[0,\rT)$ with respect to $\cal F$. Its predictable compensator is the $d_n$-dimensional process
\begin{equation}
\Big\{\Big(\int_{0}^{t}\lambda_{n,\#\scK}(\rZ_{s-},\Psi)ds, t \in [0,\rT)\Big);\scK \subset [n],\#\scK\geq 2\Big\}
\end{equation}
\end{proposition}
\begin{proof}
Fix $\scK \subset [n]$ such that its cardinality, denoted by $k := \#\scK$, is greater than $2$. It is straightforward to check that $(\rL_t(n,\scK),t \in [0,\rT))$ is a counting process adapted to the filtration $\cal F$. Similarly, one can easily verify that the process
\begin{equation*}
\Big(\int_{0}^{t}\lambda_{n,k}(\rZ_{s-},\Psi)ds, t \in [0,\rT)\Big)
\end{equation*}
is a predictable increasing process w.r.t. the filtration ${\cal F}$. Let us prove that the process
\begin{equation*}
\Big(\rL_t(n,\scK)-\int_{0}^{t}\lambda_{n,k}(\rZ_{s-},\Psi)ds,t \in [0,\rT)\Big)
\end{equation*}
is a local martingale on $[0,\rT)$ w.r.t. $\cal F$. To do so, set for all $t \in [0,\rT)$
\begin{eqnarray*}
M^{(1)}_t &:=& \rL_t(n,\scK)-\sum_{s \leq t}\Big(\frac{\Delta \rZ_s}{\rZ_s}\Big)^{k}\Big(1-\frac{\Delta \rZ_s}{\rZ_s}\Big)^{n-k}-\mathbf{1}_{\{k=2\}}\int_0^{t}\frac{\sigma^2}{\rZ_{s-}}\,ds\\
M^{(2)}_t &:=& \sum_{s \leq t}\Big(\frac{\Delta \rZ_s}{\rZ_s}\Big)^{k}\Big(1-\frac{\Delta \rZ_s}{\rZ_s}\Big)^{n-k}-\int_{0}^{t}\int_{0}^{1}x^{k}(1-x)^{n-k}\rZ_{s-}\,\nu\circ\phi_{\rZ_{s-}}^{-1}(dx)\,ds
\end{eqnarray*}
It is sufficient to show that both $M^{(1)}$ and $M^{(2)}$ are local martingales on $[0,\rT)$ w.r.t. $\cal F$.\\
Let us focus on the first one. Fix $\epsilon \in (0,1)$ and recall the definition of the stopping time $\rT_\epsilon := \inf\{t \geq 0: \rZ_t \notin (\epsilon, 1/\epsilon)\}$. Condition on $(\rZ_s,s \in [0,\rT))$ and consider a time $s > 0$ such that $\Delta \rZ_s > 0$ (note that those times are countably many a.s.). The $\mathscr{P}$-randomization procedure implies that the restriction of the random partition $\varrho_s$ to $\mathscr{P}_n$ has a probability $(\frac{\Delta \rZ_s}{\rZ_s})^{k}(1-\frac{\Delta \rZ_s}{\rZ_s})^{n-k}$ to be equal to $\tun_\scK^{[n]}$ independently of the other partitions $(\varrho_t)_{\{t\ne s : \Delta \rZ_t > 0\}}$. For all $t \geq 0$, the number of occurrences of the partition $\tun_\scK^{[n]}$ in ${\cal N}_{\nu}$ restricted to $[0,t\wedge \rT_\epsilon]\times\mathscr{P}_{n}$ is given by the following r.v.
\begin{equation}\label{EqNbOccurences}
\#\big\{s\in [0,t\wedge \rT_\epsilon]:(s,\tun_\scK^{[n]}) \in {\cal N}_{\nu}^{[n]}\big\}
\end{equation}
which is, therefore, distributed as the sum of a sequence, indexed by $\{s \in [0,t\wedge \rT_\epsilon]:\Delta \rZ_s >0\}$, of independent Bernoulli r.v. with parameters $((\frac{\Delta \rZ_s}{\rZ_s})^{k}(1-\frac{\Delta \rZ_s}{\rZ_s})^{n-k})$. Since
\begin{equation}\label{EqInequalityJumps}
\Big(\frac{\Delta \rZ_s}{\rZ_s}\Big)^{k}\Big(1-\frac{\Delta \rZ_s}{\rZ_s}\Big)^{n-k} \leq \Big(\frac{\Delta \rZ_s}{\rZ_s}\Big)^{2}
\end{equation}
a simple application of Borel-Cantelli lemma together with Lemma \ref{LemmaIntegrabilityCSBP} ensures that the r.v. of Equation (\ref{EqNbOccurences}) is finite a.s. One also easily deduces that for all $t \geq 0$
\begin{equation*}
\mathbb{E}\left[\#\big\{s\in [0,t\wedge \rT_\epsilon]:(s,\tun_\scK^{[n]}) \in {\cal N}_{\nu}^{[n]}\big\}\, \Big|\, \big(\rZ_s,s\in[0,t\wedge \rT_\epsilon]\big)\right] = \sum_{s \leq t\wedge \rT_\epsilon}\Big(\frac{\Delta \rZ_s}{\rZ_s}\Big)^{k}\Big(1-\frac{\Delta \rZ_s}{\rZ_s}\Big)^{n-k}
\end{equation*}
Furthermore when $k=2$, we deduce from the definition of ${\cal N}_{\sigma}$ that the counting process
\begin{equation*}
\#\big\{s\in [0,t]:(s,\tun_{\scK}^{[n]}) \in {\cal N}_{\sigma}^{[n]}\big\},\; t \in [0,\rT)
\end{equation*}
is a doubly stochastic Poisson process with intensity $\mathbf{1}_{\{t<\rT\}}\frac{\sigma^2}{\rZ_{t}}dt$. Therefore, for all $t \geq 0$
\begin{equation*}
\mathbb{E}\left[\#\big\{s\in [0,t\wedge \rT_\epsilon]:(s,\tun_{\scK}^{[n]}) \in {\cal N}_{\sigma}^{[n]}\big\}\, \Big|\,\big(\rZ_s,s\in[0,t\wedge \rT_\epsilon]\big)\right] = \int_0^{t\wedge \rT_\epsilon}\frac{\sigma^2}{\rZ_{s}}\,ds
\end{equation*}
Notice that the r.h.s. is finite a.s. Putting together the preceding results, we get that
\begin{equation*}
\mathbb{E}\Big[\rL_{t\wedge \rT_\epsilon}(n,\scK)\, \big|\, \big(\rZ_s,s\in[0,t\wedge \rT_\epsilon]\big)\Big] = \sum_{s \leq t\wedge \rT_\epsilon}\Big(\frac{\Delta \rZ_s}{\rZ_s}\Big)^{k}\Big(1-\frac{\Delta \rZ_s}{\rZ_s}\Big)^{n-k} + \mathbf{1}_{\{k=2\}}\int_0^{t\wedge \rT_\epsilon}\frac{\sigma^2}{\rZ_{s}}\,ds
\end{equation*}
Using Lemma \ref{LemmaIntegrabilityCSBP} and Equation (\ref{EqInequalityJumps}), we deduce that the r.h.s. of the preceding equation is integrable for all $t\geq 0$. Therefore
\begin{equation*}
\mathbb{E}\big[M^{(1)}_{t\wedge \rT_\epsilon}\big] = 0
\end{equation*}
Note that the integrability is indeed locally uniform since we deal with non-decreasing processes. In addition, we have for all $0 \leq r \leq t$,
\begin{eqnarray*}
\mathbb{E}\big[M^{(1)}_{t\wedge \rT_\epsilon}\, \big|\, {\cal F}_r\big] &=& M^{(1)}_{r\wedge \rT_\epsilon}+ \mathbb{E}\Big[\rL_{r\wedge \rT_\epsilon,t\wedge \rT_\epsilon}(n,\scK)\\
&-&\sum_{s \in (r\wedge \rT_\epsilon, t\wedge \rT_\epsilon]}\Big(\frac{\Delta \rZ_s}{\rZ_s}\Big)^{k}\Big(1-\frac{\Delta \rZ_s}{\rZ_s}\Big)^{n-k} - \mathbf{1}_{\{k=2\}}\int_{r\wedge \rT_\epsilon}^{t\wedge \rT_\epsilon}\frac{\sigma^2}{\rZ_{s}}\,ds\, \big|\,{\cal F}_{r}\Big]
\end{eqnarray*}
By applying the strong Markov property at time $r\wedge \rT_\epsilon$ to the process $\rZ$, one easily gets that the second term in the r.h.s.~is zero a.s.~using the preceding arguments. Therefore, we have proven that $(M^{(1)}_{t\wedge \rT_\epsilon},t \in [0,\rT))$ is a locally uniformly integrable martingale. Since $\rT_\epsilon \uparrow \rT$ a.s., it implies that $M^{(1)}$ is a local martingale on $[0,\rT)$.\\

We turn our attention to $M^{(2)}$. It is well-known that the dual predictable compensator of the random measure
\begin{equation}\label{EqJumpMeasure}
\sum_{\{t\geq 0: \Delta \rZ_t > 0\}}\delta_{(t,\frac{\Delta \rZ_t}{\rZ_t})}
\end{equation}
is the random measure $\mathbf{1}_{\{t<T\}}\rZ_{t-}dt\otimes\nu\circ\phi_{\rZ_{t-}}^{-1}(dx)$ on $\mathbb{R}_+\times[0,1]$. Thus Th.II.1.8 in~\cite{JacodShiryaev} ensures that $M^{(2)}$ is a local martingale on $[0,\rT)$. Indeed, it suffices to take $W(\omega,t,x):=\mathbf{1}_{\{t < \rT\}}x^k(1-x)^{n-k}$ and to apply the theorem to the random measure of Equation (\ref{EqJumpMeasure}).\\
We have proved that both $M^{(1)}$ and $M^{(2)}$ are local martingales w.r.t. ${\cal F}$, this implies that the process
\begin{equation*}
\Big(\rL_t(n,\scK)-\int_{0}^{t}\lambda_{n,k}(\rZ_{s-},\Psi)ds,\, t \in [0,\rT)\Big)
\end{equation*}
is a local martingale on $[0,\rT)$ w.r.t. $\cal F$.\\
Finally, consider the vector formed by the $d_n$ counting processes. Since we have identified for each of them their compensator in a same filtration $\cal F$, we have identified the compensator of the vector. The proposition is proved.\cqfd
\end{proof}

\section{The Eve property}
Throughout this section, $\rmm$ designates a c\`adl\`ag $\Psi$-MVBP started from the Lebesgue measure on $[0,1]$, $\rZ$ denotes its total-mass process and $\rT$ its lifetime. In the first subsection, we define the Eve property and prove Theorem \ref{TheoremEquivalence}. In the second subsection, we identify a complete sequence of Eves and prove Theorem \ref{ThEves}. Some properties of the Eves are given in the third subsection.
\subsection{Definition}
Recall the definition given in the introduction.\vspace{8pt}\\
\textbf{Definition \ref{DefEveProperty}} \textit{
We say that the branching mechanism $\Psi$ satisfies the Eve property if and only if there exists a random variable $\rme$ in $[0,1]$ such that
\begin{equation}
\frac{\rmm_{t}(dx)}{\rmm_{t}([0,1])} \underset{t \uparrow \rT}{\longrightarrow} \delta_{\rme}(dx)\mbox{ a.s.}
\end{equation}
in the sense of weak convergence of probability measures. The r.v. $\rme$ is called the primitive Eve of the population.}

\begin{lemma}
Suppose that the Eve property is verified. Then $\rme$ is uniform$[0,1]$.
\end{lemma}
\begin{proof}
Let us consider a bijection $f$ from $[0,1]$ to $[0,1]$ that preserves the Lebesgue measure. For all $t \geq 0$, we denote by $\rmm_t\circ f^{-1}$ the pushforward measure of $\rmm_t$ by the function $f$. The process $(\rmm_t\circ f^{-1}, t \in [0,\rT))$ is still a $\Psi$-MVBP whose lifetime is $\rT$. Thus there exists a r.v. $\rme' \in [0,1]$ such that
\begin{equation}
\frac{\rmm_{t}\circ f^{-1}(dx)}{\rmm_{t}([0,1])} \underset{t \rightarrow \rT}{\longrightarrow} \delta_{\rme'}(dx)\mbox{ a.s.}
\end{equation}
Moreover, it is immediate to check that $\rme' := f(\rme)$ and that $\rme'$ and $\rme$ have the same distribution. We deduce that $\rme$ is a r.v. on $[0,1]$ whose distribution is invariant under bijections that preserve the Lebesgue measure. Hence it is a uniform$[0,1]$ r.v.\cqfd
\end{proof}
The following proposition specifies an important case where the Eve property is fulfilled.
\begin{proposition}\label{PropEveTFinite}
If $\rT < \infty$ a.s. then $\Psi$ satisfies the Eve property.
\end{proposition}
\begin{proof}
Suppose that $\rT < \infty$ a.s. The branching property fulfilled by the process $(\rmm_t, t \geq 0)$ ensures that
\begin{equation*}
\big(\rmm_t([0,2^{-n})),t \in [0,\rT)\big);\,\big(\rmm_t([2^{-n},2\times 2^{-n})),t \in [0,\rT)\big);\ldots;\,\big(\rmm_t([1-2^{-n},1]),t \in [0,\rT)\big)
\end{equation*}
are $2^n$ i.i.d. $\Psi$-CSBP started from $2^{-n}$ and stopped at the infimum of their lifetimes. Since the lifetimes of these CSBP are independent and finite a.s., we deduce from Lemma \ref{LemmaDistributionAbsContinuous} that they are distinct a.s. and that $\rT$ is either the first explosion time or the last extinction time of the preceding collection. Therefore, for all $i\in [2^n]$,
\begin{equation*}
\lim\limits_{t \rightarrow \rT}\frac{\rmm_t([(i-1)2^{-n},i2^{-n}))}{\rmm_t([0,1])} \in \{0,1\}\mbox{ a.s.}
\end{equation*}
This implies that there exists a unique (random) integer $u_n \in [2^n]$ such that
\begin{equation*}
\lim\limits_{t \rightarrow \rT}\frac{\rmm_t([(u_n-1)2^{-n},u_n2^{-n}))}{\rmm_t([0,1])} = 1\mbox{ a.s.}
\end{equation*}
This holds for all $n\in \mathbb{N}$ and obviously $[(u_n-1)2^{-n},u_n2^{-n}) \supset [(u_{n+1}-1)2^{-(n+1)},u_{n+1}2^{-(n+1)})$. We can therefore introduce the following random variable
\begin{equation*}
\rme := \inf\limits_{n\in\bbN}u_n2^{-n}
\end{equation*}
We have proved that
\begin{equation}
\frac{\rmm_{t}(dx)}{\rmm_{t}([0,1])} \underset{t \rightarrow \rT}{\longrightarrow} \delta_{\rme}(dx)\mbox{ a.s.}
\end{equation}
in the sense of weak convergence of probability measures.\cqfd
\end{proof}
\begin{remark}
A complete classification of the asymptotic behaviour of $\frac{\rmm_t(.)}{\rmm_t([0,1])}$ will be established in a forthcoming work~\cite{DuquesneLabbe13}. In particular it will be shown that whenever the CSBP is supercritical, the Eve property is fulfilled if and only if the mean is infinite: an intuitive argument for this result is that two independent copies of a same CSBP have comparable asymptotic sizes iff the mean is finite.
\end{remark}
We now present a result that relates the Eve property with the behaviour of the $\Psi$ flow of partitions at the end of its lifetime $\rT$. In addition, this result provides a necessary and sufficient condition on $\rZ$ for the Eve property to hold.\vspace{8pt}\\
\textbf{Theorem \ref{TheoremEquivalence}}\textit{
There exists an exchangeable partition $\hat{\Pi}_{\rT}$ such that $\hat{\Pi}_{t} \rightarrow \hat{\Pi}_{\rT}$ almost surely as $t \uparrow \rT$. Moreover, these three assumptions are equivalent\begin{enumerate}[i)]
\item\label{EquivalencePartitions3} $\Psi$ satisfies the Eve property.
\item\label{EquivalencePartitions2} $\hat{\Pi}_{\rT} = \tun_{[\infty]}$ a.s.
\item\label{EquivalencePartitions1} $\displaystyle\sum_{\{s < \rT : \Delta \rZ_s > 0\}}\Big(\frac{\Delta \rZ_s}{\rZ_s}\Big)^2+\int_{0}^{\rT}\frac{\sigma^2}{\rZ_s}ds = \infty$ a.s.
\end{enumerate}}
\begin{remark}
This theorem should be compared with Theorem 6.1 in~\cite{DK99} where a similar condition on the total-mass process is given but for a much larger class of measure-valued processes. However their result is proved only when $\rT < \infty$, which in our particular case of branching processes, is a trivial case as we already know from Proposition \ref{PropEveTFinite} that the Eve property is fulfilled.
\end{remark}
\begin{proof}
To prove the asserted convergence, it suffices to show that for each $n\in\mathbb{N}$, the restriction $\hat{\Pi}^{[n]}_t$ of $\hat{\Pi}_t$ to $\mathscr{P}_n$ admits a limit when $t \uparrow \rT$ almost surely. We fix $n\in\mathbb{N}$ until the end of the proof.\vspace{2pt}\\
\textit{Step 1.} Remark that conditional on $\int_{0}^{\rT}\frac{\sigma^2}{\rZ_s}ds$, the r.v. $\#\{{\cal N}_{\sigma|[0,\rT)\times\mathscr{P}^*_n}\}$ has a Poisson distribution with parameter $\binom{n}{2}\int_{0}^{\rT}\frac{\sigma^2}{\rZ_s}ds$. Thus we have almost surely
\begin{equation}\label{EqEquivalence1}
\mathbb{P}\Big(\,\#\{{\cal N}_{\sigma|[0,\rT)\times\mathscr{P}^*_n}\}=0\, \big| \,\int_{0}^{\rT}\frac{\sigma^2}{\rZ_s}ds\,\Big) = \exp\Big(-\binom{n}{2}\int_{0}^{\rT}\frac{\sigma^2}{\rZ_s}\,ds\Big)
\end{equation}
In addition, thanks to Borel-Cantelli lemma we notice that
\begin{eqnarray}\label{EqEquivalence2}
&&\mathbb{P}\Big(\#\big\{{\cal N}_{\nu|[0,\rT)\times\mathscr{P}^*_n}\big\}=0 \, \big|\,\displaystyle\sum_{s < \rT}\Big(\frac{\Delta \rZ_s}{\rZ_s}\Big)^2= \infty\Big)\nonumber \leq\\
&&1-\mathbb{P}\Big(\#\big\{{\cal N}_{\nu|[0,\rT)\times\mathscr{P}^*_n}\big\}=\infty\, \big|\, \displaystyle\sum_{s < \rT}\Big(\frac{\Delta \rZ_s}{\rZ_s}\Big)^2 = \infty\Big) = 0
\end{eqnarray}
\textit{Step 2.} Introduce $t_{i}:= \inf\{t \geq 0: \hat{\Pi}^{[i]}_{t} = \tun_{[i]}\}$ for all $i\in\mathbb{N}$, we first prove that $t_n < \rT$ conditional on $\{\sum_{s < \rT}(\frac{\Delta \rZ_s}{\rZ_s})^2+\int_{0}^{\rT}\frac{\sigma^2}{\rZ_s}ds = \infty\}$, thus it will imply that $\hat{\Pi}_{t}^{[n]} \rightarrow \tun_{[n]}$ as $t \uparrow \rT$ on the same event, and also the implication \ref{EquivalencePartitions1}) $\Rightarrow$ \ref{EquivalencePartitions2}). We proceed via a recursion. At rank $i=2$, we use Equations (\ref{EqEquivalence1}) and (\ref{EqEquivalence2}) to obtain
\begin{equation*}
\mathbb{P}\Big(\#\{{\cal P}_{|[0,\rT)\times\mathscr{P}_2^*}\}=0\,\big|\,\sum_{s < \rT}\Big(\frac{\Delta \rZ_s}{\rZ_s}\Big)^2+\int_{0}^{\rT}\frac{\sigma^2}{\rZ_s}ds = \infty\Big) = 0
\end{equation*}
Hence $t_{2} < \rT$ a.s. Suppose that $t_{i-1} < \rT$ almost surely for a given integer $i \geq 3$, then we have $\hat{\Pi}^{[i-1]}_{t_{i-1}} = \tun_{[i-1]}$ a.s. Thus either $t_{i} = t_{i-1}$ and the recursion is complete, or $\hat{\Pi}^{[i]}_{t_{i-1}} = \{1,\ldots,i-1\},\{i\}$. In the latter case, we need to prove that on $[t_{i-1},\rT)$ there will be a reproduction event involving an integer in $[i-1]$ and the integer $i$. We denote by $A_i$ the subset of $\mathscr{P}_{i}^*$ whose elements are partitions with a non-singleton block containing an integer lower than $i-1$ and the integer $i$. Remark that on the event $\{\sum_{s < \rT}(\frac{\Delta \rZ_s}{\rZ_s})^2+\int_{0}^{\rT}\frac{\sigma^2}{\rZ_s}ds = \infty\}$
\begin{equation*}
\displaystyle\sum_{t_{i-1} < s < \rT}\Big(\frac{\Delta \rZ_s}{\rZ_s}\Big)^2+\int_{t_{i-1}}^{\rT}\frac{\sigma^2}{\rZ_s}\,ds = \infty\;\;\mbox{a.s.}
\end{equation*}
so that almost surely
\begin{eqnarray*}
\mathbb{P}\Big(\#\{{\cal N}_{\sigma|(t_{i-1},\rT)\times A_i}\}=0\, \big|\, \int_{0}^{\rT}\frac{\sigma^2}{\rZ_s}\,ds = \infty\Big) = 0 \\
\mathbb{P}\Big(\#\{{\cal N}_{\nu|(t_{i-1},\rT)\times A_i}\} = 0\, \big|\, \displaystyle\sum_{s < \rT}\Big(\frac{\Delta \rZ_s}{\rZ_s}\Big)^2 = \infty\Big) = 0
\end{eqnarray*}
which in turn ensures that $t_i < \rT$ a.s. The recursion is complete.\vspace{2pt}\\
\textit{Step 3.} We now prove that conditional on $\{\sum_{s<\rT}\big(\frac{\Delta \rZ_s}{\rZ_s}\big)^2+\int_{0}^{\rT}\frac{\sigma^2}{\rZ_s}ds < \infty\}$, the number of reproduction events $\#\{{\cal P}_{|[0,\rT)\times\mathscr{P}_{n}^*}\}$ is finite. This will imply that $\hat{\Pi}^{[n]}_t$ admits a limit as $t \uparrow \rT$ on the same event.\\
Thanks to the remark preceding (\ref{EqEquivalence1}), we deduce that on $\{\int_{0}^{\rT}\frac{\sigma^2}{\rZ_s}ds < \infty\}$, the r.v. $\#\{{\cal N}_{\sigma|[0,\rT)\times\mathscr{P}^*_n}\}$ is finite. In addition, for each $s \geq 0$ such that $\Delta \rZ_s > 0$, the probability that the restriction of $\varrho_s$ to $\mathscr{P}_n$ differs from $0_{[n]}$ is equal to
\begin{equation*}
\sum_{k=2}^n\binom{n}{k}\Big(\frac{\Delta \rZ_s}{\rZ_s}\Big)^k\Big(1-\frac{\Delta \rZ_s}{\rZ_s}\Big)^{n-k}
\end{equation*}
independently of the other $(\varrho_t)_{t \ne s}$. Since
\begin{equation*}
\displaystyle\sum_{s < \rT}\sum_{k=2}^n\binom{n}{k}\Big(\frac{\Delta \rZ_s}{\rZ_s}\Big)^k\Big(1-\frac{\Delta \rZ_s}{\rZ_s}\Big)^{n-k} \leq \displaystyle\sum_{s < \rT}\sum_{k=2}^n\binom{n}{k}\Big(\frac{\Delta \rZ_s}{\rZ_s}\Big)^2
\end{equation*}
an application of Borel-Cantelli lemma implies that
\begin{equation*}
\mathbb{P}\Big(\#\{{\cal N}_{\nu|[0,\rT)\times\mathscr{P}^*_n}\}<\infty\, \big|\,\displaystyle\sum_{s < \rT}\Big(\frac{\Delta \rZ_s}{\rZ_s}\Big)^2 < \infty\Big) = 1
\end{equation*}
Thus we have proved that $\#\{{\cal P}_{|[0,\rT)\times\mathscr{P}_{n}^*}\}$ is finite on the event $\{\sum_{s < \rT}(\frac{\Delta \rZ_s}{\rZ_s})^2+\int_{0}^{\rT}\frac{\sigma^2}{\rZ_s}ds < \infty\}$.\vspace{2pt}\\
\textit{Step 4.} We now prove that
\begin{equation*}
\mathbb{P}\Big(\,\displaystyle\sum_{s < \rT}\Big(\frac{\Delta \rZ_s}{\rZ_s}\Big)^2+\int_{0}^{\rT}\frac{\sigma^2}{\rZ_s}ds < \infty\Big) > 0 \Rightarrow \mathbb{P}\big(\hat{\Pi}_{\rT} \ne \tun_{[\infty]}\big) > 0
\end{equation*}
this will imply \ref{EquivalencePartitions2}) $\Rightarrow$ \ref{EquivalencePartitions1}).\\
Thanks to Equation (\ref{EqEquivalence1}), we get
\begin{equation*}
\mathbb{P}\Big(\#\{{\cal N}_{\sigma|[0,\rT)\times\mathscr{P}^*_2}\}=0\, \big|\,\int_{0}^{\rT}\frac{\sigma^2}{\rZ_s}\,ds < \infty\Big) > 0
\end{equation*}
Also, note that
\begin{equation*}
\mathbb{P}\Big(\#\{{\cal N}_{\nu|[0,\rT)\times\mathscr{P}^*_2}\}=0\,\big|\,\big\{(\frac{\Delta \rZ_s}{\rZ_s})^2;s < \rT\big\}\Big) = \prod_{s < \rT}\Big(1-\big(\frac{\Delta \rZ_s}{\rZ_s}\big)^2\Big)
\end{equation*}
One can readily prove that the r.h.s. is strictly positive on the event $\{\sum_{s < \rT}(\frac{\Delta \rZ_s}{\rZ_s})^2<\infty\}$. Therefore we have proven that
\begin{equation*}
\mathbb{P}\Big(\hat{\Pi}^{[2]}_\rT = 0_{[2]}\,\big|\,\sum_{s < \rT}\Big(\frac{\Delta \rZ_s}{\rZ_s}\Big)^2+\int_{0}^{\rT}\frac{\sigma^2}{\rZ_s}\,ds < \infty\Big) > 0
\end{equation*}
This inequality ensures the implication \ref{EquivalencePartitions2}) $\Rightarrow$ \ref{EquivalencePartitions1}).\vspace{2pt}\\
\textit{Step 5.} We turn our attention to the proof of \ref{EquivalencePartitions2}) $\Leftrightarrow$ \ref{EquivalencePartitions3}). Consider a sequence $(\upxi_0(i))_{i\geq 1}$ of i.i.d. uniform$[0,1]$ r.v. and let $(\upxi_{t}(i),t \geq 0)_{i\geq 1} := \mathscr{L}_0(\hat{\Pi},(\upxi_0(i))_{i\geq 1})$ be the lookdown process defined from this last sequence and the flow of partitions $\hat{\Pi}$. We know that $(\rZ_t\cdot\Xi_{t}(.),t\geq 0)$ is a $\Psi$-MVBP, where $(\Xi_{t},t\geq 0) := \mathscr{E}_0(\hat{\Pi},(\upxi_0(i))_{i\geq 1})$. Moreover, a.s. for all $t \in [0,\rT)$
\begin{eqnarray*}
\Xi_t\big(\{\upxi_0(1)\}\big) &=& \lim\limits_{n\rightarrow\infty}\frac{1}{n}\sum_{i=0}^{n}\mathbf{1}_{\{\upxi_t(i)=\upxi_0(1)\}}\\
&=& |\hat{\Pi}_t(1)|
\end{eqnarray*}
It is intuitively easy to see that the Eve property is equivalent with the almost sure convergence
\begin{equation*}
\Xi_t\big(\{\upxi_0(1)\}\big) \underset{t \uparrow \rT}{\longrightarrow} 1
\end{equation*}
Roughly speaking, the primitive Eve is necessarily the type $\upxi_0(1)$ in the sequence of initial types of the lookdown representation. For a rigorous proof of this result, see Proposition \ref{PropUnicityInitialTypes}. Then, it is sufficient to show the following equivalence
\begin{equation*}
\Xi_t\big(\{\upxi_0(1)\}\big) \underset{t \uparrow \rT}{\longrightarrow} 1\mbox{ a.s.} \Longleftrightarrow \hat{\Pi}_t \underset{t \uparrow \rT}{\longrightarrow} \tun_{[\infty]}\mbox{ a.s.} 
\end{equation*}
Since $\hat{\Pi}_t$ is, conditionally on $\{t < \rT\}$, an exchangeable random partition, we deduce that for all $n \geq 1$
\begin{equation*}
\mathbb{P}\Big(\hat{\Pi}_t^{[n]}=\tun_{[n]}\, \big| \,|\hat{\Pi}_t(1)|\Big)=|\hat{\Pi}_t(1)|^{n-1}
\end{equation*}
Thus
\begin{equation*}
\Xi_t\big(\{\upxi_0(1)\}\big) \underset{t \uparrow \rT}{\longrightarrow} 1\mbox{ a.s.}\, \Longleftrightarrow\, |\hat{\Pi}_t(1)| \underset{t \uparrow \rT}{\longrightarrow} 1\mbox{ a.s.} \,\Longleftrightarrow\, \forall n\geq 1,\, \hat{\Pi}_\rT^{[n]}=\tun_{[n]}\mbox{ a.s.} 
\end{equation*}
The proof is complete.\cqfd
\end{proof}
Thanks to this theorem, we can set $\hat{\Pi}_t := \hat{\Pi}_\rT$ for all $t \geq \rT$.

\subsection{An ordering of the ancestors}\label{SubsectionAncestors}
Consider the following definition of an ancestor.
\begin{definition}
Fix a point $x \in [0,1]$. If there exists a time $t > 0$ such that $\rmm_t(\{x\}) > 0$, then we say that $x$ is an ancestor of $\rmm$ and $\rmm_t(\{x\})$ is called its progeny at time $t$. We say that an ancestor $x$ becomes extinct at time $d$, if $d := \sup\{t > 0:\rmm_t(\{x\})>0\}$ is finite. If so, $d$ is called its extinction time.
\end{definition}
Thanks to the lookdown representation the set of ancestors is countable almost surely. Indeed, an ancestor is a point of the atomic support of the MVBP at a given time. Since at any time, the atomic support is included in the set of initial types (of the lookdown representation) and since this last set is countable, the result follows.
\begin{remark}\label{RemarkL\'evyTree}
In the infinite variation case, one can identify an ancestor and its progeny with a L\'evy tree among the L\'evy forest that represents the genealogy of a CSBP. For further details on L\'evy trees see~\cite{DLG02,DuquesneWinkel07,LeGallLeJan98}.
\end{remark}
The progeny $\rmm_t(\{x\})$ of an ancestor $x$ has the same possible long-term behaviours as a $\Psi$-CSBP (these behaviours have been recalled in Subsection \ref{SubsectionCSBP}). We thus propose a classification of the $\Psi$-MVBPs according to these possible behaviours; for the moment we do not require the Eve property to be verified. Recall $d_t, w_t$ from Equation (\ref{EquationLaplaceExponent}).\\\\
\textbf{Classification of the behaviours}\begin{itemize}
\item \Extinction. The total-mass process $\rZ$ reaches $0$ in finite time. All the ancestors become extinct in finite time but no two of them simultaneously. At any time $t \in (0,\rT)$, $\rmm_t$ has finitely many atoms, hence the number of ancestors that have not become extinct is finite, and $\rmm_t$ has no continuous part, that is, $d_t = 0$.
\item \Explosion. The total-mass process $\rZ$ reaches $\infty$ in finite time. All the ancestors, except the primitive Eve, have finite progenies at time $\rT$.
\item \InfNo. $\rT=\infty$ and $\Psi$ is either negative or has a second root $q \in [0,\infty)$ and verifies $\int^{\infty}\frac{du}{\Psi(u)} = \infty$. Then no ancestor becomes extinct in finite time, and their progenies reach, in infinite time, $0$ or $\infty$.
\item \InfPos. $\rT=\infty$, $\Psi$ has a second positive root $q > 0$ and $\int^{\infty}\frac{du}{\Psi(u)} < \infty$. The set of ancestors can be subdivided into those, infinitely many, which become extinct in finite time (no two of them simultaneously) and those, finitely many, whose progenies reach $\infty$ in infinite time.
\end{itemize}
Additionally in the last two cases, the number of ancestors whose progenies reach $\infty$ is Poisson with parameter $q$ under the condition that $\Psi$ is conservative.
\begin{remark}
A $\Psi$-MVBP enjoys at most two distinct behaviours: one on the event $\{\rZ_\rT=0\}$ and another on the event $\{\rZ_\rT=\infty\}$.
\end{remark}
\begin{example}
Let us give some examples that illustrate the previous cases
\begin{itemize}
\item $\Psi(u)=u^2$, the $\Psi$-CSBP reaches $0$ in finite time almost surely and so we are in the \Extinction\ case almost surely.
\item $\Psi(u)=-\sqrt{u}$, the $\Psi$-CSBP reaches $+\infty$ in finite time almost surely and so we are in the \Explosion\ case almost surely.
\item $\Psi(u)=u\ln(u)$, this is called the Neveu CSBP: it has an infinite lifetime almost surely and so we are in the \InfNo\ case almost surely. 
\item $\Psi(u)=u\ln(u)+u^2$, this CSBP reaches either $0$ in finite time or $\infty$ in infinite time. On the event $\{\rZ_\rT = \infty\}$ we are in the \InfPos\ case, while on the event $\{\rZ_\rT = 0\}$ we are in the \Extinction\ case.
\end{itemize}
\end{example}
\begin{proof}(\textbf{Classification of the behaviours})
It is plain that these four cases cover all the possible combinations of branching mechanisms and asymptotic behaviours of the total-mass processes.\\
\Extinction\ case. If $d_t > 0$ or $w_t$ is an infinite measure then we have $\mathbb{P}(\rZ_t = 0)= 0$, therefore necessarily $w_t$ is a finite measure and $d_t = 0$. Since each atom of $\rmm_t$ is associated with an ancestor with a positive progeny at time $t$, we deduce that at any time $t> 0$ only finitely many ancestors have not become extinct. Now condition on $\{t < \rT\}$ and consider two ancestors $x_1$ and $x_2$ in $[0,1]$ not yet extinct at time $t$. Their progenies after time $t$ are given by two independent $\Psi$-CSBP $(\rmm_{t+s}(\{x_1\}),s\geq 0)$ and $(\rmm_{t+s}(\{x_2\}),s\geq 0)$. The extinction times of these two ancestors are then distinct a.s. thanks to Lemma \ref{LemmaDistributionAbsContinuous}.\\
\Explosion\ case. Since two independent $\Psi$-CSBP cannot explode at the same finite time thanks to Lemma \ref{LemmaDistributionAbsContinuous}, we deduce that only one ancestor has an infinite progeny at time $\rT$.\\
\emph{Infinite lifetime} cases. The Poisson distribution of the statement can be derived from Lemma 2 in~\cite{ProlificIndividuals}. Let $x\in[0,1]$ be an ancestor. Then, there exists $t > 0$ s.t. $\rmm_t(\{x\}) > 0$. The process $(\rmm_{t+s}(\{x\}),s\geq 0)$ is a $\Psi$-CSBP started from $\rmm_t(\{x\})$, and so, either it reaches $\infty$ in infinite time with probability $1-e^{-\rmm_t(\{x\})q}$ or it reaches $0$ with the complementary probability. In the latter case, it reaches $0$ in finite time if and only if $\int^{\infty}\frac{du}{\Psi(u)} < \infty$.\\
Now consider the case $\int^{\infty}\frac{du}{\Psi(u)} < \infty$. Remark that in that case $d_t = 0$ for all $t > 0$ and that there is no simultaneous extinction (same proof as above). Let us prove that infinitely many ancestors become extinct in finite time. Consider the lookdown representation of the $\Psi$-MVBP: we stress that the set of initial types is exactly equal to the set of ancestors. We have already proved one inclusion at the beginning of this subsection: each ancestor is an initial type. The converse is obtained as follows. Observe first that $\Psi$ is necessarily the Laplace exponent of a L\'evy process with infinite variation paths since otherwise $\int^{\infty}\frac{du}{\Psi(u)} < \infty$ would not hold. Therefore Theorem \ref{ThDust} ensures that the partitions $\hat{\Pi}_{0,t}, t > 0$ have no singleton: each block has a strictly positive asymptotic frequency and therefore each initial types of the lookdown representation has a strictly positive frequency. This implies that each initial type is necessarily an ancestor. As the initial types are infinitely many, so are the ancestors: a Poisson number of them have a progeny that reaches $\infty$ in infinite time, hence infinitely many become extinct in finite time.\cqfd
\end{proof}
\begin{remark}
One should compare the \Extinction\ case with the behaviour of the $\Lambda$ Fleming-Viot processes that come down from infinity. But in that setting, the question of simultaneous loss of ancestral types remains open, see Section \ref{SectionPathwise} or~\cite{Labbe11} for further details.
\end{remark}
\begin{theorem*}\textup{\textbf{\ref{ThEves}}}
Assume that $\rZ$ does not reach $\infty$ in finite time. If the Eve property holds then one can order the ancestors by persistence/predominance. We denote this ordering $(\rme^{i})_{i\geq 1}$ and call these points the Eves. In particular, $\rme^{1}$ is the primitive Eve.
\end{theorem*}
The persistence of an ancestor refers to the extinction time of its progeny (when it reaches $0$ in finite time) while the predominance denotes the asymptotic behaviour of its progeny (when it does not become extinct in finite time). The proof of this theorem is thus split into the \Extinction\ and the \emph{Infinite lifetime} cases. Note that we have excluded the case where the $\Psi$-CSBP is non-conservative for a reason given in Remark \ref{RemarkNonConservative}.

\subsubsection{Extinction case}
One can enumerate the ancestors of $(\rmm_t,t \in [0,\rT))$, say $(\rme^{i})_{i\geq 1}$, in the decreasing order of their extinction times $(d^i)_{i\geq 1}$, that is, $\rT = d^1 > d^2 > d^3 \ldots > 0$. In particular, $\rme^{1}$ is the primitive Eve.

\subsubsection{Infinite lifetime case}
We let $\rme^1$ be the primitive Eve of the population: necessarily it does not become extinct in finite time. Then we use the following result.
\begin{lemma}\label{LemmaRecursiveEves}
In the \InfNo\ case: there exists a sequence $(\rme^i)_{i\geq 2}$ such that for all $i\geq 2$
\begin{equation*}
\lim\limits_{t\rightarrow\infty}\frac{\rmm_t(\{\rme^i\})}{\rmm_t([0,1]\backslash\{\rme^1,\ldots,\rme^{i-1}\})}=1
\end{equation*}
In the \InfPos\ case: let $K$ be the random number of ancestors which never become extinct. There exists a sequence $(\rme^i)_{i\geq 2}$ such that for all $i\in \{2,\ldots,K\}$
\begin{equation*}
\lim\limits_{t\rightarrow\infty}\frac{\rmm_t(\{\rme^i\})}{\rmm_t([0,1]\backslash\{\rme^1,\ldots,\rme^{i-1}\})}=1
\end{equation*}
and $(\rme^i)_{i> K}$ are the remaining ancestors in the decreasing order of their extinction times $(d^i)_{i>K}$, that is, $\rT=\infty > d^{K+1} > d^{K+2} > d^{K+3} \ldots > 0$. 
\end{lemma}
\begin{proof}
We focus on the \InfNo\ case, the other case is then a mixture of the latter with the \Extinction\ case. For every $n\in\mathbb{N}$, we subdivide $[0,1]$ into
\begin{equation*}
[0,2^{-n}),[2^{-n},2\times 2^{-n}),\ldots,[1-2^{-n},1]
\end{equation*}
Since these intervals are disjoint, the restrictions of $\rmm_t$ to each of them are independent. Therefore, we define for each $i\in[2^n]$, the random point $e(i,n)$ as the Eve of the process
\begin{equation*}
\big(\rmm_t(.\cap[(i-1)2^{-n},i2^{-n})),t\geq 0\big)
\end{equation*}
(note that for $i=2^n$ we take $[1-2^{-n},1]$). In addition, one can define an ordering of the collection $(e(i,n))_{i\in[2^n]}$ according to the asymptotic behaviours of their progenies. More precisely, for two integers $i\ne j\in[2^n]$, thanks to the classification of the behaviours and the Eve property, we have
\begin{equation*}
\lim\limits_{t\rightarrow\infty}\frac{\rmm_t(\{e(i,n)\})}{\rmm_t(\{e(j,n)\})}\in\{0,\infty\}
\end{equation*}
Thus, there exist two r.v. $i_1^n\ne i_2^n \in [2^n]$ such that for all $i \in [2^n]$
\begin{eqnarray*}
\lim\limits_{t\rightarrow\infty}\frac{\rmm_t\big(\{e(i_1^n,n)\}\big)}{\rmm_t\big(\{e(i,n)\}\big)}\!\!&=&\!\!\infty\;\;\mbox{if}\;i\ne i_1^n\\
\lim\limits_{t\rightarrow\infty}\frac{\rmm_t\big(\{e(i_2^n,n)\}\big)}{\rmm_t\big(\{e(i,n)\}\big)}\!\!&=&\!\!\infty\;\;\mbox{if}\;i\ne i_1^n,i_2^n
\end{eqnarray*}
We set $e^1(n) := e(i_1^n,n)$ and $e^2(n) := e(i_2^n,n)$. We claim that almost surely the sequences $(e^1(n))_{n\geq 1}$ and $(e^2(n))_{n\geq 1}$ are eventually constant. This is clear for $(e^1(n))_{n\geq 1}$ since for each $n\geq 1$, $e^1(n) = \rme^{1}$, which is the primitive Eve of the entire population $[0,1]$. We turn our attention to the sequence $(e^2(n))_{n\geq 1}$, in that case the claim is not so clear. Roughly speaking, the wild behaviour this sequence could have is the following: infinitely often, the second Eve $e^{2}(n)$ is "hidden" in the interval $[(i_1^{n-1}-1)2^{-(n-1)},i_1^{n-1}2^{-(n-1)})$ containing the first Eve $\rme^{1}$ at rank $n-1$, but we will see that it cannot occur. Suppose that the claim does not hold. Thus there exists an event $E$ of positive probability on which there exists a sequence $(n_k)_{k\geq 1}$ of integers such that $e^2(n_k-1)\ne e^2(n_k)$ for every $k\geq 1$. From the consistency of the restrictions of the MVBP $\rmm$ to the subintervals defined at ranks $n_k-1$ and $n_k$, we deduce that $e^{2}(n_k)$ is in $[(i_1^{n_k-1}-1)2^{-(n_k-1)},i_1^{n_k-1}2^{-(n_k-1)})$, that is, the same interval as $\rme^1$ at rank $n_k-1$. Hence on the event $E$
\begin{equation}\label{EqExistenceEves}
|\rme^1-e^2(n)| \rightarrow 0 \mbox{ as }n\uparrow\infty
\end{equation}
We now exhibit a contradiction. By the exchangeability of the increments of the MVBP $\rmm$, we know that $(i_1^n,i_2^n)$ is distributed uniformly among the pairs of integers in $[2^n]$. Therefore, one easily deduces that for all $n \geq p$
\begin{equation*}
\mathbb{P}(|\rme^1-e^2(n)|\leq 2^{-p}) \leq \mathbb{P}(|i_1^n-i_2^n|\leq 2^{n-p}+1)\leq 2^{2-p}
\end{equation*}
This implies that the convergence of Equation (\ref{EqExistenceEves}) holds with probability $0$, and $E$ cannot have positive probability. Therefore our initial claim is proved and we can define $\rme^2 := \lim\limits_{n\rightarrow\infty}e^2(n)$.\\
The property is proved for the first two ancestors $\rme^1$ and $\rme^2$. The general case is obtained similarly.\cqfd
\end{proof}

\begin{remark}\label{RemarkNonConservative}
In the case where $\rZ$ reaches $\infty$ in finite time (non-conservative case), we cannot obtain a relevant ordering. Indeed, in that case all the progenies $\rmm_\rT(\{x\})$ of the ancestors $x \ne \rme$ are finite at time $\rT$. Therefore, no natural order appears in that setting.
\end{remark}

\subsection{The Eves and the lookdown representation}\label{SubsectionEvesLD}
Let us motivate the previous ordering by presenting a striking connection with the lookdown representation. The following proposition implies that, if the process of limiting empirical measures of a lookdown process is equal to a given $\Psi$-MVBP, then the initial types are necessarily the sequence of Eves of the $\Psi$-MVBP. We denote by $\rmr_t$ the probability measure obtained by rescaling $\rmm_t$ by its total-mass $\rZ_t$.
\begin{proposition}\label{PropUnicityInitialTypes}
Assume that the Eve property is fulfilled and that the $\Psi$-CSBP does not explode in finite time. Consider a $\Psi$ flow of partitions $(\hat{\Pi}_{s,t},0 \leq s \leq t < \rT)$ defined from the $\Psi$-CSBP $\rZ$ and a sequence $(\upxi_{0}(i))_{i\geq 1}$ of r.v. taking distinct values in $[0,1]$. Let $(\Xi_{t},t \in [0,\rT)) := \mathscr{E}_0(\hat{\Pi},(\upxi_{0}(i))_{i\geq 1})$ be the limiting empirical measures of the lookdown process defined from these objects. If $(\Xi_t,t \in [0,\rT)) = (\rmr_t,t \in [0,\rT))$ a.s., then $(\upxi_{0}(i))_{i\geq 1} = (\rme^i)_{i\geq 1}$ a.s.
\end{proposition}
\begin{proof}
We prove the proposition in the \Extinction\ and \InfNo\ cases, as the \InfPos\ case is a combination of these two cases. Consider the lookdown process
\begin{equation*}
(\upxi_{t}(i),t\in[0,\rT))_{i\geq 1}:=\mathscr{L}_0(\hat{\Pi},(\upxi_{0}(i))_{i\geq 1})
\end{equation*}
Suppose this lookdown process verifies the assumptions of the proposition: there exists an event $\Omega^*$ of probability $1$ on which
\begin{equation}\label{EqIdentityXr}
(\Xi_{t},t \in [0,\rT)) = (\rmr_t,t \in [0,\rT))
\end{equation}
We have to prove that $\upxi_{0}(i) = \rme^{i}$ for all $i\geq 1$ a.s. First we notice that each initial type $\upxi_{0}(i)$ (resp. each ancestor $\rme^{i}$) is associated with a process of frequencies $\big(|\hat{\Pi}_{0,t}(i)|,t\geq 0\big)$ (resp. $\big(\rmr_t(\{\rme^{i}\}),t\geq 0\big)$). In addition, we have
\begin{equation}\label{EqFreqH}
\Xi_t(dx) = \sum_{i\geq 1}|\hat{\Pi}_{0,t}(i)|\delta_{\upxi_{0}(i)}(dx) + \Big(1-\sum_{i\geq 1}|\hat{\Pi}_{0,t}(i)|\Big)dx
\end{equation}
We work on the event $\Omega^*$ throughout this proof.\\
\Extinction\ case. There is no drift part in Equation (\ref{EqFreqH}), and the two sets $\{\upxi_{0}(i);i\geq 1\}$ and $\{\rme^i;i\geq 1\}$ are equal. The initial types $\{\upxi_{0}(i);i\geq 1\}$ of the lookdown process are ordered by decreasing persistence by construction. The Eves of the $\Psi$-MVBP $(\rmm_t,t \in [0,\rT))$ are also ordered by decreasing persistence. Therefore $\upxi_{0}(i) = \rme^{i}$ for all $i\geq 1$.\\
\InfNo\ case. From Equations (\ref{EqIdentityXr}) and (\ref{EqFreqH}), we know that
\begin{equation}\label{EqEqualityHr}
\big(|\hat{\Pi}_{0,t}(i)|\big)^{\downarrow}_{i\geq 1} = \big(\rmr_t(\{\rme^{i}\})\big)^{\downarrow}_{i\geq 1}
\end{equation}
for all $t \in [0,\rT)$. By definition of the ancestors, we know that for every $i\geq 1$
\begin{equation*}
\mathbb{P}\Big[\big(\rmr_t(\{\rme^{j}_{0}\})\big)_{1 \leq j \leq i} = \big(\rmr_t(\{\rme^{j}_{0}\})\big)^{\downarrow}_{1 \leq j \leq i}\Big] \underset{t \rightarrow \infty}{\longrightarrow} 1
\end{equation*}
Thus, using the exchangeability of the partition $\hat{\Pi}_{0,t}$, Equation (\ref{EqEqualityHr}) and the last identity, we deduce that for every $i\geq 1$
\begin{equation*}
\mathbb{P}\Big[\big(|\hat{\Pi}_{0,t}(j)|\big)_{1 \leq j \leq i} = \big(|\hat{\Pi}_{0,t}(j)|\big)^{\downarrow}_{1 \leq j \leq i}\Big] \underset{t\rightarrow\infty}{\longrightarrow} 1
\end{equation*}
which entails that
\begin{equation*}
\mathbb{P}\Big[\big(\upxi_{0}(j)\big)_{1 \leq j \leq i} = (\rme^{j})_{1 \leq j \leq i}\Big] = 1
\end{equation*}
This concludes the proof.\cqfd
\end{proof}
We now determine the distribution of the sequence of Eves $(\rme^i)_{i\geq 1}$.
\begin{proposition}\label{PropositionLawAncestors}
The sequence $(\rme^i)_{i\geq 1}$ is i.i.d. uniform$[0,1]$ and is independent of the sequence of processes $\big(\rmm_t(\{\rme^i\}),t \in [0,\rT)\big)_{i\geq 1}$
\end{proposition}
\begin{proof}
Consider a sequence $(\upxi_{0}(i))_{i\geq 1}$ of i.i.d. uniform$[0,1]$ r.v. and a $\Psi$ flow of partitions $(\hat{\Pi}_{s,t},0 \leq s \leq t < \rT)$ defined from the $\Psi$-CSBP $(\rZ_t,t \in [0,\rT))$. Let $(\Xi_{t},t \in [0,\rT)) := \mathscr{E}_{0}(\hat{\Pi},(\upxi_{0}(i))_{i\geq 1})$ be the limiting empirical measures of the corresponding lookdown process. Denote by $\Phi$ the measurable map that associates to a $\Psi$-MVBP its sequence of Eves. From Proposition \ref{PropUnicityInitialTypes}, we deduce that a.s.
\begin{equation*}
\Phi\big((\rZ_t\cdot \Xi_t(.),t \in [0,\rT))\big)=(\upxi_{0}(i))_{i\geq 1}
\end{equation*}
Since $(\rZ_t\cdot \Xi_t(.),t \in [0,\rT)) \stackrel{(d)}{=} (\rmm_t(.),t\in[0,\rT))$, we obtain thanks to Proposition \ref{PropUnicityInitialTypes} the following identity:
\begin{equation}\label{EquationPhi}
\big((\rmm_t(.),t\in[0,\rT)),(\rme^{i})_{i\geq 1}\big) \stackrel{(d)}{=} \big((\rZ_t\cdot \Xi_t(.),t \in [0,\rT)),(\upxi_{0}(i))_{i\geq 1}\big)
\end{equation}
Therefore $(\rme^{i})_{i\geq 1}$ is a sequence of i.i.d. uniform$[0,1]$ r.v. In addition the collection of asymptotic frequencies $(\rZ_t\cdot \Xi_t(\{\upxi_{0}(i)\}),t\in[0,\rT))_{i\geq 1}$ only depends on $\hat{\Pi}$, thus it is independent of the initial types $(\upxi_{0}(i))_{i\geq 1}$. The asserted result follows.\cqfd
\end{proof}

\section{Some properties of the genealogy}\label{SectionLimitTh}
Consider a branching mechanism $\Psi$, a $\Psi$-CSBP $\rZ$ started from $1$ assumed to be c\`adl\`ag and a $\Psi$ flow of partitions $(\hat{\Pi}_{s,t},0 \leq s \leq t < \rT)$ defined from the $\Psi$-CSBP $\rZ$. We present some properties of the $\Psi$ flow of partitions before stating a limit theorem. For the sake of simplicity, let $\hat{\Pi}_t := \hat{\Pi}_{0,t}$ for all $t \geq 0$ (recall that Theorem \ref{TheoremEquivalence} allows us to extend this process after time $\rT$).

\subsection{Dust and modification}
\begin{theorem*}\textup{\textbf{\ref{ThDust}}}
The following dichotomy holds:
\begin{itemize}
\item If $\sigma = 0$ and $\int_{0}^{\infty}(1\wedge h)\nu(dh) < \infty$, then almost surely for all $t \in (0,\rT)$, the partition $\hat{\Pi}_{t}$ has singleton blocks.
\item Otherwise, almost surely for all $t\in(0,\rT)$, the partition $\hat{\Pi}_t$ has no singleton blocks.
\end{itemize}
Furthermore when $\sigma = 0$, almost surely for all $t \in (0,\rT]$ the asymptotic frequency of the dust component of $\hat{\Pi}_t$ is equal to $\prod_{s \leq t}(1-\frac{\Delta \rZ_s}{\rZ_s})$ whereas when $\sigma > 0$, almost surely for all $t \in (0,\rT]$ there is no dust.
\end{theorem*}
\begin{remark}
The condition $\sigma = 0$ and $\int_{0}^{\infty}(1\wedge h)\nu(dh) < \infty$ is equivalent to saying that $\Psi$ is the Laplace exponent of a L\'evy process with finite variation paths.
\end{remark}
\begin{proof}
By the definition of a $\Psi$ flow of partitions, we know that, conditional on $t < \rT$, $\hat{\Pi}_t$ is distributed as a paint-box on the subordinator $\rS_{0,t}$ therefore it has no singleton blocks iff $d_t = 0$ (recall that $d_t$ is the drift term of the Laplace exponent $u_t(.)$). Since for all $t,s \geq 0$, $u_{t+s}(.) = u_t\circ u_s(.)$, classical results ensure that $d_{t+s} = d_t.d_s$. Therefore
\begin{equation*}
\exists t > 0, d_t > 0 \Leftrightarrow \forall t > 0, d_t > 0
\end{equation*}
Also, the equivalence $d_t > 0 \Leftrightarrow \sigma = 0$ and $\int_{0}^{\infty}(1\wedge h)\nu(dh) < \infty$ can be found in~\cite{Silverstein68}.\\
Suppose now that $\sigma = 0$. Classical results on exchangeable partitions (see~\cite{BertoinRandomFragmentation} for instance) ensure that the asymptotic frequency of the dust is almost surely equal to the probability that the first block is a singleton conditional on the mass partition. If $t < \rT$, then $\hat{\Pi}_t(1)$ is a singleton iff no elementary reproduction event has involved $1$. This occurs with probability $\prod_{s \leq t}(1-\frac{\Delta \rZ_s}{\rZ_s})$ conditionally on $\rZ$. If $t = \rT$, then either $\hat{\Pi}_\rT$ has finitely many blocks and in that case it cannot have dust, or it has infinitely many blocks. In the latter case $\hat{\Pi}_\rT(1)$ is a singleton iff for all $t < \rT$ $\hat{\Pi}_t(1)$ is a singleton. This occurs with probability $\prod_{s < \rT}(1-\frac{\Delta \rZ_s}{\rZ_s})$.\\
When $\sigma > 0$, the number of blocks in $\hat{\Pi}_t$ is finite almost surely since we are either in the \Extinction\ case or in the \InfPos\ case.\\
Finally, for all $t\in[0,\rT)$ we have proved that the asserted properties hold almost surely. Since the process of asymptotic frequencies of $\hat{\Pi}_t$ is c\`adl\`ag, we deduce that these properties hold almost surely for all $t\in[0,\rT)$.\cqfd
\end{proof}
Another interesting question about genealogical structures is the following: can we recover the population size from the genealogy ?
\begin{proposition}
The process $(\rZ_t,t\geq 0)$ is measurable in the filtration ${\cal F}^{\hat{\Pi}}_t := \sigma\{\hat{\Pi}_{r,s},0 \leq r \leq s \leq t\}, t\geq 0$.
\end{proposition}
\begin{proof}
We give a sketch of the proof. Suppose that $\sigma > 0$ then the infinitesimal jumps due to binary coagulation events allow one to recover the jump rates which is $\sigma^2/\rZ_t$ at any given time $t \geq 0$ thus the process $\rZ$ is entirely recovered from this only information. Now suppose that $\sigma = 0$. The rescaled jumps $(\frac{\Delta \rZ_t}{\rZ_t},t\geq 0)$ are measurable w.r.t. $({\cal F}^{\hat{\Pi}}_t,t\geq 0)$. Conjointly with the knowledge of the deterministic drift $\gamma$, we are able to recover the paths of the process.\cqfd
\end{proof}
Recall that the trajectories of a stochastic flow of partitions are not necessarily deterministic flows of partitions: the cocycle property does not necessarily hold simultaneously for all triplets $r < s < t$. But we have mentioned that this property is actually verified in the particular case of a flow of partitions defined from a c\`adl\`ag CSBP as presented in Subsection \ref{SubsectionStoFlow}. The goal of what follows is to prove that any $\Psi$ flow of partitions admits a modification whose trajectories are deterministic flows of partitions. The following two results are proved in Section \ref{Proofs}.
\begin{lemma}\label{LemmaFeller}
The process $(\rZ_t,\hat{\Pi}_t; t \geq 0)$ is a Markov process in its own filtration with a Feller semigroup.
\end{lemma}
\begin{proposition}\label{PropositionRegularization}
Consider a $\Psi$ flow of partitions $(\hat{\Pi}_{s,t},0 \leq s \leq t < \rT)$ with underlying $\Psi$-CSBP $(\rZ_t, 0 \leq t < \rT)$. There exists a process $(\tilde{\hat{\Pi}}_{s,t},0 \leq s \leq t < \rT)$ such that:\begin{itemize}
\item For all $s \leq t$, almost surely on the event $\{t < \rT\}$ $\hat{\Pi}_{s,t} = \tilde{\hat{\Pi}}_{s,t}$.
\item For $\mathbb{P}$-a.a. $\omega \in \Omega$, $\tilde{\hat{\Pi}}(\omega)$ is a deterministic flow of partitions without simultaneous mergers.
\end{itemize}
\end{proposition}

\subsection{A limit theorem}\label{SubsectionLimitTh}
We now turn our attention to the continuity properties of the law of $(\hat{\Pi}_t,t\geq 0)$ according to its branching mechanism $\Psi$. To motivate this study we provide a convergence result for sequences of $\Psi$-CSBPs, but this requires first to introduce a suitable topology to compare c\`adl\`ag functions that possibly reach $\infty$ in finite time. At first reading, one can replace our topology with the usual Skorohod's topology and skip the next paragraph.\\
Our topology is the same as the one introduced in~\cite{CaballeroLambertBravo09}. Let $\bar{d}$ be a metric on $[0,+\infty]$ that makes this space homeomorphic to $[0,1]$. We denote by $\mathscr{D}([0,+\infty],[0,+\infty])$ the space of c\`adl\`ag functions $f : [0,+\infty] \rightarrow [0,+\infty]$ such that $f(t)=0$ (resp. $\infty$) implies $f(t+s)=0$ (resp. $\infty$) and $\lim\limits_{t\rightarrow\infty}f(t)$ exists in $[0,+\infty]$ and is equal to $f(\infty)$.\\
We define $\Lambda_{\infty}$ as the set of increasing homeomorphisms of $[0,+\infty]$ into itself. Let $\bar{d}_{\infty}$ be the following metric on $\mathscr{D}([0,+\infty],[0,+\infty])$
\begin{equation*}
\bar{d}_{\infty}(f,g) := 1 \wedge \inf_{\lambda \in \Lambda_{\infty}}\left(\sup_{s\geq 0}\bar{d}(f(s),g\circ\lambda(s))\vee\sup_{s\geq 0}|s-\lambda(s)|\right)
\end{equation*}

Let $(\Psi_m)_{m \in \mathbb{N}}$ be a sequence of branching mechanisms such that Equation (\ref{EquationPsi}) is fulfilled with the triplet $(\gamma_{m},\sigma_{m},\nu_{m})_{m\in\mathbb{N}}$ verifying the corresponding assumptions and denote by $\rZ^{m}$ a $\Psi_{m}$-CSBP started from $1$. Let $\Psi$ be another branching mechanism and $\rZ$ a $\Psi$-CSBP. We consider the following assumption.\vspace{-5pt}
\begin{assumption}\label{AssumptionPsi}
For all $u \in \mathbb{R}_+$, we have $\Psi_{m}(u) \rightarrow \Psi(u)$ as $m \rightarrow \infty$.
\end{assumption}
\begin{remark}
This assumption is equivalent with
\begin{equation*}
\gamma_m-\nu_{m}((1,\infty)) \underset{m\rightarrow\infty}{\longrightarrow} \gamma-\nu((1,\infty))
\end{equation*}
in $\mathbb{R}$ and
\begin{equation*}
\sigma^2_m\delta_0(dh)+(1\wedge h^2)\,\nu_m(dh) \underset{m\rightarrow\infty}{\longrightarrow} \sigma^2\delta_0(dh)+(1\wedge h^2)\,\nu(dh)
\end{equation*}
in the sense of weak convergence in the set $\mathscr{M}_f(\mathbb{R}_+)$ of finite measures on $\mathbb{R}_+$. See Theorem VII.2.9 and Remark VII.2.10 in~\cite{JacodShiryaev}.
\end{remark}
The following proposition yields a convergence result on sequences of CSBP, which is a consequence of the work of Caballero, Lambert and Uribe Bravo in~\cite{CaballeroLambertBravo09}.
\begin{proposition}\label{PropositionConvergenceCSBP}
Under Assumption \ref{AssumptionPsi}, we have
\begin{equation}
(\rZ^{m}_t,t\geq 0) \stackrel{(d)}{\underset{m\rightarrow\infty}{\longrightarrow}} (\rZ_t,t\geq 0)\mbox{ , in the sense of weak convergence in }\mathscr{D}([0,+\infty],[0,+\infty])
\end{equation}
\end{proposition}
\begin{proof}
The proof of Proposition 6 in~\cite{CaballeroLambertBravo09} ensures that there exists a sequence $(\rY^\rmm_t, t \geq 0)$ of $\Psi_m$-L\'evy processes started from $1$ stopped whenever reaching $0$ that converges almost surely to a $\Psi$-L\'evy process $(\rY_t,t\geq 0)$ stopped whenever reaching $0$, where the convergence holds in $\mathscr{D}([0,+\infty],[0,+\infty])$. Furthermore, Proposition 5 in~\cite{CaballeroLambertBravo09} yields that $L^{-1}$ is continuous on $(\mathscr{D}([0,+\infty],[0,+\infty]),\bar{d}_{\infty})$ where $L$ is a time change due to Lamperti, see Subsection \ref{AppendixLamperti} for the definition. Therefore, we deduce that
$$ \bar{d}_{\infty}(L^{-1}(\rY^m),L^{-1}(\rY)) \underset{m\rightarrow\infty}{\longrightarrow} 0 \mbox{ a.s.}$$
Since $L^{-1}(\rY^m)$ (resp. $L^{-1}(\rY)$) is a $\Psi_m$ (resp. $\Psi$) CSBP for all $m\geq 1$, this concludes the proof.\cqfd
\end{proof}
Let $(\hat{\Pi}_{s,t}^m,0\leq s \leq t)$ be a $\Psi_m$ flow of partitions, for each $m\geq 1$.
\begin{theorem*}\textup{\textbf{\ref{TheoremLimit}}}
Suppose that
\begin{enumerate}[i)]
\item For all $u \in \mathbb{R}_+$, $\Psi_m(u) \rightarrow \Psi(u)$ as $m\rightarrow\infty$.
\item The branching mechanism $\Psi$ satisfies the Eve property.
\item $\Psi$ is not the Laplace exponent of a compound Poisson process.
\end{enumerate}
then
\begin{equation*}
(\hat{\Pi}^{m}_{t},t \geq 0) \stackrel{(d)}{\underset{m\rightarrow\infty}{\longrightarrow}}(\hat{\Pi}_{t},t \geq 0)
\end{equation*}
in $\mathbb{D}(\mathbb{R}_+,\mathscr{P}_{\infty})$.
\end{theorem*}
The proof of this theorem requires a preliminary lemma. Recall that $\rT_\epsilon:=\inf\{t\geq 0:\rZ_t\notin(\epsilon,1/\epsilon)\}$.
\begin{lemma}\label{LemmaCVCSBPkilled}
Under the hypothesis of the theorem, suppose that $\rZ^m \rightarrow \rZ$ almost surely as $m\uparrow\infty$ in $\mathscr{D}([0,+\infty],[0,+\infty])$, then for all $\epsilon \in (0,1)$ we have $\rT^m_\epsilon \underset{m\rightarrow\infty}{\longrightarrow} \rT_\epsilon$ a.s. and
\begin{equation*}
\big(\rZ^m_{t\wedge \rT^m_\epsilon},t \geq 0) \underset{m \rightarrow\infty}{\longrightarrow} \big(\rZ_{t\wedge \rT_\epsilon},t \geq 0\big)\mbox{ a.s.}
\end{equation*}
in $\mathbb{D}(\mathbb{R}_+,\mathbb{R}_+)$.
\end{lemma}
We postpone the proof of this lemma to Section \ref{Proofs}. We are now ready to prove the theorem.
\begin{proof}(Theorem \ref{TheoremLimit})
The definition of the topology on $\mathscr{P}_{\infty}$ entails that it suffices to show that for every $n\in\mathbb{N}$,
\begin{equation}\label{EqCVFlow}
\big(\hat{\Pi}^{m,[n]}_{t},t \geq 0\big) \stackrel{(d)}{\underset{m\rightarrow\infty}{\longrightarrow}}\big(\hat{\Pi}^{[n]}_{t},t \geq 0\big)
\end{equation}
in $\mathbb{D}(\mathbb{R}_+,\mathscr{P}_n)$. So we fix $n\in\mathbb{N}$. Our proof consists in showing that
\begin{enumerate}[a)]
\item\label{ItemCVFlow2} $\mathbb{P}\Big(\hat{\Pi}^{[n]}_{\rT_\epsilon}=\tun_{[n]}\Big) \rightarrow 1$ as $\epsilon\downarrow 0$.
\item\label{ItemCVFlow} $\Big(\hat{\Pi}^{m,[n]}_{t\wedge \rT^m_\epsilon},t \geq 0\Big) \stackrel{(d)}{\underset{m\rightarrow\infty}{\longrightarrow}}\Big(\hat{\Pi}^{[n]}_{t\wedge \rT_\epsilon},t \geq 0\Big)
$ in $\mathbb{D}(\mathbb{R}_+,\mathscr{P}_n)$ for every $\epsilon \in (0,1)$.
\end{enumerate}
To see that those two properties imply the asserted convergence, observe that they entail 
\begin{equation*}
\mathbb{P}\big(\hat{\Pi}^{m,[n]}_{\rT_\epsilon}=\tun_{[n]}\big) \rightarrow 1
\end{equation*}
as $m\rightarrow \infty$ and $\epsilon\downarrow 0$. Since $\tun_{[n]}$ is an absorbing state for the processes $\hat{\Pi}^{[n]}$ and $\hat{\Pi}^{m,[n]}$, we deduce that
\begin{eqnarray*}
\mathbb{P}\big[(\hat{\Pi}^{[n]}_{t\wedge \rT_\epsilon},t \geq 0) =(\hat{\Pi}^{[n]}_{t},t \geq 0)\big] &\rightarrow& 1\\
\mathbb{P}\Big[\big(\hat{\Pi}^{m,[n]}_{t\wedge \rT^m_\epsilon},t \geq 0\big) = \big(\hat{\Pi}^{m,[n]}_{t},t \geq 0\big)\Big] &\rightarrow& 1
\end{eqnarray*}
as $m\rightarrow \infty$ and $\epsilon\downarrow 0$. Hence, the asserted convergence follows. We now prove \ref{ItemCVFlow2}) and \ref{ItemCVFlow}).\vspace{3pt}\\
The first property \ref{ItemCVFlow2}) derives from Theorem \ref{TheoremEquivalence} and the fact that $\rT_\epsilon \rightarrow \rT$ a.s. as $\epsilon \downarrow 0$.\\
Let us prove the second property \ref{ItemCVFlow}). Note that this is sufficient to show that for every $\epsilon\in(0,1)$
\begin{equation*}
\Big\{\big(\rL_{t\wedge \rT^m_\epsilon}^m(n,K),t \geq 0\big);K\subset [n],\#K\geq 2\Big\} \underset{m\uparrow\infty}{\stackrel{(d)}{\longrightarrow}} \Big\{\big(\rL_{t\wedge \rT_\epsilon}(n,K),t \geq 0\big);K\subset [n],\#K\geq 2\Big\}
\end{equation*}
in $\mathbb{D}(\mathbb{R}_+,\mathbb{R}_+^{d_n})$, where $\{(\rL_t(n,K),t\geq 0); K \subset [n],\#K\geq 2\}$ are the counting processes whose jump rates have been characterized in Subsection \ref{SubsectionJumpRates}. Indeed the knowledge of these processes allows to determine the elementary reproduction events $\hat{\Pi}_{t-,t}^{[n]}$, and so, is sufficient to recover the process $(\hat{\Pi}_t^{[n]},t\geq 0)$. Obviously, this also holds when the processes are stopped at $\rT_\epsilon$.\\
Let $\rZ^m$ (resp. $\rZ$) be a $\Psi_m$-CSBP (resp. $\Psi$-CSBP) such that
$\rZ^m \rightarrow \rZ$ almost surely as $m\uparrow\infty$ in $\mathscr{D}([0,+\infty],[0,+\infty])$. Fix $\epsilon \in (0,1)$, we know from Lemma \ref{LemmaCVCSBPkilled} that
\begin{equation*}
\big(\rZ^m_{t\wedge \rT^m_\epsilon},t \geq 0\big) \underset{m \rightarrow\infty}{\longrightarrow} \big(\rZ_{t\wedge \rT_\epsilon},t \geq 0\big)\mbox{ a.s.}
\end{equation*}
in $\mathbb{D}(\mathbb{R}_+,\mathbb{R}_+)$. Using Proposition \ref{PropContinuityLambda} and the definition of the stopping times $\rT^m_\epsilon$, we deduce that for each $2 \leq k \leq n$
\begin{equation}\label{EqCVCompensators}
\Big(\int_{0}^{t\wedge \rT^{m}_\epsilon}\lambda_{n,k}(\rZ^m_{s-},\Psi^m)ds, t \geq 0\Big) \underset{m\uparrow\infty}{\longrightarrow} \Big(\int_{0}^{t\wedge \rT_\epsilon}\lambda_{n,k}(\rZ_{s-},\Psi)ds, t \geq 0\Big)\;\mbox{a.s.}
\end{equation}
in $\mathbb{D}(\mathbb{R}_+,\mathbb{R}_+)$.\\
From each $\Psi_m$-CSBP $\rZ^m$, we define a flow of partitions and consider the corresponding $d_n$-dimensional counting process
\begin{equation*}
\Big\{\big(\rL_{t\wedge \rT^m_\epsilon}^m(n,K),t \geq 0\big);K \subset [n],\#K\geq 2\Big\}
\end{equation*}
We do the same from the $\Psi$-CSBP $\rZ$ and define
\begin{equation*}
\Big\{\big(\rL_{t\wedge \rT_\epsilon}(n,K),t \geq 0\big);K \subset [n],\#K\geq 2\Big\}
\end{equation*}
We have shown in Proposition \ref{PropCompensator} that this process is a $d_n$-dimensional pure-jump semimartingale whose predictable compensator is the $d_n$-dimensional process
\begin{equation*}
\Big\{\big(\int_{0}^{t\wedge \rT_\epsilon}\lambda_{n,\#K}(\rZ_{s-},\Psi)ds, t \geq 0\big);K\subset [n],\#K\geq 2\Big\}
\end{equation*}
and similarly for $\{(\rL_{t\wedge \rT^m_\epsilon}^m(n,K),t \geq 0);K \subset [n],\#K\geq 2\}$, for each $m\geq 1$.
From Theorem VI.4.18 in~\cite{JacodShiryaev} we deduce that the collection of $d_n$-dimensional processes $\{(\rL_{t\wedge \rT^m_\epsilon}^m(n,K),t \geq 0);K\subset [n],\#K\geq 2\}_{m\geq 1}$ is tight $\mathbb{D}(\mathbb{R}_+,\mathbb{R}_+^{d_n})$. Indeed, in the notation of~\cite{JacodShiryaev} conditions (i) and (ii) are trivially verified, while condition (iii) is a consequence of Equation (\ref{EqCVCompensators}). Furthermore, this last equation ensures that any limit of a subsequence of the collection of $d_n$-dimensional semimartingales is a $d_n$-dimensional semimartingale whose predictable compensator is
\begin{equation*}
\Big\{\big(\int_{0}^{t\wedge \rT_\epsilon}\lambda_{n,\#K}(\rZ_{s-},\Psi)ds, t \geq 0\big);K\subset [n],\#K\geq 2\Big\}
\end{equation*}
which characterizes uniquely the semimartingale $\{(\rL_{t\wedge \rT_\epsilon}(n,K),t \geq 0);K\subset [n],\#K\geq 2\}$ (see for instance Theorem IX.2.4 in~\cite{JacodShiryaev}). This ensures the following convergence
\begin{equation}\label{EqCVCounting}
\Big\{\big(\rL_{t\wedge \rT^m_\epsilon}^m(n,K),t \geq 0\big);K\subset [n],\#K\geq 2\Big\} \underset{m\uparrow\infty}{\stackrel{(d)}{\longrightarrow}} \Big\{\big(\rL_{t\wedge \rT_\epsilon}(n,K),t \geq 0\big);K\subset [n],\#K\geq 2\Big\}
\end{equation}
in $\mathbb{D}(\mathbb{R}_+,\mathbb{R}_+^{d_n})$.\cqfd
\end{proof}

\section{From a flow of subordinators to the lookdown representation}\label{SectionPathwise}
\subsection{Connection with generalized Fleming-Viot and motivation}
In~\cite{Labbe11}, we considered the class of generalized Fleming-Viot processes: they are Markov processes that take values in the set of probability measures on a set of genetic types, say $[0,1]$, and that describe the evolution of the asymptotic frequencies of genetic types in a population of constant size $1$. A flow of generalized Fleming-Viot processes $(\uprho_{s,t},-\infty < s \leq t < +\infty)$ is a consistent collection of generalized Fleming-Viot processes and is completely encoded by a stochastic flow of bridges, see the article~\cite{BertoinLeGall-1} of Bertoin and Le Gall. This object is similar to the stochastic flow of subordinators (restricted to an initial population $[0,1]$), but while in the former the population size is constant in the latter it varies as a CSBP. This will be a major difficulty in the present work. We introduced the notion of ancestral type for a generalized Fleming-Viot process, similarly as we have identified the ancestors of a MVBP. Then we restricted our study to the following two subclasses of generalized Fleming-Viot:
\begin{itemize}
\item \emph{Eves - extinction}: the ancestral types with a positive frequency are finitely many at any positive time almost surely and any two of them do never get extinct simultaneously. Thus we can order them by decreasing extinction times, hence obtaining a sequence $(\rme^{i})_{i\geq 1}$ called the Eves.
\item \emph{Eves - persistent}: the ancestral types do never become extinct and can be ordered according to the asymptotic behaviours of their progenies as $t$ tends to $\infty$, we called the corresponding sequence $(\rme^i)_{i\geq 1}$ the Eves as well.
\end{itemize}
At each time $s \in \mathbb{R}$, we defined the sequence of Eves $(\rme^{i}_s)_{i\geq 1}$ as the ancestral types of the generalized Fleming-Viot process $(\uprho_{s,t},t\in [s,\infty))$. Then, expressing the genealogical relationships between those Eves in terms of partitions of integers we obtained a stochastic flow of partitions $(\hat{\Pi}_{s,t},-\infty < s \leq t < \infty)$. These two objects catch all the information encoded by the flow of generalized Fleming-Viot processes: $(\rme^{i}_s)_{i\geq 1}$ is the sequence of types carried by the population started at time $s$ while $(\hat{\Pi}_{s,t},t\in[s,\infty))$ tells how the frequencies of these types evolve in time. Additionally, it provides a pathwise connection with the lookdown representation: the main result of~\cite{Labbe11} asserts that for every time $s\in\mathbb{R}$, the process of limiting empirical measures $\mathscr{E}_s(\hat{\Pi},(\rme^{i}_{s})_{i\geq 1})$ is almost surely equal to $(\uprho_{s,t},t\in[s,\infty))$.\vspace{3pt}

Many connections exist between generalized Fleming-Viot processes and $\Psi$-MVBP: in~\cite{BertoinLeGall-0}, Bertoin and Le Gall proved that the Bolthausen-Sznitman coalescent is the genealogy of the Neveu branching process, in~\cite{Article7} Birkner et al. exhibited a striking connection between $\alpha$-stable branching processes and Beta$(2-\alpha,\alpha)$ Fleming-Viot processes and in~\cite{BertoinLeGall-3} Bertoin and Le Gall proved that a generalized Fleming-Viot process has a behaviour locally (\textit{i.e.} for a small subpopulation) identical with a branching process. It is thus natural to expect that a result similar to the one stated in~\cite{Labbe11} holds in the present setting of branching processes.\\
For our construction to hold, we need the following assumptions:
\begin{itemize}
\item $\Psi$ is conservative, that is, the $\Psi$-CSBP does not reach $\infty$ in finite time a.s.
\item The branching mechanism $\Psi$ enjoys the Eve property.
\end{itemize}
These assumptions ensure the existence of the ordering of ancestors presented in Subsection \ref{SubsectionAncestors} in three different cases: \Extinction, \InfNo\ and \InfPos.\vspace{3pt}

From now on, we consider a flow of $\Psi$-MVBP $(\rmm_{s,t},0 \leq s \leq t < \rT)$ defined from a $\Psi$ flow of subordinators. As explained in Subsection \ref{SubsectionMVBP}, we can assume that each process $(\rmm_{s,t},t\in[s,\rT))$ is c\`adl\`ag. Note that in this section, we use $\rT$ instead of $\rT^\rS$ for the lifetime of the flow and we set $\rZ_t := \rmm_{0,t}([0,1])$ for all $t\in [0,\rT)$ instead of the notation $\rS_t$. Finally, recall the definition of the probability measure $\rmr_{s,t}$ via the rescaling
\begin{equation*}
\rmr_{s,t}(dx) := \frac{\rmm_{s,t}(\rZ_s\cdot dx)}{\rZ_t}
\end{equation*}
In the next subsection, we define the Eves process $(\rme^{i}_s,s \in [0,\rT))_{i\geq 1}$ and a $\Psi$-flow of partitions $(\hat{\Pi}_{s,t},0 \leq s \leq t < \rT)$ pathwise from the flow of $\Psi$-MVBP. In particular, we prove Theorem \ref{TheoremFoP2}. In the last subsection, we introduce for all $s \in [0,\rT)$, the lookdown process $(\upxi_{s,t}(i),t \in [s,\rT))_{i\geq 1} := \mathscr{L}_s(\hat{\Pi},(\rme^i_s)_{i\geq 1})$ and define the measure-valued process $(\Xi_{s,t},t \in [s,\rT)) := \mathscr{E}_s(\hat{\Pi},(\rme^i_s)_{i\geq 1})$. The rest of that subsection is devoted to the proof of the following result.
\begin{theorem*}\textup{\textbf{\ref{ThLookdown}}}
The flow of subordinators can be uniquely decomposed into two random objects: the Eves process $(\rme^{i}_s,s\in [0,\rT))$ and the flow of partitions $(\hat{\Pi}_{s,t},0\leq s \leq t < \rT)$.
\begin{enumerate}[i)]
\item \textbf{Decomposition}. For each $s \in \mathbb{R}$, a.s. $\mathscr{E}_s(\hat{\Pi},(\rme^{i}_{s})_{i\geq 1}) = (\rmr_{s,t},t \in [s,\rT))$
\item \textbf{Uniqueness}. Let $(\rH_{s,t},0 \leq s \leq t < \rT)$ be a $\Psi$ flow of partitions defined from the $\Psi$-CSBP $\rZ$, and for each $s \in [0,\rT)$, consider a sequence $(\upchi_s(i))_{i\geq 1}$ of r.v. taking distinct values in $[0,1]$. If for each $s \in [0,\rT)$, a.s. $\mathscr{E}_{s}(\rH,(\upchi_s(i))_{i\geq 1}) = (\rmr_{s,t},t \in [s,\rT))$ then\begin{itemize}
\item For each $s \in [0,\rT)$, a.s. $(\upchi_s(i))_{i\geq 1} = (\rme^{i}_s)_{i\geq 1}$.
\item Almost surely $\rH = \hat{\Pi}$.
\end{itemize}
\end{enumerate}
\end{theorem*}

\subsection{Eves process and flow of partitions}
For each $s\in[0,\rT)$, the process $(\rmm_{s,t},t\in[s,\rT))$ is a $\Psi$-MVBP started from the Lebesgue measure on $[0,\rZ_s]$. Therefore, we introduce the sequence $(\rme^{i}_s)_{i\geq 1}$ defined as its sequence of Eves (according to the definition given in Subsection \ref{SubsectionAncestors}) but rescaled by the mass $\rZ_s$ in order to obtain r.v. in $[0,1]$. The process $(\rme^{i}_s,s\in[0,\rT))_{i\geq 1}$ is then called the \textit{Eves process}.\\
A motivation for the rescaling of the Eves by the mass $\rZ_s$ is given by the following lemma.
\begin{lemma}\label{LemmaIndepEves}
For every $s \in [0,\rT)$, the sequence $(\rme^{i}_s)_{i\geq 1}$ is i.i.d. uniform$[0,1]$ and independent of the past of the flow until time $s$, that is, of $\sigma\{\rmm_{u,v},0 \leq u \leq v \leq s\}$.
\end{lemma}
\begin{proof}
An easy adaptation of Proposition \ref{PropositionLawAncestors} shows that, conditional on $\sigma\{\rmm_{u,v},0 \leq u \leq v \leq s\}$, the sequence $(\rZ_s\rme^{i}_s)_{i\geq 1}$ is i.i.d. uniform$[0,\rZ_s]$. Therefore, the sequence $(\rme^{i}_s)_{i\geq 1}$ is i.i.d. uniform$[0,1]$ and independent of $\sigma\{\rmm_{u,v},0 \leq u \leq v \leq s\}$.\cqfd
\end{proof}
We now express the genealogical relationships between the Eves in terms of partitions. It is convenient to define the process $\rF_{s,t}$ as the distribution function of $\rmr_{s,t}$ for all $0 \leq s \leq t < \rT$. One easily shows that this process is a bridge in the sense of~\cite{BertoinLeGall-1}: it is a non-decreasing random process from $0$ to $1$ with exchangeable increments. We define an exchangeable random partition $\hat{\Pi}_{s,t}$ for all $0 \leq s \leq t < \rT$ thanks to the following equivalence relation
\begin{equation*}
i \stackrel{\hat{\Pi}_{s,t}}{\sim} j \Leftrightarrow \rF^{-1}_{s,t}(\rme^{i}_t)=\rF^{-1}_{s,t}(\rme^{j}_t)
\end{equation*}
for all integers $i,j$.
\begin{proposition}\label{PropCompositionRule}
For all $0 \leq s \leq t < \rT$, almost surely $(\rF_{s,t},(\rme^{i}_t)_{i\geq 1},(\rme^{i}_s)_{i\geq 1})$ follows the composition rule, that is:
\begin{itemize}
\item $(\rme^{i}_t)_{i\geq 1}$ is i.i.d. uniform$[0,1]$.
\item $\hat{\Pi}_{s,t}$ is an exchangeable random partition independent of $(\rme^{i}_s)_{i\geq 1}$. Denote its blocks by $(A_j)_{j\geq 1}$ in the increasing order of their least elements. Then, for each $j$ and any $i\in A_j$, we have $\rme^{j}_s = \rF_{s,t}^{-1}(\rme^{i}_t)$.
\item $(\rme^{i}_s)_{i\geq 1}$ is i.i.d. uniform$[0,1]$.
\end{itemize}
\end{proposition}
The proof of this proposition follows from very similar arguments to those developed in Section 5 of~\cite{Labbe11}, where the \emph{Eves - extinction case} corresponds here to the \Extinction\ case while the \emph{Eves - persistent case} corresponds to \InfNo\ here. Once again, the \InfPos\ case is obtained as a mixture of the previous two ones.
\begin{theorem*}\textup{\textbf{\ref{TheoremFoP2}}}
The collection of partitions $(\hat{\Pi}_{s,t},0 \leq s \leq t < \rT)$ defined from the flow of subordinators and the Eves process is a $\Psi$ flow of partitions.
\end{theorem*}
\begin{proof}
Fix $0 \leq r < s < t < \rT$. We know that for all integers $i,j$
\begin{equation*}
i \stackrel{\hat{\Pi}_{r,t}}{\sim} j \Leftrightarrow \rF^{-1}_{r,t}(\rme^{i}_t) = \rF^{-1}_{r,t}(\rme^{j}_t)
\end{equation*}
Recall that $\rF^{-1}_{r,t}(\rme^{i}_t) = \rF^{-1}_{r,s}\circ \rF^{-1}_{s,t}(\rme^{i}_t)$ a.s. and similarly for $j$. Proposition \ref{PropCompositionRule} shows that there exists an integer $k_i$ (resp. $k_j$) such that $\rF^{-1}_{s,t}(\rme^{i}_t) = \rme^{k_i}_s$ a.s. (resp. $j$ instead of $i$). Then we obtain that a.s.
\begin{equation*}
i \stackrel{\hat{\Pi}_{r,t}}{\sim} j \Leftrightarrow k_i \stackrel{\hat{\Pi}_{r,s}}{\sim} k_j
\end{equation*}
From the definition of the coagulation operator, we deduce that a.s. $\hat{\Pi}_{r,t} = \Coag(\hat{\Pi}_{s,t},\hat{\Pi}_{r,s})$.\\
Now we prove the property on the finite dimensional marginals via a recursion on $n$. Implicitly $f_i$ (resp. $g_i$) will denote a bounded Borel map from $\mathscr{P}_{\infty}$ (resp. $\mathbb{R}_+$) to $\mathbb{R}$ while $\phi$ will be a bounded Borel map from $[0,1]^{\mathbb{N}}$ to $\mathbb{R}$. For any sequence $0=t_0 < t_1 < t_2 < \ldots < t_n$, $\rH_{t_{i-1},t_i} := \mathscr{P}(\rS_{t_{i-1},t_i})$ will denote the random partitions obtained via independent paint-box schemes based on $\rS_{t_{i-1},t_i}$, with $i\in[n]$. In addition, we will consider a more general setting in which the flow of subordinators is taken at time $0$ with an initial population $[0,z]$ for a given $z > 0$ (whereas in this section we consider only the case $z=1$). Then we make use of $\mathbb{P}_z$ to emphasize the dependence on $z > 0$. We will prove that for any integer $n\geq 1$, for all $z > 0$, $0 < t_1 \ldots < t_n$, and all $f_1,\ldots,f_n,g_1,\ldots,g_n$ we have
\begin{eqnarray*}
&&\mathbb{E}_z\big[f_1(\hat{\Pi}_{0,t_1})g_1(\rZ_{t_1})\ldots f_n(\hat{\Pi}_{t_{n-1},t_{n}})g_n(\rZ_{t_{n}})\,\big|\,(\rme^{i}_0)_{i\geq 1}\big]\\
&=&\mathbb{E}_z\big[f_1(\rH_{0,t_1})g_1(\rZ_{t_1})\ldots f_n(\rH_{t_{n-1},t_n})g_n(\rZ_{t_n})\big]
\end{eqnarray*}
This identity will ensure the asserted distribution for finite dimensional marginals of $\hat{\Pi}$.\\
At rank $n=1$, we use Lemma \ref{LemmaIndepEves} to deduce that the sequence $(\rme^i_{t_1})_{i\geq 1}$ is independent from the subordinator $\rS_{0,t_1}$. Therefore, we can assume that $\rH_{0,t_1}$ is defined according to the paint-box scheme with this sequence of i.i.d. uniform$[0,1]$ and the subordinator $\rS_{0,t_1}$, that is, $\rH_{0,t_1} = \hat{\Pi}_{0,t_1}$. It suffices to prove that $\hat{\Pi}_{0,t_1}$ and $\rZ_{t_1}$ are independent from the sequence $(\rme^{i}_0)_{i\geq 1}$. The first independence comes from Proposition \ref{PropCompositionRule}. The second independence can be obtained from Lemma \ref{LemmaIndepEves}. The identity follows.\\
Now suppose that the identity holds at rank $n-1$ for all $z >0$, and all $f_1,g_1,\ldots,f_{n-1},g_{n-1}$. At rank $n$, we get for any $f_1,g_1,\ldots,f_n,g_n,\phi$ and $z > 0$
\begin{eqnarray*}
&&\!\!\mathbb{E}_z\Big[f_1(\hat{\Pi}_{0,t_1})g_1(\rZ_{t_1})\ldots f_n(\hat{\Pi}_{t_{n-1},t_{n}})g_n(\rZ_{t_n})\phi((\rme^{i}_0)_{i\geq 1})\Big]\\
&=&\!\!\mathbb{E}_z\Big[f_1(\hat{\Pi}_{0,t_1})g_1(\rZ_{t_1})\phi((\rme^{i}_0)_{i\geq 1})\,\mathbb{E}\big[f_2(\hat{\Pi}_{t_1,t_2})g_2(\rZ_{t_2})\ldots f_n(\hat{\Pi}_{t_{n-1},t_n})g_n(\rZ_{t_n})\,\big|\,{\cal F}_{t_1},(\rme_{t_1}^{i})_{i\geq 1}\big]\Big]
\end{eqnarray*}
where ${\cal F}_{t_1}$ is the $\sigma$-field generated by the flow of subordinators until time $t_1$. Remark that we have used the measurability of $(\rme^i_0)_{i\geq 1}$ from ${\cal F}_{t_1}$ and $(\rme^i_{t_1})_{i\geq 1}$, given by Proposition \ref{PropCompositionRule}. We now apply the Markov property to the process $(\rZ_t,t\geq 0)$ to obtain
\begin{equation*}
=\mathbb{E}_z\Big[f_1(\hat{\Pi}_{0,t_1})g_1(\rZ_{t_1})\phi((\rme^{i}_0)_{i\geq 1})\,\mathbb{E}_{\rZ_{t_1}}\big[f_2(\hat{\Pi}_{0,t_2-t_1})g_2(\rZ_{t_2-t_1})\ldots g_n(\rZ_{t_n-t_1})\,\big|\,(\rme_{t_1}^{i})_{i\geq 1}\big]\Big]
\end{equation*}
Notice that we use an abusive notation when conditioning on $(\rme_{t_1}^{i})_{i\geq 1}$: we mean that the sequence of ancestors at time $0$ in the shifted (by $t_1$) process is equal to the sequence $(\rme_{t_1}^{i})_{i\geq 1}$ of the original flow of subordinators. We believe that an accurate notation would have greatly burdened the preceding equations. We now apply the recursion hypothesis and the case $n=1$ to obtain
\begin{eqnarray*}
&=&\mathbb{E}_z\Big[f_1(\rH_{0,t_1})g_1(\rZ_{t_1})\mathbb{E}_{\rZ_{t_1}}[f_2(\rH_{0,t_2-t_1})g_2(\rZ_{t_2-t_1})\ldots g_n(\rZ_{t_n-t_1})]\Big]\,\mathbb{E}_z\big[\phi((\rme^{i}_0)_{i\geq 1})\big]\\
&=&\mathbb{E}_z[f_1(\rH_{0,t_1})g_1(\rZ_{t_1})f_2(\rH_{t_1,t_2})\ldots f_n(\rH_{t_{n-1},t_n})g_n(\rZ_{t_n})]\,\mathbb{E}_z[\phi((\rme^{i}_0)_{i\geq 1})]
\end{eqnarray*}
where the last equality is due to the Markov property applied to the chain $\big(\rZ_{t_i},\rH_{t_{i-1},t_i},(t_{i+1}-t_i)\big)_{1\leq i\leq n}$. Note that this discrete chain is homogeneous in time since we include in the state-space the length of the \textit{next} time interval. The recursion is complete.\cqfd
\end{proof}

\subsection{The pathwise lookdown representation}
So far, we have defined pathwise from the flow of $\Psi$-MVBP $(\rmm_{s,t},0\leq s \leq t < \rT)$ the Eves process $(\rme^{i}_s,s \in [0,\rT))_{i\geq 1}$ and a $\Psi$ flow of partitions $(\hat{\Pi}_{s,t},0 \leq s \leq t < \rT)$. Thanks to Proposition \ref{PropositionRegularization}, we can consider a regularized modification of the flow of partitions that we still denote $(\hat{\Pi}_{s,t},0 \leq s \leq t < \rT)$ for convenience. We are now able to define a particle system $(\upxi_{s,t}(i),0 \leq s \leq t < \rT)_{i\geq 1}$ as follows. For all $s \in [0,\rT)$, let $(\upxi_{s,t}(i),t \in [s,\rT))_{i\geq 1} := \mathscr{L}_s(\hat{\Pi},(\rme^i_s)_{i\geq 1})$ and define the measure-valued process $(\Xi_{s,t},t \in [s,\rT)) := \mathscr{E}_s(\hat{\Pi},(\rme^i_s)_{i\geq 1})$. The rest of this subsection is devoted to the proof of Theorem \ref{ThLookdown}.
\begin{proof}(Theorem \ref{ThLookdown})
Fix $s\geq 0$ and work conditionally on $\{s < \rT\}$. From the lookdown representation, we know that for all $t \in [s,\rT)$, almost surely
\begin{equation*}
\Xi_{s,t}(dx) = \sum_{i\geq 1}|\hat{\Pi}_{s,t}(i)|\delta_{\rme^{i}_{s}}(dx) + (1-\sum_{i\geq 1}|\hat{\Pi}_{s,t}(i)|)dx
\end{equation*}
Thanks to Proposition \ref{PropCompositionRule}, for all $t \in [s,\rT)$, almost surely for all $i\geq 1$, $|\hat{\Pi}_{s,t}(i)|=\rmr_{s,t}(\{\rme^{i}_{s}\})$ and
\begin{equation*}
\rmr_{s,t} = \sum_{i\geq 1}\rmr_{s,t}(\{\rme^{i}_{s}\})\delta_{\rme^{i}_{s}}(dx) + (1-\sum_{i\geq 1}\rmr_{s,t}(\{\rme^{i}_{s}\}))dx
\end{equation*}
we obtain that almost surely for all $t\in[s,\rT)\cap\mathbb{Q}$, $\Xi_{s,t} = \rmr_{s,t}$. Since both processes are c\`adl\`ag, we deduce they are equal almost surely.\vspace{6pt}

We now turn our attention to the proof of the uniqueness property. Let $(\rH_{s,t},0 \leq s \leq t < \rT)$ be a $\Psi$-flow of partitions  defined from $\rZ$ and $(\upchi_s(i),s\in[0,\rT))_{i\geq 1}$ be, at each time $s \in [0,\rT)$, a sequence of r.v. taking distinct values in $[0,1]$. Define for each $s\in[0,\rT)$
\begin{equation*}
(\rX_{s,t},t\in [s,\rT)) := \mathscr{E}_{s}(\rH,(\upchi_s(i))_{i\geq 1})
\end{equation*}
and suppose that a.s. $(\rX_{s,t},t\in [s,\rT)) = (\rmr_{s,t},t \in [s,\rT))$.\vspace{3pt}

From Proposition \ref{PropUnicityInitialTypes}, we deduce that for each $s \in [0,\rT)$, almost surely $(\upchi_s(i))_{i\geq 1} = (\rme^{i}_s)_{i\geq 1}$. So the first uniqueness property is proved. We now prove the second uniqueness property. There exists an event $\Omega^*$ of probability $1$ such that on this event, for every rational numbers $s,t$ such that $0 \leq s \leq t < \rT$ and every integer $i\geq 1$ we have
\begin{equation}\label{EquationEqualityFreq}
\rmr_{s,t}(\{\rme^{i}_s\}) = |\hat{\Pi}_{s,t}(i)| = |\rH_{s,t}(i)|
\end{equation}
In the rest of the proof, we work on the event $\Omega^*$. Our proof relies on the following claim.\\\\
\textbf{Claim} The flow of partitions $\hat{\Pi}$ is entirely defined by the knowledge of the quantities $|\hat{\Pi}_{s,t}(i)|$ for every rational values $0 \leq s \leq t < \rT$ and every integer $i\geq 1$.\\\\
Obviously, the same then holds for the flow of partitions $\rH$. Thanks to this result and Equation (\ref{EquationEqualityFreq}), we deduce that $\hat{\Pi} = \rH$ almost surely.\cqfd
\end{proof}
It remains to prove the \textbf{Claim}. This is achieved thanks to the following two lemmas.
\begin{lemma}\label{LemmaUniqueness1}
Almost surely, for every $s\leq t$ in $[0,\rT)$ such that $t$ is rational, $\hat{\Pi}_{s,t}$ admits asymptotic frequencies and the process $r\mapsto|\hat{\Pi}_{t-r,t}(i)|$ is l\`adc\`ag for every integer $i$.
\end{lemma}
\begin{lemma}\label{LemmaUniqueness2}
Let $I$ be a subset of $\mathbb{N}$. The following assertions are equivalent
\begin{enumerate}[i)]
\item $\hat{\Pi}_{s-,s}$ has a unique non-singleton block $I$.
\item For every $i\geq 1$, let $b(i)$ be the smallest integer such that $i \!=\! b(i)\! -\!(\#\{I\cap[b(i)]\}\!-\!1)\vee 0$. Then for every $i\ne \min I$, $(|\hat{\Pi}_{s-,s+t}(i)|,t\geq 0) \!=\! (|\hat{\Pi}_{s,s+t}(b(i))|,t\geq 0)$.
\end{enumerate}
and similarly when $\hat{\Pi}$ is replaced by $\rH$.
\end{lemma}
\begin{proof}(\textbf{Claim})
The knowledge of $|\hat{\Pi}_{s,t}(i)|$ for every rational values $0 \leq s \leq t < \rT$ and every integer $i\geq 1$, entails, thanks to Lemma \ref{LemmaUniqueness1}, the knowledge of the quantities $|\hat{\Pi}_{r-,t}(i)|$ and $|\hat{\Pi}_{r,t}(i)|$ for all $r\in(0,t)$. Then, Lemma \ref{LemmaUniqueness2} ensures that the elementary reproduction events $\hat{\Pi}_{s-,s}$ are obtained from the preceding quantities.\cqfd
\end{proof}
\begin{proof}(\textbf{Lemma \ref{LemmaUniqueness1}})
From the exchangeability properties of $\hat{\Pi}$, we know that almost surely the quantities $|\hat{\Pi}_{s,t}(i)|$ exist simultaneously for all rational values $s\leq t$ and integers $i\geq 1$. Fix the rational value $t$. We differentiate three cases. First if $w_t$ is a finite measure and $d_t=0$, then the process $r\mapsto\hat{\Pi}_{t-r,r}$ has finitely many blocks and no dust, and evolves at discrete times by coagulation events. Thus, for all $s \in (0,\rT)$, there exist rational values $p < s < q$ such that $\hat{\Pi}_{s-,t}=\hat{\Pi}_{p,t}$ and $\hat{\Pi}_{s,t} = \hat{\Pi}_{q,t}$. The result follows. Second if $d_t > 0$. Then, one can easily prove that the rate at which the $i$-th block is involved in a coalescence event is finite, for every $i\geq 1$. Therefore, the same identities, but for the $i$-th block, as in the previous case hold.\\
Finally, consider the case where blocks have infinitely many blocks but no dust. Then, one can adapt the arguments used in Lemma 6.2 of~\cite{Labbe11} to obtain the result.\cqfd
\end{proof}
\begin{proof}(\textbf{Lemma \ref{LemmaUniqueness2}})
The objects are well-defined thanks to Lemma \ref{LemmaUniqueness1}. One can adapt the proof of Lemma 5.3 of~\cite{Labbe11} in this setting to obtain the asserted result. For this, it is enough to remark that the processes $(|\hat{\Pi}_{s,t}(i)|,t \in (s,\rT)\cap\mathbb{Q})_{i\geq 1}$ are distinct by pair since either they reach $0$ at distinct finite times or their asymptotic behaviours are distinct.\cqfd
\end{proof}

\section{Appendix}
\label{Proofs}
\subsection{The Lamperti representation}\label{AppendixLamperti}
The Lamperti representation provides a time change that maps a $\Psi$-L\'evy process to a $\Psi$-CSBP. It relies on the following objects. Define for all $t \geq 0$ and any $f \in \mathscr{D}([0,+\infty],[0,+\infty])$, 
\begin{equation*}
I(f)_t := \inf\{s \geq 0 : \int_{0}^{s}f(u)du > t\}
\end{equation*}
Then we define $L:\mathscr{D}([0,+\infty],[0,+\infty])\rightarrow \mathscr{D}([0,+\infty],[0,+\infty])$ by setting
\begin{equation*}
L(f) := f \circ I(f)\mbox{ for all }f \in \mathscr{D}([0,+\infty],[0,+\infty])
\end{equation*}
Conversely, one can verify that $L^{-1}(g) = g\circ J(g)$ where, for all $g \in \mathscr{D}([0,+\infty],[0,+\infty])$
\begin{equation*}
J(g)_t := \inf\{s \geq 0: \int_{0}^{s}g^{-1}(u)du > t\}
\end{equation*}
Consider a $\Psi$-L\'evy process $\rY$ started from $1$ stopped whenever reaching $0$, the result of Lamperti ensures that $L^{-1}(\rY)$ is a $\Psi$-CSBP started from $1$, and that $L(\rZ)$ is a $\Psi$-L\'evy started from $1$ stopped whenever reaching $0$.

\subsection{Proof of Lemma \ref{LemmaDistributionAbsContinuous}}
A simple calculation ensures that $t\mapsto u_t(\infty)$ (when $u_t(\infty) < \infty$) and $t\mapsto u_t(0+)$ are differentiable, with derivatives equal to $-\Psi(u_t(\infty))$ and $-\Psi(u_t(0+))$ respectively. Therefore $t \mapsto \mathbb{P}(T \leq t)=e^{-u_t(\infty)}+1-e^{-u_t(0+)}$ is differentiable as well. So the distribution of $\rT$ is absolutely continuous with respect to the Lebesgue measure on $(0,\infty)$ on the event $\{\rT < \infty\}$.\cqfd

\subsection{Proof of Lemma \ref{LemmaIntegrabilityCSBP}}
Let $(\rY_t, t \geq 0)$ be a $\Psi$-L\'evy started from $1$. Using the Lamperti's result, we define a $\Psi$-CSBP $\rZ := L^{-1}(\rY)$ started from $1$. For all $\epsilon \in (0,1)$, we introduce the stopping time
\begin{equation*}
\rT^{\rY}_\epsilon := \inf\{t \geq 0: \rY_t \notin (\epsilon,1/\epsilon)\}
\end{equation*}
and use the stopping time $\rT_\epsilon$ introduced for $\rZ$ previously. It is immediate to check that for all $t \geq 0$, $\rY_{J(\rY)_t\wedge \rT^{\rY}_\epsilon} = \rZ_{t\wedge \rT_\epsilon}$. Fix $t \geq 0$ and notice that
\begin{equation*}
J(\rY)_t \wedge \rT^{\rY}_\epsilon \leq \frac{t}{\epsilon}
\end{equation*}
Thus, we get that
\begin{eqnarray*}
\mathbb{E}\bigg[\sum_{s \leq t\wedge \rT_\epsilon: \Delta \rZ_s > 0}\Big(\frac{\Delta \rZ_s}{\rZ_s}\Big)^2\bigg] &=& \mathbb{E}\bigg[\sum_{\substack{s \leq J(\rY)_t\wedge \rT^{\rY}_\epsilon:\\0 < \Delta \rY_s < 1}} \Big(\frac{\Delta \rY_s}{\rY_s}\Big)^2\bigg] + \mathbb{E}\bigg[\sum_{\substack{s \leq J(\rY)_t\wedge \rT^{\rY}_\epsilon:\\\Delta \rY_s \geq 1}} \Big(\frac{\Delta \rY_s}{\rY_s}\Big)^2\bigg]\\
&\leq& \frac{1}{\epsilon^2}\mathbb{E}\bigg[\sum_{\substack{s \leq J(\rY)_t\wedge \rT^{\rY}_\epsilon:\\0 < \Delta \rY_s < 1}} \big(\Delta \rY_s\big)^2\bigg] + \mathbb{E}\Big[\#\{s \leq J(\rY)_t\wedge \rT^{\rY}_\epsilon: \Delta \rY_s \geq 1\}\Big]\\
&\leq& \frac{1}{\epsilon^2}\mathbb{E}\bigg[\sum_{s \leq t/\epsilon: 0 < \Delta \rY_s < 1} \big(\Delta \rY_s\big)^2\bigg] + \mathbb{E}\Big[\#\{s \leq t/\epsilon: \Delta \rY_s \geq 1\}\Big]\\
&\leq& \frac{t}{\epsilon^3}\int_{(0,1)}h^2\nu(dh) + \frac{t}{\epsilon}\nu([1,\infty)) < \infty
\end{eqnarray*}
where the last inequality derives from the very definition of $\nu$.\cqfd

\subsection{Proof of Proposition \ref{PropContinuityLambda}}
Fix $2 \leq k \leq n$. Consider a sequence $(z_m)_{m\geq 1} \in \mathbb{R}_+^*$ that converges to $z > 0$, and a sequence of branching mechanisms $(\Psi_m)_{m\geq 1}$ such that Assumption \ref{AssumptionPsi} is verified. We have to prove that
\begin{equation*}
\int_{0}^{1}x^k(1-x)^{n-k}\Big(\frac{\sigma_m^2}{z_m}x^{-2}\delta_{0}(dx)+z_m\nu_m\circ\phi_{z_m}^{-1}(dx)\Big) \rightarrow \int_{0}^{1}x^k(1-x)^{n-k}\Big(\frac{\sigma^2}{z}x^{-2}\delta_{0}(dx)+z\nu\circ\phi_z^{-1}(dx)\Big)
\end{equation*}
Since $x \mapsto x^{k-2}(1-x)^{n-k}$ is continuous on $[0,1]$, it suffices to prove that
\begin{equation}\label{EqWeakConv}
\frac{\sigma_m^2}{z_m}\,\delta_{0}(dx)+z_m x^2\nu_m\circ\phi_{z_m}^{-1}(dx) \stackrel{(w)}{\longrightarrow} \frac{\sigma^2}{z}\,\delta_{0}(dx)+z x^2\nu\circ\phi_z^{-1}(dx)
\end{equation}
in the sense of weak convergence in $\mathscr{M}_f([0,1])$. Let $f:[0,1]\rightarrow\mathbb{R}$ be a continuous function. Set
\begin{equation*}
I_m :=\!\! \int_{0}^{1}f(x)\Big(\frac{\sigma_m^2}{z_m}\delta_{0}(dx)+z_m x^2\nu_m\circ\phi_{z_m}^{-1}(dx)\Big) \!=\! \int_{0}^{\infty}\!\!\!f(\frac{h}{h+z_m})\Big(\frac{\sigma_m^2}{z_m}\delta_{0}(dh)+z_m \big(\frac{h}{h+z_m}\big)^2\nu_m(dh)\Big)
\end{equation*}
We decompose $I_m = A_m + B_m$ where
\begin{eqnarray*}
A_m \!\!&:=&\!\!\! \int_{0}^{\infty}\!\!\!f(\frac{h}{h+z})(\frac{h}{h+z})^2\frac{z}{1\wedge h^2}(\sigma_m^2\delta_{0}(dh)+(1\wedge h^2)\nu_m(dh))\\
B_m \!\!&:=&\!\!\! \int_{0}^{\infty}\!\!\!\Big(f(\frac{h}{h+z_m})(\frac{h}{h+z_m})^2 z_m-f(\frac{h}{h+z})(\frac{h}{h+z})^2 z\Big)\frac{1}{1\wedge h^2}\big(\sigma_m^2\delta_{0}(dh)+(1\wedge h^2)\nu_m(dh)\big)
\end{eqnarray*}
Using the remark below Assumption \ref{AssumptionPsi}, one can check that $A_m \rightarrow A$ where
\begin{eqnarray*}
A &:=& \int_{0}^{\infty}\!\!f\Big(\frac{h}{h+z}\Big)\Big(\frac{h}{h+z}\Big)^2\frac{z}{1\wedge h^2}\big(\sigma^2\delta_{0}(dh)+(1\wedge h^2)\nu(dh)\big)\\
&=&\int_{0}^{1}\!\!f(x)\big(\frac{\sigma^2}{z}\delta_{0}(dx)+z x^2\nu\circ\phi_{z}^{-1}(dx)\big)
\end{eqnarray*}
Therefore, the proof of Equation (\ref{EqWeakConv}) reduces to show that $B_m \rightarrow 0$ as $m\uparrow\infty$. To that end, we get
\begin{eqnarray*}
B_m &=&\!\! \int_{0}^{\infty}\!\!f(\frac{h}{h+z_m})\frac{h}{h+z_m}\Big(\frac{h}{h+z_m}\frac{z_m}{1\wedge h^2}-\frac{h}{h+z}\frac{z}{1\wedge h^2}\Big)\big[\sigma_m^2\delta_{0}(dh)+(1\wedge h^2)\nu_m(dh)\big]\\
&+&\!\! \int_{0}^{\infty}\!\!\Big(f(\frac{h}{h+z_m})\frac{h}{h+z_m}-f(\frac{h}{h+z})\frac{h}{h+z}\Big)\frac{h}{h+z}\frac{z}{1\wedge h^2}\big[\sigma_m^2\delta_{0}(dh)+(1\wedge h^2)\nu_m(dh)\big]
\end{eqnarray*}
Set $g_m(h) := \displaystyle\frac{h}{h+z_m}\frac{z_m}{1\wedge h^2}-\frac{h}{h+z}\frac{z}{1\wedge h^2}$. One can easily prove that $g_m(h) \rightarrow 0$ as $m\uparrow\infty$ uniformly in $h \in \mathbb{R}_+$. Therefore we have the following upper bound for the absolute value of the first term on the r.h.s. of the preceding equation
\begin{equation*}
||f||_{\infty}\int_{0}^{\infty}|g_m(h)|[\sigma_m^2\delta_{0}(dh)+(1\wedge h^2)\nu_m(dh)] \rightarrow 0\mbox{ as }m \uparrow\infty
\end{equation*}
Finally remark that the second term in the r.h.s. of the preceding equation gives
\begin{eqnarray*}
&&\int_{0}^{\infty}\Big(f(\frac{h}{h+z_m})\frac{h}{h+z_m}-f(\frac{h}{h+z})\frac{h}{h+z}\Big)\frac{h}{h+z}\frac{z}{1\wedge h^2}\,\big[\sigma_m^2\delta_{0}(dh)+(1\wedge h^2)\nu_m(dh)\big]\\
&=& \int_{0}^{\infty}\Big(f(\frac{h}{h+z_m})-f(\frac{h}{h+z})\Big)(\frac{h}{h+z})^2\frac{z}{1\wedge h^2}\,\big[\sigma_m^2\delta_{0}(dh)+(1\wedge h^2)\nu_m(dh)\big]\\
&+& \int_{0}^{\infty}f(\frac{h}{h+z_m})\Big(\frac{1}{h+z_m}-\frac{1}{h+z}\Big)\frac{h^2}{h+z}\frac{z}{1\wedge h^2}\,\big[\sigma_m^2\delta_{0}(dh)+(1\wedge h^2)\nu_m(dh)\big]
\end{eqnarray*}
Denote by $C_m$ and $D_m$ respectively the first and second term on the r.h.s. of the preceding equation. One can easily prove that $h\mapsto\displaystyle\frac{h}{h+z_m}-\frac{h}{h+z}$ converges to $0$ as $m\uparrow\infty$ uniformly in $h \in \mathbb{R}_+$. And since $f$ is uniformly continuous on $[0,1]$, we deduce that $h\mapsto f(\frac{h}{h+z_m})-f(\frac{h}{h+z})$ converges to $0$ as $m\uparrow\infty$ uniformly on $\mathbb{R}_+$. Therefore it is immediate to check that $|C_m| \rightarrow 0$ as $m\uparrow\infty$. Moreover,
\begin{equation*}
|f(\frac{h}{h+z_m})(\frac{1}{h+z_m}-\frac{1}{h+z})\frac{h^2}{h+z}\frac{z}{1\wedge h^2}| \leq ||f||_{\infty}\frac{|z-z_m|}{z_m}(\frac{1}{z}\vee z)
\end{equation*}
Hence, $D_m \rightarrow 0$ as $m\uparrow\infty$. This ends the proof of the proposition.\cqfd

\subsection{Proof of Lemma \ref{LemmaFeller}}
The state space of this process is $(0,\infty)\times\mathscr{P}_{\infty}$ to which is added formally a cemetery point $\partial$ that gathers all the states of the form $(0,\pi)$ and $(\infty,\pi)$ where $\pi$ is any partition. The semigroup has been completely defined in Theorem \ref{TheoremFoP}, and it follows that the corresponding process $(\rZ_t,\hat{\Pi}_t;t\geq 0)$ is Markov. To prove that this semigroup verifies the Feller property, we have to show first that the map $(z,\pi) \mapsto \mathbb{E}[f(\rZ_t,\hat{\Pi}_t)]$ is continuous and vanishes at $\partial$, and second that $\mathbb{E}[f(\rZ_t,\hat{\Pi}_t)] \rightarrow f(z,\pi)$ as $t\downarrow 0$, for any given continuous map $f:(0,\infty)\times\mathscr{P}_{\infty}\rightarrow\mathbb{R}$ that vanishes at $\partial$, and any initial condition $(z,\pi) \in (0,\infty)\times\mathscr{P}_{\infty}$ for the process $(\rZ,\hat{\Pi})$.\\
To show the first assertion, we consider the map
\begin{equation*}
(z,\pi) \mapsto \mathbb{E}[f(\rS_{0,t}(z),\Coag(\mathscr{P}(\rS^{z}_{0,t}),\pi))]
\end{equation*}
where $(\rS_{0,t}(a),a \geq 0)$ is a subordinator with Laplace exponent $u_t(.)$ and $\rS^{z}_{0,t}$ is its restriction to $[0,z]$. Let $(z_m,\pi_m)_{m\geq 1}$ be a sequence converging to $(z,\pi) \in (0,\infty)\times\mathscr{P}_{\infty}$. For all $\epsilon > 0$, there exists $m_0\geq 1$ such that for all $m\geq m_0$, we have
\begin{eqnarray*}
\mathbb{P}\Big(\frac{|\rS_{0,t}(z)-\rS_{0,t}(z_m)|}{\rS_{0,t}(z)\vee \rS_{0,t}(z_m)}>\epsilon\, \big| \, \rS_{0,t}(z) \notin\{0,\infty\}\Big) < \epsilon\\
\mathbb{P}\big(\rS_{0,t}(z_m)\ne \rS_{0,t}(z)\,|\, \rS_{0,t}(z) \in\{0,\infty\}\big) < \epsilon 
\end{eqnarray*}
Then, there are two cases: on the event $\{\rS_{0,t}(z) \in \{0,\infty\}\}$, the process starting from $(z,\pi)$ is in the cemetery point at time $t$ and with conditional probability greater than $1-\epsilon$, this is also the case at time $t$ for the process starting from $(z_m,\pi_m)$, for every $m\geq m_0$. On the complementary event $\{\rS_{0,t}(z) \notin \{0,\infty\}\}$, fix $m\geq m_0$ and $n\geq 1$. Without loss of generality, we can suppose that $z_m \leq z$. Let $(U_i)_{i\geq 1}$ be a sequence of i.i.d. uniform$[0,\rS_{0,t}(z)]$ r.v. and introduce the partitions $\mathscr{P}(\rS^{z}_{0,t})$ by applying the paint-box scheme to the subordinator $\rS^{z}_{0,t}$ using the $(U_i)_{i\geq 1}$. Let $(V_i)_{i \geq 1}$ be an independent sequence of i.i.d. uniform$[0,\rS_{0,t}(z_m)]$ r.v. and define the sequence $W_i := U_i\mathbf{1}_{\{U_i\leq \rS_{0,t}(z_m)\}}+V_i\mathbf{1}_{\{U_i> \rS_{0,t}(z_m)\}}$. This sequence is also i.i.d. uniform$[0,S_{0,t}(z_m)]$, and we apply the paint-box scheme to the subordinator $S^{z_m}_{0,t}$ with that sequence $(W_i)_{i\geq 1}$, then obtaining a partition $\mathscr{P}(S^{z_m}_{0,t})$. We have (recall the definition of the metric $d_{\mathscr{P}}$ given in Equation (\ref{EqMetricPartition}))
\begin{eqnarray*}
&&\mathbb{P}\big(d_{\mathscr{P}}(\mathscr{P}(\rS^{z_m}_{0,t}),\mathscr{P}(\rS^{z}_{0,t})) \leq 2^{-n}\, |\, \rS_{0,t}(z) \notin\{0,\infty\}\big) \\
&\geq& \mathbb{P}\bigg[\,\underset{i\leq n}{\bigcap}\big\{U_i= W_i\big\}\bigcap\Big\{\frac{|\rS_{0,t}(z)-\rS_{0,t}(z_m)|}{\rS_{0,t}(z)\vee \rS_{0,t}(z_m)}<\epsilon\Big\}\, \big| \, \rS_{0,t}(z) \notin\{0,\infty\}\bigg]\\
&\geq& \mathbb{P}\bigg[\,\underset{i\leq n}{\bigcap}\big\{U_i\leq \rS_{0,t}(z_m)\big\}\, \Big|\, \frac{|\rS_{0,t}(z)-\rS_{0,t}(z_m)|}{\rS_{0,t}(z)\vee \rS_{0,t}(z_m)}<\epsilon\,;\,\rS_{0,t}(z) \notin\{0,\infty\}\bigg]\\
&&\hspace{-10pt}\times\,\mathbb{P}\bigg[\frac{|\rS_{0,t}(z)-\rS_{0,t}(z_m)|}{\rS_{0,t}(z)\vee \rS_{0,t}(z_m)}<\epsilon\, \big|\,\rS_{0,t}(z) \notin\{0,\infty\}\bigg]\\
&\geq& \mathbb{P}\Big(\frac{U_1}{\rS_{0,t}(z)}\leq 1-\epsilon\Big)^n(1-\epsilon)
\geq (1-\epsilon)^{n+1}
\end{eqnarray*}
Putting all these arguments together and using the facts that the coagulation operator is bicontinuous and that $f$ is continuous and vanishes near $\partial$, one deduces that
\begin{equation*}
\mathbb{E}\big[f(\rS_{0,t}(z),\Coag(\mathscr{P}(\rS^{z}_{0,t}),\pi))-f(\rS_{0,t}(z_m),\Coag(\mathscr{P}(\rS^{z_m}_{0,t}),\pi_m))\big] \underset{m\rightarrow\infty}{\longrightarrow} 0
\end{equation*}
and the continuity property follows. The fact that it vanishes at $\partial$ is elementary. Let us now prove that for all $(z,\pi) \in (0,\infty)\times\mathscr{P}_{\infty}$, we have
\begin{equation*}
\mathbb{E}\big[f(\rS_{0,t}(z),\Coag(\mathscr{P}(\rS^{z}_{0,t}),\pi))\big] \underset{t\downarrow0}{\rightarrow} f(z,\pi)
\end{equation*}
This convergence follows from the c\`adl\`ag property of $t\mapsto \rS_{0,t}(z)$ and the fact that $\mathscr{P}(\rS^{z}_{0,t})$ tends to $\tO_{[\infty]}$ in distribution as $t\downarrow 0$. The Feller property follows.\cqfd

\subsection{Proof of Proposition \ref{PropositionRegularization}}
Let $(\hat{\Pi}_{s,t},0 \leq s \leq t < \rT)$ be a $\Psi$-flow of partitions with underlying $\Psi$-CSBP $\rZ$. The idea of the proof is the following: we consider the rational marginals of the flow and show that for $\mathbb{P}$-a.a. $\omega$, $(\hat{\Pi}_{s,t}(\omega),0 \leq s \leq t < \rT; s,t \in \mathbb{Q})$ is a deterministic flow of partitions. Thus we extend this flow to the entire interval $[0,\rT)$ and show that its trajectories are still deterministic flows of partitions, almost surely.\\
There exists an event $\Omega_{\hat{\Pi}}$ of probability $1$ such that on this event, we have:
\begin{itemize}
\item	For every $r < s < t \in [0,\rT)\cap\mathbb{Q}$, $\hat{\Pi}_{r,t} = \Coag(\hat{\Pi}_{s,t},\hat{\Pi}_{r,s})$.
\item For every $s\in(0,\rT)$, $\forall n\geq 1,\exists\epsilon>0$ s.t. $\forall p,q\in(s-\epsilon,s)\cap\mathbb{Q}$, $\hat{\Pi}_{p,q}^{[n]}=\tO_{[n]}$.
\item For every $s\in[0,\rT)$, $\forall n\geq 1,\exists\epsilon>0$ s.t. $\forall p,q\in(s,s+\epsilon)\cap\mathbb{Q}$, $\hat{\Pi}_{p,q}^{[n]}=\tO_{[n]}$.
\item For every $s\in[0,\rT)\cap\mathbb{Q}$, $\hat{\Pi}_{s,s}=0_{[\infty]}$ and the process $(\hat{\Pi}_{s,t},t \in [s,\rT)\cap\mathbb{Q})$ is c\`adl\`ag.
\end{itemize}
The existence of this event follows from the following arguments. First, for each given triplet the coagulation property holds a.s. So it holds simultaneously for all rational triplets, a.s. Second, the probability that $\hat{\Pi}_{r,t}$ is close to $\tO_{[\infty]}$ increases to $1$ as $t-r \downarrow 0$. Together with the coagulation property this ensures the second and third properties. Finally, Lemma \ref{LemmaFeller} shows that the process $(\rZ_t,\hat{\Pi}_t; t\geq 0)$ is Markov with a Feller semigroup, so it admits a c\`adl\`ag modification. This ensures the last assertion.\\
We now define a process $(\tilde{\hat{\Pi}}_{s,t},0 \leq s \leq t < \rT)$ as follows. On $\Omega_{\hat{\Pi}}$, we set for all $0 \leq s \leq t < \rT$
\begin{equation*}
\tilde{\hat{\Pi}}_{s,t} := \begin{cases}\hat{\Pi}_{s,t}&\mbox{ if }s,t \in \mathbb{Q}\\
\lim\limits_{r\downarrow t,r\in\mathbb{Q}}\hat{\Pi}_{s,r}&\mbox{ if }s\in\mathbb{Q},t\notin\mathbb{Q}\\
\lim\limits_{r\downarrow s,r\in\mathbb{Q}}\hat{\Pi}_{r,t}&\mbox{ if }t\in\mathbb{Q},s\notin\mathbb{Q}\\
\tO_{[\infty]}&\mbox{ if }s=t\\
\Coag(\tilde{\hat{\Pi}}_{r,t},\tilde{\hat{\Pi}}_{s,r})&\mbox{ if }s,t \notin\mathbb{Q}\mbox{ with any given }r\in(s,t)\cap\mathbb{Q}
\end{cases}
\end{equation*}
On the complementary event $\Omega\backslash\Omega_{\hat{\Pi}}$, set any arbitrary values to the flow $\tilde{\hat{\Pi}}$. A long but easy enumeration of all possible cases proves that this defines a modification of $\hat{\Pi}$ and that almost surely, the trajectories are deterministic flows of partitions.\cqfd

\subsection{Proof of Lemma \ref{LemmaCVCSBPkilled}}
Let $(f_m)_{m\geq 1}$ and $f$ be elements of $\mathscr{D}([0,+\infty],[0,+\infty])$ without negative jumps and introduce for all $\epsilon \in (0,1)$
\begin{equation*}
T^{f}(\epsilon) := \inf\{t \geq 0: f(t)\notin (\epsilon,1/\epsilon)\}\mbox{ and }T^{f_m}(\epsilon) := \inf\{t \geq 0: f_m(t)\notin (\epsilon,1/\epsilon)\}\mbox{ for all }m\geq 1
\end{equation*}
We make the following assumptions\begin{enumerate}[i)]
\item\label{Item1} $f_m(0)=f(0)=1$ and $f_m \underset{m\rightarrow\infty}{\longrightarrow} f$ for the distance $\bar{d}_{\infty}$.
\item\label{Item2} $T^{f}(\epsilon) = \inf\{t \geq 0: f(t)\notin [\epsilon,1/\epsilon]\}$.
\item\label{Item3} $\Delta f(T^{f}(\epsilon)) > 0 \Rightarrow f(T^{f}(\epsilon)) > 1/\epsilon$.
\item\label{Item4} For all $r \in [0,T^{f}(\epsilon))$, $\inf\limits_{s\in[0,r]}f(s) > \epsilon$. Moreover when $\Delta f(T^{f}(\epsilon)) > 0$, it remains true with $r =T^{f}(\epsilon)$.
\end{enumerate}
We fix $\epsilon \in (0,1)$ until the end of the proof.\\
\textit{Step 1.} We stress that $T^{f_m}(\epsilon) \rightarrow T^{f}(\epsilon)$ as $m\rightarrow\infty$. Indeed, suppose that there exists $\delta > 0$ such that $T^{f_m}(\epsilon) < T^{f}(\epsilon) - \delta$ for an infinity of $m\geq 1$ (for simplicity, say for all $m\geq 1$). Then, for all $m\geq 1$ we use \ref{Item4}) to deduce
\begin{equation*}
\bar{d}_{\infty}(f,f_m) \geq \frac{\delta}{2}\wedge\inf\limits_{\{s\in[0,T^{f}(\epsilon)-\frac{\delta}{2}]\}}\{|1/\epsilon-f(s)|\wedge|f(s)-\epsilon|\} > 0
\end{equation*}
which contradicts the convergence hypothesis. Similarly, suppose that there exists $\delta > 0$ such that $T^{f_m}(\epsilon) > T^{f}(\epsilon) + \delta$ for an infinity of $m\geq 1$ (here again, say for all $m\geq 1$). For all $m\geq 1$, using \ref{Item2}) we have
\begin{equation*}
\bar{d}_{\infty}(f,f_m) \geq \frac{\delta}{2}\wedge\sup\limits_{\{s\in[T^{f}(\epsilon),T^{f}(\epsilon)+\frac{\delta}{2}]\}}\{|f(s)-1/\epsilon|\wedge|\epsilon-f(s)|\} > 0
\end{equation*}
which also contradicts the convergence hypothesis. Therefore, the asserted convergence $T^{f_m}(\epsilon) \rightarrow T^{f}(\epsilon)$ as $m\rightarrow\infty$ holds. For simplicity, we now write $T^{f}$ instead of $T^{f}(\epsilon)$ to alleviate the notation.\vspace{3pt}\\
\textit{Step 2.} We now prove that $(f^{m}(t\wedge T^{f_m}),t \geq 0) \rightarrow (f(t\wedge T^{f}),t \geq 0)$ for the distance $\bar{d}_{\infty}$. We consider two cases.\\
\textit{Step 2a.} Suppose that $T^{f}$ is a continuity point of $f$. Fix $\eta > 0$, there exists $\delta > 0$ such that for all $t\in [T^{f}-\delta,T^{f}+\delta]$, we have
\begin{equation*}
\bar{d}(f(T^{f}),f(t)) < \eta
\end{equation*}
From \ref{Item1}), we know there exists an integer $m_0 \geq 1$ and a sequence $(\lambda_m)_{m\geq m_0}$ of homeomorphisms of $[0,\infty)$ into $[0,\infty)$ such that for all $m\geq m_0$, $|T^{f_m}-T^{f}| < \delta / 4$ and
\begin{equation*}
\sup\limits_{s\geq 0}|\lambda_m(s)-s| < \delta / 4\mbox{ and }\sup\limits_{s\geq 0}\bar{d}(f_m(\lambda_m(s)),f(s)) < \eta
\end{equation*}
We consider now any integer $m\geq m_0$. For all $r\in [0,\delta/2]$, we have
\begin{eqnarray*}
\bar{d}(f_m(T^{f_m}-r),f(T^{f})) &\leq& \bar{d}(f^{m}(T^{f_m}-r),f(\lambda_m^{-1}(T^{f_m}-r))) + \bar{d}(f(\lambda_m^{-1}(T^{f_m}-r)),f(T^{f}))\\
&<& 2\eta
\end{eqnarray*}
using the preceding inequalities. In particular we have proven that $f^{m}(T^{f_m}) \rightarrow f(T^{f})$ as $m\rightarrow\infty$. Finally for all $t \geq 0$ and all $m\geq m_0$ we have
\begin{eqnarray*}
\bar{d}(f_m(\lambda_m(t)\wedge T^{f_m}),f(t\wedge T^{f})) &\leq& \bar{d}(f_m(\lambda_m(t)),f(t)) + \bar{d}(f_m(T^{f_m}),f(t))\mathbf{1}_{\{\lambda_m(t)>T^{f_m},t<T^f\}}\\
&+& \bar{d}(f_m(\lambda_m(t)),f(T^{f}))\mathbf{1}_{\{\lambda_m(t)\leq T^{f_m},t \geq T^f\}} + \bar{d}(f_m(T^{f_m}),f(T^{f}))
\end{eqnarray*}
The first and the fourth term in the r.h.s are inferior to $\eta$ and $2\eta$ thanks to the preceding inequalities. Concerning the second term, one can show that $|t-T^{f}| < \delta/2$ when $\{\lambda_m(t)>T^{f_m},t<T^f\}$ which ensures that
\begin{equation*}
\bar{d}(f_m(T^{f_m}),f(t))\mathbf{1}_{\{\lambda_m(t)>T^{f_m},t<T^f\}} \leq \bar{d}(f_m(T^{f_m}),f(T^f)) + \bar{d}(f(T^f),f(t))\mathbf{1}_{\{|t-T^f|<\delta/2\}} \leq 3\eta
\end{equation*}
Similarly, when $\{\lambda_m(t)\leq T^{f_m},t \geq T^f\}$ the quantity $r:=T^{f_m}-\lambda_m(t)$ belongs to $[0,\delta/2]$ and thus
\begin{eqnarray*}
\bar{d}(f_m(\lambda_m(t)),f(T^{f}))\mathbf{1}_{\{\lambda_m(t)\leq T^{f_m},t \geq T^f\}} < 2\eta
\end{eqnarray*}
Hence, we have $\bar{d}(f_m(\lambda_m(t)\wedge T^{f_m}),f(t\wedge T^{f})) \leq 8\eta$. This proves the asserted convergence when $T^{f}$ is a continuity point of $f$.\\
\textit{Step 2b.} Now suppose that $f$ jumps at time $T^{f}$. Recall that in that case, $f(T^{f})>1/\epsilon$. We denote by $S:=\sup\{f(s):s \in [0,T^{f})\}$ the supremum of $f$ before time $T^{f}$, which is strictly inferior to $1/\epsilon$, and similarly $I:=\inf\{f(s):s \in [0,T^{f})\}$ which is strictly superior to $\epsilon$ thanks to \ref{Item4}). Set
\begin{equation*}
\eta := [\bar{d}(1/\epsilon,S)\wedge\bar{d}(1/\epsilon,f(T^{f}))\wedge\bar{d}(\epsilon,I)]/2
\end{equation*}
Thanks to \ref{Item1}), there exists an integer $m_0 \geq 1$ and a sequence $(\lambda_m)_{m\geq m_0}$ of homeomorphisms of $[0,\infty)$ into itself such that for all $m\geq m_0$ we have
\begin{equation*}
\sup\limits_{s\geq 0}|\lambda_m(s)-s| < \eta\mbox{ and }\sup\limits_{s\geq 0}\bar{d}(f^{m}(\lambda_m(s)),f(s)) < \eta
\end{equation*}
Suppose that $\lambda_m(T^{f}) > T^{f_m}$. Then necessarily,
\begin{equation*}
\bar{d}(f^{m}(T^{f_m}),f(\lambda_m^{-1}(T^{f_m}))) > \bar{d}(1/\epsilon,S)\wedge\bar{d}(\epsilon,I) > \eta
\end{equation*}
which contradicts the hypothesis. Similarly, if $\lambda_m(T^{f}) < T^{f_m}$, then
\begin{equation*}
\bar{d}(f^{m}(\lambda_m(T^{f})),f(T^{f})) > \bar{d}(1/\epsilon,f(T^{f})) > \eta
\end{equation*}
which also contradicts the hypothesis. Hence, $\lambda_m(T^{f}) = T^{f_m}$. And we conclude that
\begin{equation*}
(f^{m}(t\wedge T^{f_m}),t \geq 0) \rightarrow (f(t\wedge T^{f}),t \geq 0)
\end{equation*}
in $(\mathscr{D}([0,+\infty],[0,+\infty]),\bar{d}_{\infty})$. But these functions are elements of $\mathbb{D}(\mathbb{R}_+,\mathbb{R}_+)$, so the last convergence also holds in the usual Skorohod's topology (see the remark below Proposition 5 in~\cite{CaballeroLambertBravo09}).\\

To finish the proof, it suffices to apply these deterministic results to the processes $\rZ^m$ and $\rZ$ once we have verified that their trajectories fulfil the required assumptions a.s. Recall that a $\Psi$-CSBP $\rZ$ can be obtained via the Lamperti time change (see Appendix \ref{AppendixLamperti}) of a $\Psi$ L\'evy process $\rY$
\begin{equation*}
\rZ_t = \rY_{J(\rY)_t}
\end{equation*}
and that the map $t\mapsto J(\rY)_t$ is continuous.\\
As we have assumed that $\Psi$ is not the Laplace exponent of a compound Poisson process, we deduce that a.s.
$\rT_\epsilon = \inf\{s\geq 0: \rZ_s \notin [\epsilon,1/\epsilon]\}$ and that if $\rZ$ jumps at time $\rT_\epsilon$ then $\rZ_{\rT_\epsilon} > 1/\epsilon$ a.s. Moreover, $\inf\limits_{s\in[0,r]} \rZ_s > \epsilon$ a.s. for all $r\in [0,\rT_\epsilon)$, and also for $r=\rT_\epsilon$ when $\rZ$ jumps at $\rT_\epsilon$. Otherwise, the c\`adl\`ag inverse of the infimum of $\rZ$ would admit a fixed discontinuity at time $\epsilon$ with positive probability, but the latter is (the Lamperti time-change of) the opposite of a subordinator, and so, it does not admit any fixed discontinuity.\cqfd
\paragraph{Acknowledgements.} This is part of my PhD thesis. I would like to thank my supervisors, Julien Berestycki and Amaury Lambert, for their useful comments on a draft version of this paper as well as an anonymous referee for his/her careful reading of the manuscript.


\end{document}